\documentclass{article}
\usepackage{amsmath, amsthm, amssymb, amsfonts}
\usepackage{geometry}
\usepackage{graphicx}
\usepackage[titletoc]{appendix}
\usepackage[pdfstartview = FitH]{hyperref}
\usepackage{enumitem}
\usepackage{mathtools}
\usepackage{natbib}
\mathtoolsset{showonlyrefs}

\hypersetup{
colorlinks=false,
citebordercolor={1 1 1} 
}

\newtheorem{theorem}{Theorem}[section]
\newtheorem{lemma}[theorem]{Lemma}
\newtheorem{proposition}[theorem]{Proposition}

\theoremstyle{definition}

\newtheorem{assumption}{Assumption}
\newtheorem{remark}[theorem]{Remark}

\newcommand{\norm}[1]{\Vert {#1} \Vert} 
\newcommand{\sumab}[2]{\sum \limits_{#1}^{#2}} 
\newcommand{\Nat}{\mathbb{N}}
\newcommand{\Rel}{\mathbb{R}}
\newcommand{\geschw}[1]{\left\{ #1 \right\}}

\newcommand{\secret}[1]{}

\DeclareMathOperator{\rank}{rank}
\DeclareMathOperator{\lspan}{span}

\DeclareMathOperator{\boundary}{bd}

\DeclareMathOperator{\closure}{cl}
\DeclareMathOperator{\adj}{adj}
\DeclareMathOperator{\grad}{grad}

\DeclareMathOperator{\diag}{diag}

\begin{document}
\title{Finite Sample Properties of Tests Based on Prewhitened Nonparametric Covariance Estimators}
\author{David Preinerstorfer\thanks{Department of Statistics and Operations Research, University of Vienna, Oskar-Morgenstern-Platz 1, 1090 Wien, Austria. E-mail: david.preinerstorfer@univie.ac.at. This research was supported by the Austrian Science Fund (FWF): P27398. I am grateful to Benedikt M. P\"otscher for many helpful discussions and for feedback on earlier versions of this manuscript.}}
\date{First version: August 2014 \\ This version: May 2015}


\maketitle

\begin{abstract}
We analytically investigate size and power properties of a popular family of procedures for testing linear restrictions on the coefficient vector in a linear regression model with temporally dependent errors. The tests considered are autocorrelation-corrected F-type tests based on prewhitened nonparametric covariance estimators that possibly incorporate a data-dependent bandwidth parameter, e.g., estimators as considered in \cite{A92}, \cite{NW94}, or \cite{Rho2013}. For design matrices that are generic in a measure theoretic sense we prove that these tests either suffer from extreme size distortions or from strong power deficiencies. Despite this negative result we demonstrate that a simple adjustment procedure based on artificial regressors can often resolve this problem. 
\end{abstract}


\medskip


\section{Introduction}\label{Intro}

The construction of tests for hypotheses on the coefficient vector in linear regression models with dependent errors is highly practically relevant and has received lots of attention in the statistics and econometrics literature. The main challenge is to obtain tests with good size and power properties in situations where the nuisance parameter governing the dependence structure of the errors is high- or possibly infinite-dimensional and allows for strong correlations. The large majority of available procedures are autocorrelation-corrected F-type tests, based on nonparametric covariance estimators trying to take into account the autocorrelation in the disturbances. These tests can roughly be categorized into two groups, the distinction depending on the choice of a bandwidth parameter in the construction of the covariance estimator. The first group of such tests, so-called `HAC tests', incorporates bandwidth parameters that lead to consistent covariance estimators, and to an asymptotic $\chi^2$-distribution of the corresponding test statistics under the null hypothesis, the quantiles of which are used for testing. Concerning `HAC tests', important contributions in the econometrics literature are \cite{NW87}, \cite{A91}, \cite{A92}, and \cite{NW94}. It is safe to say that the covariance estimators introduced in the latter two articles currently constitute the gold standard for obtaining `HAC tests'. In contrast to the estimator suggested earlier by \cite{NW87} - structurally $2\pi$ times a standard kernel spectral density estimator (\cite{Bartlett1950}, \cite{Jow1955}, \cite{Hannan1957}, and \cite{GR57} Section 7.9) evaluated at frequency $0$ - the covariance estimators suggested in \cite{A92} and \cite{NW94} both incorporate an additional prewhitening step based on an auxiliary vector autoregressive (VAR) model, as well as a data-dependent bandwidth parameter. A distinguishing feature of the estimators introduced by \cite{A92} on the one hand and \cite{NW94} on the other hand is the choice of the bandwidth parameter: \cite{A92} used an approach introduced by \cite{A91}, where the bandwidth parameter is chosen based on auxiliary parametric models. In contrast to that, \cite{NW94} suggested a nonparametric approach for choosing the bandwidth parameter. Even though simulation studies have shown that the inclusion of a prewhitening step and the data-dependent choice of the bandwidth parameter can improve the finite sample properties of `HAC tests', the more sophisticated `HAC tests' so obtained still suffer from size distortions and power deficiencies. For this reason \cite{KVB2000}, \cite{KV2002}, and \cite{KV2005} suggested to choose the bandwidth parameter as a fixed proportion of the sample size. This framework leads to an inconsistent covariance estimator and to a non-standard limiting distribution of the corresponding test statistic under the null, the quantiles of which are used to obtain so called `fixed-b tests'. In simulation studies it has been observed that `fixed-b tests' still suffer from size distortions in finite samples, but less so than `HAC tests'. However, this is at the expense of some loss in power. Similar as in `HAC testing' simulation results in \cite{KV2005} and \cite{Rho2013} suggest that the finite sample properties of `fixed-b tests' can be improved by incorporating a prewhitening step. In the latter paper it was furthermore shown that the asymptotic distribution under the null of the test suggested by \cite{KVB2000} is the same whether or not prewhitening is used. 

A number of recent studies (\cite{RV2001}, \cite{J04}, \cite{SPJ08, SPJ11}, \cite{ZhangShao2013}) tried to use higher order expansions to uncover the mechanism leading to size distortions and power deficiencies of `HAC tests' and `fixed-b tests'. These higher-order asymptotic results (and also the first-order results discussed above) are pointwise in the sense that they are obtained under the assumption of a fixed underlying data-generating-process. Hence, while they inform us about the limit of the rejection probability and the rate of convergence to this limit for a fixed underlying data-generating-process, they do not inform us about the \textit{size} of the test or its limit as sample size increases, nor about the \textit{power function} or its asymptotic behavior. Size and power properties of tests in regression models with dependent errors were recently studied in \cite{PP13}: In a general finite sample setup and under high-level conditions on the structure of the test and the covariance model, they derived conditions on the design matrix under which a concentration mechanism due to strong dependencies leads to extreme size distortions or power deficiencies. Furthermore, they suggested an adjustment-procedure to obtain a modified test with improved size and power properties. Specializing their general theory to a covariance model that includes at least all covariance matrices corresponding to stationary autoregressive processes of order one (AR(1)), they investigated finite sample properties of `HAC tests' and `fixed-b tests' based on \textit{non-prewhitened} covariance estimators with \textit{data-independent} bandwidth parameters (covering \textit{inter alia} the procedures in \cite{NW87}, Sections 3-5 of \cite{A91}, \cite{H92}, \cite{KVB2000}, \cite{KV2002, KV2005}, \cite{J02, J04}, but \textit{not} the methods considered by \cite{A92}, \cite{NW94} or \cite{Rho2013}). In this setup \cite{PP13} demonstrated that these tests break down in terms of their size or power behavior for generic design matrices. Despite this negative result, they also showed that the adjustment procedure can often solve these problems, if elements of the covariance model which are close to being singular can be well approximated by AR(1) covariance matrices. 

\cite{PP13} did not consider tests based on prewhitened covariance estimators or data-dependent bandwidth parameters. Therefore the question remains, whether the more sophisticated `HAC tests' typically used in practice (i.e., tests based on the estimators by \cite{A92} or \cite{NW94}) and the prewhitened `fixed-b tests' (i.e., tests as considered in \cite{Rho2013}) also suffer from extreme size distortions and power deficiencies, or if prewhitening and the use of data-dependent bandwidth parameters can indeed resolve or at least substantially alleviate these problems. In the present paper we investigate finite sample properties of tests based on prewhitened covariance estimators or data-dependent bandwidth parameters. In particular our analysis covers tests based on prewhitened covariance estimators using auxiliary AR(1) models for the construction of the bandwidth parameter as discussed in \cite{A92}, tests based on prewhitened covariance estimators as discussed in \cite{NW94}, and prewhitened `fixed-b' tests as discussed in \cite{Rho2013}. We show that the tests considered, albeit being structurally much more complex, exhibit a similar behavior as their \textit{non-prewhitened} counterparts with \textit{data-independent} bandwidth parameters: First, we establish conditions on the design matrix under which the tests considered have (i) size equal to one, or (ii) size not smaller than one half, or (iii) nuisance-minimal power equal to zero, respectively. We then demonstrate that at least one of these conditions is generically satisfied, showing that the tests considered break down for generic design matrices. Motivated by this negative result, we introduce an adjustment procedure. Under the assumption that elements of the covariance model which are close to being singular can be well approximated by AR(1) covariance matrices, we show that the adjustment procedure, if applicable, leads to tests that do not suffer from extreme size distortions or power deficiencies. Finally, it is shown that the adjustment procedure is applicable under generic conditions on the design matrix, unless the regression includes the intercept \textit{and} the hypothesis to be tested restricts the corresponding coefficient. On a technical level we employ the general theory developed in \cite{PP13}. We remark, however, that the genericity results in particular do not follow from this general theory. Rather they are obtained by carefully exploiting the specific structure of the procedures under consideration.

The paper is organized as follows: The framework is introduced in Section \ref{framework}. In Section \ref{tpwc} we introduce the test statistics, covariance estimators, and bandwidth parameters we analyze. In Section \ref{neg} we establish our negative result and its genericity. In Section \ref{pos} we discuss the adjustment-procedure and its generic applicability. Section \ref{concl} concludes. The proofs are collected in Appendices \ref{AA}-\ref{AD}.

\section{The Framework} \label{framework}

Consider the linear regression model 
\begin{equation} \label{lm}
\mathbf{Y}=X\beta +\mathbf{U}, 
\end{equation}
where $X$ is a (real) $n\times k$ dimensional non-stochastic design matrix satisfying $n > 2$, $\rank(X)=k$ and $1 \leq k < n$. Here, $\beta \in \mathbb{R}^{k}$ denotes the unknown regression parameter vector, and the disturbance vector $\mathbf{U} = (\mathbf{u}_1, \hdots, \mathbf{u}_n)'$ is Gaussian, has mean zero and its unknown covariance matrix is given by $\sigma^2 \Sigma$. The parameter $\sigma^2$ satisfies $0 < \sigma^2 < \infty$ and $\Sigma$ is assumed to be an element of a prescribed (non-void) set of positive definite and symmetric $n \times n$ matrices $\mathfrak{C}$, which we shall refer to as the \textit{covariance model}. Throughout we impose the assumption on $\mathfrak{C}$ that the parameters $\sigma^2$ and $\Sigma$ can be uniquely determined from $\sigma^2 \Sigma$. 

\begin{remark}\label{stationary} 
The leading case we have in mind is the situation where $\mathbf{u}_1, \hdots, \mathbf{u}_n$ are $n$ consecutive elements of a weakly stationary process. In such a setup a covariance model is typically obtained from a prescribed (non-void) set of spectral densities $\mathcal{F}$. Assuming that no element of $\mathcal{F}$ vanishes identically almost everywhere, the covariance model corresponding to $\mathcal{F}$ is then given by
\begin{equation}
\mathfrak{C}(\mathcal{F}) = \geschw{\Sigma(f): f \in \mathcal{F}},
\end{equation}
with 
\begin{equation}\label{covmat}
\Sigma(f) = \left( \int_{-\pi}^{\pi} \exp(-\iota \lambda (i-j)) f(\lambda) d\lambda \bigg / \int_{-\pi}^{\pi} f(\lambda) d \lambda  \right)_{i,j = 1}^n,
\end{equation}
and where $\iota$ denotes the imaginary unit. Every such $\Sigma(f)$ is positive definite and symmetric. Furthermore, since $\Sigma(f)$ is a correlation matrix, $\sigma^2$ and $\Sigma(f)$ can uniquely be determined from $\sigma^2 \Sigma(f)$. As outlined in the Introduction the tests we shall investigate in this paper are particularly geared towards setups where $\mathcal{F}$ is a nonparametric class of spectral densities, i.e., where the corresponding set $\mathfrak{C}(\mathcal{F})$ is rich. A typical example is the class $\mathcal{F}_{\xi}$, which consists of all spectral densities of linear processes the coefficients of which satisfy a certain summability condition, i.e., spectral densities of the form
\begin{equation}
f(\lambda) = (2\pi)^{-1} \bigg| \sum_{j = 0}^{\infty} c_j \exp(-\iota j \lambda)  \bigg|^2,
\end{equation}
where, for a fixed $\xi \geq 0$, the summability condition $0 < \sum_{j = 0}^{\infty} j^{\xi} |c_j| < \infty$ is satisfied. We observe that $\mathfrak{C}(\mathcal{F}_{\xi})$ contains in particular all correlation matrices corresponding to spectral densities of stationary autoregressive moving average models of arbitrary large order. 
\end{remark}

The linear model described in \eqref{lm} induces a collection of distributions on $(\mathbb{R}^{n}, \mathcal{B}(\Rel^n))$, the sample space of $\mathbf{Y}$. Denoting a Gaussian probability measure with mean $\mu \in \mathbb{R}^{n}$ and covariance matrix $\sigma ^{2} \Sigma $ by $P_{\mu ,\sigma ^{2}\Sigma }$ and denoting the regression manifold by $\mathfrak{M}=\lspan(X)$, the induced collection of distributions is given by 
\begin{equation}
\left\{ P_{\mu ,\sigma ^{2}\Sigma }:\mu \in \mathfrak{M},0<\sigma
^{2}<\infty , \Sigma \in \mathfrak{C}\right\}.  \label{lm2}
\end{equation}
Since every $\Sigma \in \mathfrak{C}$ is positive definite by definition, each element $P_{\mu ,\sigma ^{2}\Sigma }$ of the set in the previous display is absolutely continuous with respect to (w.r.t.) Lebesgue measure on $\mathbb{R}^{n}$. 

In this setup we shall consider the problem of testing a linear hypothesis on the parameter vector $\beta \in \mathbb{R}^{k}$, i.e., the problem of testing the null $ R\beta =r$ against the alternative $R\beta \neq r$, where $R$ is a $q\times k$ matrix of rank $q \geq 1$ and $r\in \mathbb{R}^{q}$. Define the affine space
\begin{equation*}
\mathfrak{M}_{0}=\left\{ \mu \in \mathfrak{M}:\mu =X\beta \text{ and }R\beta
=r\right\}
\end{equation*}
and let 
\begin{equation*}
\mathfrak{M}_{1}=\mathfrak{M}\backslash \mathfrak{M}_{0}=\left\{ \mu \in 
\mathfrak{M}:\mu =X\beta \text{ and }R\beta \neq r\right\} .
\end{equation*}
Adopting these definitions, the above testing problem can be written as 
\begin{equation}
H_{0}:\mu \in \mathfrak{M}_{0}, ~ 0<\sigma ^{2}<\infty, ~ \Sigma \in \mathfrak{C} ~~\text{ vs. }~~H_{1}:\mu \in \mathfrak{M}_{1}, ~ 0<\sigma ^{2}<\infty , ~ \Sigma
\in \mathfrak{C},  \label{testing problem}
\end{equation}
where it is emphasized that the testing problem is a compound one. It is immediately clear that size and power properties of tests in this setup depend in a crucial way on the richness of the covariance model $\mathfrak{C}$. 

\smallskip

Before we close this section by introducing some further terminological and notational conventions, some comments on how the above assumptions can be relaxed are in order: We remark that even though our setup assumes a non-stochastic design matrix, the results immediately carry over to a setting where the data generating processes of the design and the disturbances are independent of each other. In such a setup our results deliver size and power properties conditional on the design. The Gaussianity assumption might seem to be restrictive. However, as in Section 5.5 of \cite{PP13}, we mention that the negative results given in Section \ref{neg} of the present paper immediately extend in a trivial way without imposing the Gaussianity assumption on the error vector $\mathbf{U}$ in \eqref{lm}, as long as the assumptions on the feasible error distributions are weak enough to ensure that the implied set of distributions for $\mathbf{Y}$ contains the set in Equation \eqref{lm2}, but possibly contains also other distributions. Furthermore, by applying an invariance argument (explained in \cite{PP13} Section 5.5) one can easily show that all statements about the null-behavior of the procedures under consideration derived in the present paper carry over to the more general distributional setup where $\mathbf{U}$ is assumed to be elliptically distributed. This is to be understood as $\mathbf{U}$ having the same distribution as $\mathbf{m}\sigma\Sigma^{1/2} \mathbf{E}$, where $0 < \sigma < \infty$, $\Sigma \in \mathfrak{C}$, $\mathbf{E}$ is a random vector uniformly distributed on the unit sphere $S^{n-1}$, and $\mathbf{m}$ is a random variable distributed independently of $\mathbf{E}$ and which is positive with probability one.

\smallskip

We next collect some further terminology and notation used throughout the whole paper. A (non-randomized) \textit{test} is the indicator function of a set $W \in \mathcal{B}(\Rel^n)$, i.e., the corresponding \textit{rejection region}. The \textit{size} of such a test (rejection region) is the supremum over all rejection probabilities under the null hypothesis $H_{0}$, i.e., 
\begin{equation}
\sup_{\mu \in \mathfrak{M}_{0}}\sup_{0<\sigma ^{2}<\infty }\sup_{\Sigma \in \mathfrak{C}} P_{\mu ,\sigma ^{2}\Sigma }(W).
\end{equation}
Throughout the paper we let $\hat{\beta}_X(y) = \left( X^{\prime }X\right) ^{-1}X^{\prime }y$, where $X$ is the design matrix appearing in \eqref{lm} and $y\in \mathbb{R}^{n}$. The corresponding ordinary least squares (OLS) residual vector is denoted by $\hat{u}_X(y) = y - X\hat{\beta}_X(y)$. The subscript $X$ is omitted whenever this does not cause confusion. Random vectors and random variables are always written in bold capital and bold lower case letters, respectively. We use $\Pr$ as a generic symbol for a probability measure and denote by $E$ the corresponding expectation operator. Lebesgue measure on $\mathbb{R}^{n}$ will be denoted by $\lambda _{\mathbb{R}^{n}}$. The Euclidean norm is denoted by $\left\Vert \cdot \right\Vert $, while $d(x,A)$ denotes the Euclidean distance of the point $x\in \mathbb{R}^{n}$ to the set $A\subseteq \mathbb{R}^{n}$. For a vector $x$ in Euclidean space we define the symbol $\left\langle x\right\rangle $ to denote $\pm x$ for $x\neq 0$, the sign being chosen in such a way that the first nonzero component of $\left\langle x\right\rangle $ is positive, and we set $\left\langle 0\right\rangle =0$. The $j$-th standard basis vector in $\mathbb{R}^{n}$ is denoted by $e_{j}(n)$. Let $B^{\prime}$ denote the transpose of a matrix $B$ and let $\lspan\left( B\right) $ denote the space spanned by its columns. For a linear subspace $\mathcal{L}$ of $\mathbb{R}^{n}$ we let $\mathcal{L}^{\bot }$ denote its orthogonal complement and we let $\Pi _{\mathcal{L}}$ denote the orthogonal projection onto $\mathcal{L}$. The set of real matrices of dimension $m\times n$ is denoted by $\mathbb{R}^{m\times n}$. Lebesgue measure on this set equipped with its Borel $\sigma$-algebra is denoted by $\lambda_{\Rel^{m \times n}}$. We use the convention that the adjoint of a $1 \times 1$ dimensional matrix $D$, i.e., $\adj(D)$, equals one. Given a vector $v \in \Rel^m$ the symbol $\diag(v)$ denotes the $m \times m$ diagonal matrix with main diagonal $v$. We define 
\begin{equation}
\mathfrak{X}_0 = \geschw{X \in \Rel^{n \times k}: \rank(X) = k},
\end{equation}
i.e., the set of $n \times k$ design matrices of full rank, and whenever $k \geq 2$ we define 
\begin{equation}
\tilde{\mathfrak{X}}_0 = \geschw{\tilde{X} \in \Rel^{n \times (k-1)}: \rank((e_+, \tilde{X})) = k},
\end{equation}
which is canonically identified (as a set) with the set of $n \times k$ design matrices of full column rank the first column of which is the intercept $e_+ = (1, \hdots, 1)' \in \Rel^n$.

\section{Tests based on prewhitened covariance estimators}\label{tpwc}

In the present section we formally describe the construction of tests based on prewhitened covariance estimators. These tests (cf. Remark \ref{Rdef} below and the discussion preceding it) reject for large values of a statistic 
\begin{equation}\label{tslrv}
T(y) = \begin{cases}
(R\hat{\beta}(y) - r)' \hat{\Omega}^{-1}(y) (R\hat{\beta}(y) - r) & \text{ if } y \notin N^*(\hat{\Omega}), \\
0 & \text{ else,}
\end{cases}
\end{equation}
where 
\begin{equation}
\hat{\Omega}(y) = nR(X'X)^{-1} \hat{\Psi}(y) (X'X)^{-1} R', 
\end{equation}
and 
\begin{equation}
N^*(\hat{\Omega}) = \geschw{y \in \Rel^n: \hat{\Omega}(y) \text{ is not invertible or not well defined}}.
\end{equation}
The quantity $\hat{\Psi}$ appearing in the definition of $\hat{\Omega}$ above denotes a (VAR-) prewhitened nonparametric estimator of $n^{-1} E(X'\mathbf{U}\mathbf{U}'X)$ that incorporates a bandwidth parameter which might depend on the data. Such an estimator is completely specified by three core ingredients: First, a \textit{kernel} $\kappa: \Rel \rightarrow \Rel$, i.e., an even function satisfying $\kappa(0) = 1$, such as, e.g., the Bartlett or Parzen kernel; second, a (non-negative) possibly data-dependent \textit{bandwidth parameter} $M$; and third, a deterministic \textit{prewhitening order} $p$, i.e., an integer satisfying $1 \leq p \leq n/(k+1)$ (cf. Remark \ref{Defp}). Specific choices of $M$ are discussed in detail in Section \ref{BW}. All possible combinations of $\kappa$, $M$ and $p$ we analyze are specified in Assumption \ref{weightsPWRB} of Section \ref{comb}. Once these core ingredients have been chosen, one obtains a prewhitened estimator $\hat{\Psi}$, which is computed at an observation $y$ following the Steps (1) - (3) outlined subsequently (cf. also \cite{HL97}). We here assume that the quantities involved (e.g., inverse matrices) are well defined, cf. Remark \ref{WDPSI} below, and follow the convention in the literature and leave the estimator undefined at $y$ else. Using this convention $\hat{\Psi}(y)$ is obtained as follows:

\begin{center}
\begin{enumerate}
	\item To prewhiten the data a VAR(p) model is fitted via ordinary least squares to the columns of $\hat{V}(y) = X'\diag(\hat{u}(y))$. One so obtains the VAR(p) residual matrix $\hat{Z}(y) \in \Rel^{k \times (n-p)}$ with columns
	\begin{equation}
	\hat{Z}_{\cdot (j-p)}(y) = \hat{V}_{\cdot j}(y) - \sum_{l = 1}^p \hat{A}_l^{(p)}(y) \hat{V}_{\cdot (j-l)}(y)   \hspace{1cm} \text{ for } j = p+1, \hdots, n.
	\end{equation}
The $k \times (kp)$-dimensional VAR(p)-OLS estimator is given by
	\begin{equation}
	\hat{A}^{(p)}(y) = \left(\hat{A}^{(p)}_1(y), \hdots, \hat{A}^{(p)}_p(y)\right) = \hat{V}_p(y) \hat{V}_1'(y)\left(\hat{V}_1(y)\hat{V}_1'(y)\right)^{-1},
	\end{equation}
where $\hat{V}_p(y) = \left(\hat{V}_{\cdot(p+1)}(y), \hdots, \hat{V}_{\cdot n}(y)\right) \in \Rel^{k \times (n-p)}$ and the $j$-th column of $\hat{V}_1(y) \in \Rel^{kp \times (n-p)}$ equals $\left(\hat{V}'_{\cdot j+p-1}(y), \hdots, \hat{V}'_{\cdot j+1}(y), \hat{V}'_{\cdot j}(y)\right)' \in \Rel^{kp}$ for $j = 1, \hdots, n-p$. In matrix form we clearly have $\hat{Z}(y) = \hat{V}_p(y) - \hat{A}^{(p)}(y)\hat{V}_1(y)$.  
	\item Then, one computes the quantities
	\begin{equation}
	\check{\Gamma}_i(y) = \begin{cases}
	\frac{1}{n-p} \sum_{j = i+1}^{n-p}\hat{Z}_{\cdot j}(y)\hat{Z}'_{\cdot (j-i)}(y) & \text{ if } 0 \leq i \leq n-p-1,  \\
	\check{\Gamma}'_{-i}(y) & \text{ if } 0 < -i \leq n-p-1,
	\end{cases}
	\end{equation}
	and defines the preliminary estimate 
	\begin{equation}
	\check{\Psi}(y) = \sum_{i = -(n-p-1)}^{n-p-1} \kappa(i/M(y)) \check{\Gamma}_i(y),
	\end{equation}
	where in case $M(y) = 0$ one sets $\kappa(i/M(y)) = 0$ for $i \neq 0$ and $\kappa(i/M(y)) = \kappa(0)$ for $i = 0$.
	\item Finally, the preliminary estimate $\check{\Psi}(y) $ is `recolored' using the transformation
	\begin{equation}
	\hat{\Psi}(y) = \left(I_k-\sum_{l = 1}^p \hat{A}^{(p)}_{l}(y)\right)^{-1} \check{\Psi}(y) \left[\left(I_k-\sum_{l = 1}^p \hat{A}^{(p)}_{l}(y)\right)^{-1}\right]'.
	\end{equation}
\end{enumerate}
\end{center}

\begin{remark}\label{WDPSI}
The construction of $\hat{\Psi}(y)$ outlined above clearly assumes that (i) $\hat{A}^{(p)}(y)$ is well defined, which is equivalent to $\rank(\hat{V}_1(y)) = kp$; that (ii) $M(y)$ is well defined, which depends on the specific choice of $M$ (cf. Section \ref{BW}); and that (iii) $I-\sum_{i = 1}^p \hat{A}^{(p)}_{i}(y)$ is invertible. 
\end{remark}

\begin{remark}\label{Defp}
By assumption, all possible VAR orders $p$ we consider must satisfy $p \leq n/(k+1)$. This is done to rule out degenerate cases: for if $p  > n/(k+1)$, then $\rank(\hat{V}_1(y)) < kp$ would follow, because of $\hat{V}_1(y) \in \Rel^{kp \times (n-p)}$. Hence the covariance estimator would nowhere be well defined for such a choice, because (i) in Remark \ref{WDPSI} would then clearly be violated at every observation $y$.
\end{remark}

\begin{remark}\label{VAROLS}
In the present paper we focus on VAR prewhitening based on the OLS estimator. This is in line with the original suggestions by \cite{NW94}, as well as with \cite{Rho2013}. Alternatively, for $p = 1$, \cite{A92} suggested to use an eigenvalue adjusted version of the OLS estimator, the adjustment being applied if the matrix $I_k-\hat{A}^{(1)}_{1}(y)$ is close to being singular. We shall focus on the unadjusted OLS estimator for the following reasons: \cite{NW94} reported that the finite sample properties show little sensitivity to this eigenvalue adjustment. Furthermore, it is the unadjusted estimator that is often used in implementations of the method suggested by \cite{A92} in software packages for statistical and econometric computing (e.g., its implementation in the \texttt{R} package \texttt{sandwich} by \cite{Zeileissandwich}, or its implementation in \texttt{EViews}, e.g., \cite{schwert2009eviews}, p. 784.). We remark, however, that one can obtain a negative result similar to Theorem \ref{thmlrvPW}, and a positive result concerning an adjustment procedure similar to Theorem \ref{TU_3}, also for tests based on prewhitened estimators with eigenvalue adjustment. Furthermore, we conjecture that it is possible to prove (similar to Proposition \ref{generic}) the genericity of such a negative result, and to show that one can (similar to Proposition \ref{genericADJ}) generically resolve this problem by using the adjustment procedure. We leave the question of which estimator to choose for prewhitening to future research.
\end{remark}
In a typical asymptotic analysis of tests based on prewhitened covariance estimators the event $N^*(\hat{\Omega})$ is asymptotically negligible (since $\hat{\Omega}$ converges to a positive definite, or almost everywhere positive definite matrix). Hence there is no need to be specific about the definition of the test statistic for $y \in N^*(\hat{\Omega})$, and one can work directly with the statistic 
\begin{equation}\label{QF}
y \mapsto (R\hat{\beta}(y) - r)' \hat{\Omega}^{-1}(y) (R\hat{\beta}(y) - r),
\end{equation}
which is left undefined for $y \in N^*(\hat{\Omega})$. In a finite sample setup, however, one has to think about the definition of the test statistic also for $y \in N^*(\hat{\Omega})$. Our decision to assign the value $0$ to the test statistic for $y \in N^*(\hat{\Omega})$ is of course completely arbitrary. That this assignment does not affect our results at all is discussed in detail in the following remark.
\begin{remark}\label{Rdef}
Given that the estimator $\hat{\Omega}$ is based on a triple $\kappa$, $M$, $p$ that satisfies Assumption \ref{weightsPWRB} introduced below (which is assumed in all of our main results, and which is satisfied for covariance estimators using auxiliary AR(1) models for the construction of the bandwidth parameter as considered in \cite{A92}, for covariance estimators as considered in \cite{NW94}, and for covariance estimators as considered in \cite{Rho2013}), it follows from Lemma \ref{N*PW} that $N^*(\hat{\Omega})$ is either a $\lambda_{\Rel^n}$-null set, or that it coincides with $\Rel^n$. In the first case, which is generic under weak dimensionality constraints as shown in Lemma \ref{nullN*PW}, the definition of the test statistic on $N^*(\hat{\Omega})$ does  hence not influence the rejection probabilities, because our model is dominated by $\lambda_{\Rel^n}$ ($\mathfrak{C}$ contains only positive definite matrices). Therefore, size and power properties are not affected by the definition of the test statistic for $y \in N^*(\hat{\Omega})$. In the second case, i.e., if $N^*(\hat{\Omega})$ coincides with $\Rel^n$, the statistic in \eqref{QF} is nowhere well defined, and hence, regardless of which value is assigned to it for observations $y \in N^*(\hat{\Omega})$, the resulting test statistic is constant, and thus the test breaks down trivially.
\end{remark}

\subsection{Bandwidth parameters}\label{BW}

In the following we describe bandwidth parameters $M$ that are typically used in Step 2 in the construction of the prewhitened estimator $\hat{\Psi}$ as discussed above: The parametric approach (based on auxiliary AR(1) models) suggested by \cite{A91} and \cite{A92}, the nonparametric approach introduced by \cite{NW94}, and a data-independent approach which was already investigated in \cite{KV2005} in simulation studies and which has recently been theoretically investigated by \cite{Rho2013}. Since the bandwidth parameter $M$ is computed in Step 2 in the construction of $\hat{\Psi}(y)$, we assume that $\kappa$, $p$ and $y$ are given and that Step 1 has already been successfully completed, i.e., all operations in Step 1 are well defined at $y$, in particular $\hat{Z}(y)$ is available for the construction of $M$. If not, we leave the bandwidth parameter (and hence the covariance estimator) undefined at $y$. We also implicitly assume that the quantities and operations appearing in the procedures outlined subsequently are well defined and leave the bandwidth parameter (and hence the covariance estimator) undefined else. A detailed structural analysis of the subset of the sample space where a prewhitened estimator $\hat{\Omega}$ is well defined is then later given in Lemma \ref{NPWRB} in Section \ref{struct}. Finally, we emphasize that the bandwidth parameters discussed subsequently all require the choice of additional tuning parameters. These tuning parameters are typically chosen independently of $y$ and $X$, an assumption we shall maintain throughout the whole paper (but see Remark \ref{Rtuning} for some generalizations). 

\subsubsection{The parametric approach of \cite{A92}}\label{MA92} 

Let $\omega \in \Rel^k$ be such that $\omega \neq 0$ and $\omega_i \geq 0$ for $i = 1, \hdots, k$, i.e., $\omega$ is a \textit{weights vector}. Based on this weights vector the bandwidth parameter is now obtained as follows: First, univariate AR(1) models are fitted via OLS to $\hat{Z}_{i \cdot}(y)$ for $i = 1, \hdots, k$, giving 
\begin{align}
\hat{\rho}_i(y) &= \sum_{j = 2}^{n-p} \hat{Z}_{ij}(y) \hat{Z}_{i(j-1)}(y) ~~ \bigg \slash ~~ \sum_{j = 1}^{n-p-1} \hat{Z}_{ij}(y)^2 &&\text{ for } i = 1, \hdots, k, \\
\hat{\sigma}_i^2(y) &= (n-p-1)^{-1}\sum_{j = 2}^{n-p}\left(\hat{Z}_{ij}(y) - \hat{\rho}_i(y) \hat{Z}_{i(j-1)}(y)\right)^2 &&\text{ for } i = 1, \hdots, k,
\end{align}
where we note that $n-p-1 > 0$ holds as a consequence of $n > 2$ and $1 \leq p \leq \frac{n}{k+1}$. Then, one calculates 
\begin{align}
	&\hat{\alpha}_1(y) = \sumab{i = 1}{k}\omega_i \frac{4 \hat{\rho}_i^2(y) \hat{\sigma}_i^4(y)}{(1-\hat{\rho}_i(y))^6(1+\hat{\rho}_i(y))^2} ~~ \bigg \slash ~~ \sumab{i = 1}{k}\omega_i \frac{ \hat{\sigma}^4_i(y)}{(1-\hat{\rho}_i(y))^4}, \\
  &\hat{\alpha}_2(y) = \sumab{i = 1}{k}\omega_i \frac{4 \hat{\rho}_i(y)^2 \hat{\sigma}_i^4(y)}{(1-\hat{\rho}_i(y))^8} ~~ \bigg \slash ~~ \sumab{i = 1}{k}\omega_i \frac{ \hat{\sigma}_i^4(y)}{(1-\hat{\rho}_i(y))^4}.
\end{align}
Finally, bandwidth parameters are obtained via 
\begin{equation}
M_{AM, j, \omega, c}(y) = c_{1}\left(\hat{\alpha}_j(y)n\right)^{c_{2}} ~~ \text{ for } ~~ j = 1, 2,
\end{equation}
where to obtain a bandwidth parameter, one has to fix the constants $c_{1} > 0$, $c_{2} > 0$ and $j$ and where $c = (c_1, c_2)$. Typically the choice of these constants and the choice of $j$ depends on certain characteristics of $\kappa$ (for specific choices see \cite{A91}, Section 6, in particular p. 834). For example, if $\kappa$ is the Bartlett kernel one uses $c_{1} = 1.1447$, $c_{2} = 1/3$ and $j = 1$, or if $\kappa$ is the Quadratic-Spectral kernel one would use $c_{1} = 1.13221$, $c_{2} = 1/5$ and $j = 2$. Since we do not need such a specific dependence to derive our theoretical results, we do not impose any further assumptions on these constants beyond being positive (and independent of $y$ and $X$). We shall denote by $\mathbb{M}_{AM}$ the set of all bandwidth parameters that can be obtained as special cases of the method in the present section, by appropriately choosing - functionally independently of $y$ and $X$ - a weights vector $\omega$, constants $c_{1} > 0$, $c_{2} > 0$ and a $j \in \geschw{1, 2}$. 

\begin{remark}
Since $n$, $k$ and $q$ are fixed quantities, the tuning parameters $\omega$, $c_i$ for $i = 1, 2$ and $j$ might also depend on them, although we do not signify this in our notation. A similar remark applies to the constants appearing in Section \ref{MNW94} and in Section \ref{MR}. Although we do not provide any details, we furthermore remark that one can extend our analysis to bandwidth parameters as above, but based on estimators other than $\hat{\rho}_i$, e.g., all estimators satisfying Assumption 4 of \cite{PP13} such as the Yule-Walker estimator or variants of the OLS estimator.
\end{remark}

\subsubsection{The non-parametric approach of \cite{NW94}}\label{MNW94}

Let $\omega \in \Rel^k$ be as in Section \ref{MA92} and let $w(i) \geq 0$ for $|i| = 0, \hdots, n-p-1$ be real numbers such that $w(0) = 1$. For example, \cite{NW94} suggested to use rectangular weights, i.e., 
 \begin{equation}
 w^*(i) = \begin{cases}
 1 & \text{if } |i| \leq \lfloor4(n/100)^{2/9}\rfloor,\\
 0 & \text{else},
 \end{cases}
 \end{equation}
where $\lfloor . \rfloor$ denotes the floor function. Define for every $|i| = 0, \hdots n-p-1$
 \begin{equation}
 \bar{\sigma}_i(y) = \omega' \check{\Gamma}_i(y) \omega =  (n-p)^{-1}\sum_{j=|i| + 1}^{n-p} \omega'\hat{Z}_{\cdot j}(y)\hat{Z}_{\cdot (j-|i|)}'(y)\omega.
 \end{equation}
A bandwidth parameter is then obtained via 
\begin{equation}
	M_{NW, \omega, w, \bar{c}}(y) = \bar{c}_2 \left(\left[ \sum_{i = -(n-p-1)}^{n-p-1} |i|^{\bar{c}_1} w(i) \bar{\sigma}_i(y) ~~ \bigg \slash ~~ \sum_{i = -(n-p-1)}^{n-p-1} w(i) \bar{\sigma}_i(y) \right]^2 n \right)^{\bar{c}_3},
\end{equation}
where $\bar{c}_1$ is a positive integer, where $\bar{c}_2$ and $\bar{c}_3$ are positive real numbers and where $\bar{c} = (\bar{c}_1, \bar{c}_2, \bar{c}_3)$. These numbers are constants independent of $y$ and $X$ and have to be chosen by the user. The choice typically depends on the kernel (for the specific choices we refer the reader to \cite{NW94}, Section 3). As in the previous section, we do not impose any assumptions beyond positivity (and independence of $y$ and $X$) on the constants. Furthermore, we shall denote by $\mathbb{M}_{NW}$ the set of all bandwidth parameters that can be obtained as special cases of the method in the present section, by appropriately choosing - functionally independently of $y$ and $X$ - a weights vector, numbers $w(i) \geq 0$ for $|i| = 0, \hdots, n-p-1$, $\bar{c}_1$ a positive integer, $\bar{c}_2 > 0$ and $\bar{c}_3 > 0$.
\begin{remark}
(i) The method described here is the `real-bandwidth' approach suggested in \linebreak \cite{NW94}, as opposed to the `integer-bandwidth' approach. In the latter approach one would use $1 + \left\lfloor  M_{NW, \omega, w, \bar{c}}(y) \right\rfloor$ instead of $M_{NW, \omega, w, \bar{c}}(y)$. Both approaches are asymptotically equivalent (\cite{NW94}, Theorem 2) for most kernels (including the Bartlett kernel which is suggested in \cite{NW94}). Therefore, they are equally plausible in terms of their theoretical foundation. For the sake of simplicity and comparability with the bandwidth parameter as suggested by \cite{A92}, which is not an integer in general, we have chosen to focus on the `real-bandwidth' approach. \\
(ii) \cite{NW94}, p. 637, in principle also allow for $\bar{c}_1 = 0$ ($q = 0$ in their notation) in the definition of their estimator. We do not allow for such a choice. However, note that $\bar{c}_1 = 0$ implies $M_{NW, \omega, w, \bar{c}}(y) \equiv \bar{c}_2 n^{\bar{c}_3}$. This is a data-independent bandwidth parameter. These parameters are separately treated in Section \ref{MR}. \\
\end{remark}

\subsubsection{Data-independent bandwidth parameters}\label{MR}

\cite{KV2005} and \cite{Rho2013} studied properties of prewhitened `fixed-b tests'. Here one sets $M \equiv b (n-p)$ where $b \in (0, 1]$ is functionally independent of $y$ and $X$. For example, in \cite{Rho2013} the choice $b = 1$ is studied. These approaches all lead to bandwidth parameters $M_{KV}  > 0$, that are functionally independent of both $X$ and $y$. We denote the set of such bandwidth parameters by $\mathbb{M}_{KV}$.

\subsection{Assumptions on $\kappa$, $M$ and $p$}\label{comb}

Different combinations of kernels $\kappa$, bandwidth parameters $M$ and VAR orders $p$ obviously lead to different estimators. We indicate the dependence of the estimator on these quantities by writing $\hat{\Omega}_{\kappa, M, p}$. In the present paper we shall consider estimators $\hat{\Omega}_{\kappa, M, p}$ based on a triple $\kappa$, $M$, $p$ which satisfies the following assumption:
\begin{assumption}\label{weightsPWRB}
The triple $\kappa$, $M$, $p$ satisfies:
\begin{enumerate}
	\item $\kappa: \Rel \rightarrow \Rel$ is an even function and $\kappa(0) = 1$. Furthermore, $\kappa$ is continuous, satisfies $\lim_{x \rightarrow \infty}\kappa(x) = 0$, and for every real number $s > 0$ and every positive integer $J$ the $J \times J$ symmetric Toeplitz matrix with $ij$-th coordinate $\kappa((i-j)/s)$ is positive definite.
	\item $M \in \mathbb{M}_{AM} \cup \mathbb{M}_{NW} \cup \mathbb{M}_{KV}$.
	\item $p$ is an integer satisfying $1 \leq p \leq n/(k+1)$.
\end{enumerate}
\end{assumption}
\begin{remark}\label{Rweights}
First, we remark that the positive definiteness assumption in Part 1 of Assumption \ref{weightsPWRB} is natural in our context, because it guarantees that $\hat{\Omega}_{\kappa, M, p}$ is nonnegative definite whenever it is well defined. Furthermore, it allows us to derive simple conditions for positive definiteness of $\hat{\Omega}_{\kappa, M, p}$ (cf. Lemma \ref{N*PW}). It is well known that many kernels used in practice satisfy the positive definiteness assumption, e.g., the Bartlett, Parzen, and Quadratic-Spectral kernel. Secondly, we note that in principle Assumption \ref{weightsPWRB} does not prohibit a combination of $M_{AM, 1, \omega, c} \in \mathbb{M}_{AM}$ with a second order kernel, or the combination of $M_{AM, 2, \omega, c}  \in \mathbb{M}_{AM}$ with a first order kernel. It also allows for a combination of elements of $\mathbb{M}_{NW}$ with a prewhitening order $p > 1$ and for the combination of elements of $\mathbb{M}_{KV}$ with a kernel other than the Bartlett kernel. This goes well beyond the original suggestions in \cite{A92}, \cite{NW94} and \cite{Rho2013}, but we include these additional possibilities for convenience. We also remark that since we assume throughout that $n > k$, the set of VAR orders satisfying the third part of Assumption \ref{weightsPWRB} always includes the order $p = 1$. 
\end{remark}

\begin{remark}[\textit{Tuning parameters depending on the design}]\label{Rtuning}
The tuning parameters used in the construction of $M \in \mathbb{M}_{AM} \cup \mathbb{M}_{NW} \cup \mathbb{M}_{KV}$, e.g., the weights vector $\omega$ used in the construction of $M \in \mathbb{M}_{AM} \cup \mathbb{M}_{NW}$, are by definition functionally independent of $y$ \textit{and} $X$. Requiring that the tuning parameters are independent of $X$ is not a restriction in all results of the present paper in which the design matrix $X$ is \textit{fixed} (i.e., Theorem \ref{thmlrvPW}, Proposition \ref{excPW}, and Theorem \ref{TU_3}). To see this, suppose that a design matrix $X$ as in \eqref{lm} is given, that $\kappa$ and $p$ satisfy the first and third part of Assumption \ref{weightsPWRB}, respectively, and that $M$ is constructed as in one of the Sections \ref{MA92}, \ref{MNW94}, \ref{MR}, but with a vector of tuning parameters $c^*(.)$, say, that is not constant on $\mathfrak{X}_0$. The triple $\kappa, M, p$ hence does not satisfy Assumption \ref{weightsPWRB}. Let $\tilde{M}$ be the bandwidth parameter that is obtained from $M$ by replacing the vector of tuning parameters $c^*(.)$ by $\tilde{c} \equiv c^*(X)$. Clearly, $\kappa, \tilde{M}, p$ satisfies Assumption \ref{weightsPWRB}, and the test statistics as in Equation \eqref{tslrv} based on $\hat{\Omega}_{\kappa, \tilde{M}, p}$ and $\hat{\Omega}_{\kappa,M, p}$, respectively, coincide for this specific $X$. 
\end{remark}

\subsection{Structural properties of prewhitened covariance estimators} \label{struct}

The study of finite sample properties of a test based on the statistic in Equation \eqref{tslrv} with $\hat{\Omega} = \hat{\Omega}_{\kappa, M, p}$ requires a detailed understanding of definiteness properties of the covariance estimator $\hat{\Omega}_{\kappa, M, p}$, and of the structure of the set $N^*(\hat{\Omega}_{\kappa, M, p})$. Denoting the subset of the sample space $\Rel^n$ where $\hat{\Omega}_{\kappa, M, p}$ is not well defined by $N(\hat{\Omega}_{\kappa, M, p})$, we can write
\begin{equation}
N^*(\hat{\Omega}_{\kappa, M, p}) = N(\hat{\Omega}_{\kappa, M, p}) \cup \geschw{y \in \Rel^n \backslash N(\hat{\Omega}_{\kappa, M, p}): \det(\hat{\Omega}_{\kappa, M, p}(y)) = 0}.
\end{equation}
As a first step we study $N(\hat{\Omega}_{\kappa, M, p})$ in the subsequent lemma, where it is shown that $N(\hat{\Omega}_{\kappa, M, p})$ is algebraic. The lemma also characterizes the dependence of $N(\hat{\Omega}_{\kappa, M, p})$ on the design matrix, which will later be useful for obtaining our genericity results.
\begin{lemma} \label{NPWRB}
Assume that the triple $\kappa$, $M$, $p$ satisfies Assumption \ref{weightsPWRB}. Then,
\begin{equation}
N(\hat{\Omega}_{\kappa, M, p}) = \geschw{y \in \Rel^n: g_{\kappa, M, p}(y, X) = 0},
\end{equation}
where $g_{\kappa, M, p}: \Rel^n \times \Rel^{n \times k} \rightarrow \Rel$ is a multivariate polynomial (explicitly constructed in the proof). As a consequence $N(\hat{\Omega}_{\kappa, M, p})$ is an algebraic set. Furthermore, $g_{\kappa, M, p}$ does not depend on the hypothesis $(R,r)$.  
\end{lemma}
The subsequent lemma discusses definiteness and regularity properties of $\hat{\Omega}_{\kappa, M, p}$ and shows that $N^*(\hat{\Omega}_{\kappa, M, p})$ is an algebraic subset of $\Rel^n$. Again the dependence of this algebraic set on the design is clarified. Given a prewhitening order $p$ satisfying Part 3 of Assumption \ref{weightsPWRB}, we define for every $y \in \Rel^n$ such that $\hat{A}^{(p)}(y)$ is well defined and such that $I_k - \sum_{l = 1}^p \hat{A}^{(p)}_l(y)$ is invertible the matrix
\begin{equation}\label{defB}
B_p(y) = R(X'X)^{-1}\left(I_k - \sum_{l = 1}^p \hat{A}^{(p)}_l(y)\right)^{-1} \hat{Z}(y).
\end{equation}
\begin{lemma}\label{N*PW}
Assume that the triple $\kappa$, $M$, $p$ satisfies Assumption \ref{weightsPWRB}. Then the following holds:
\begin{enumerate}
	\item $\hat{\Omega}_{\kappa, M, p}(y)$ is nonnegative definite if and only if $g_{\kappa, M, p}(y, X) \neq 0$.
	\item $\hat{\Omega}_{\kappa, M, p}(y)$ is singular if and only if $g_{\kappa, M, p}(y, X) \neq 0$ and $\rank({B}_p(y)) < q$.
	\item $\hat{\Omega}_{\kappa, M, p}(y) = 0$ if and only if $g_{\kappa, M, p}(y, X) \neq 0$ and ${B}_p(y) = 0$. 
	\item $\hat{\Omega}_{\kappa, M, p}(y)$ is positive definite if $g_{\kappa, M, p}(y, X) \neq 0$ and $\rank(\hat{Z}(y)) = k$. 
	\item We have 
	\begin{equation}
	N^*(\hat{\Omega}_{\kappa, M, p}) = \geschw{y \in \Rel^n : g^*_{\kappa, M, p}(y, X, R) = 0},
	\end{equation}
where $g^*_{\kappa, M, p}: \Rel^n \times \Rel^{n \times k} \times \Rel^{q \times k} \rightarrow \Rel$ is a multivariate polynomial (explicitly constructed in the proof). As a consequence $N^*(\hat{\Omega}_{\kappa, M, p})$ is an algebraic set. Furthermore, $g^*_{\kappa, M, p}$ is independent of $r$.
\end{enumerate}
\end{lemma}

It is a well known fact that an algebraic subset of $\Rel^n$ is either a closed $\lambda_{\Rel^n}$-null set, or coincides with $\Rel^n$ (for a proof see, e.g., \cite{okamoto1973}). The latter case occurs if and only if a (multivariate) polynomial defining the algebraic set vanishes everywhere. Together with Part 5 of Lemma \ref{N*PW} this implies that $N^*(\hat{\Omega}_{\kappa, M, p})$ is either a closed $\lambda_{\Rel^n}$-null set, or coincides with $\Rel^n$, depending on whether $g^*_{\kappa, M,p}(., X, R) \not \equiv 0$ or $g^*_{\kappa, M,p}(., X, R) \equiv 0$ holds, respectively. In the latter case, every test based on the test statistic defined in Equation \eqref{tslrv} with $\hat{\Omega} = \hat{\Omega}_{\kappa, M, p}$ trivially breaks down, because in this case the test statistic vanishes identically on $\Rel^n$. Obviously, studying size and power properties of tests based on this test statistic in a sample of size $n$ is only interesting, if we can guarantee that $g^*_{\kappa, M,p}(., X, R) \not \equiv 0$ holds for a sufficiently large set of design matrices. That this is indeed the case is the content of the subsequent lemma. More precisely it is shown that $g^*_{\kappa, M,p}(., X, R) \not \equiv 0$ is generically satisfied whenever $n$ exceeds a certain threshold. It is also shown that the threshold we give can not be substantially improved. The notion of genericity employed is further discussed in Remark \ref{Rgenericityomega} following the lemma. 
\begin{lemma}\label{nullN*PW}
Assume that the triple $\kappa$, $M$, $p$ satisfies Assumption \ref{weightsPWRB}. Then the following holds:
\begin{enumerate}
\item If $n < k(p+1) + p$ and $q = k$, then 
\begin{equation}
g^*_{\kappa, M, p}(., X, R) \equiv 0 \text{ for every } X \in \mathfrak{X}_0.  
\end{equation}
\item If $k(p+1) + p + \mathbf{1}_{\mathbb{M}_{AM}}(M) \leq n$, then 
\begin{equation}
g^*_{\kappa, M, p}(., X, R) \not \equiv 0 \text{ for } \lambda_{\Rel^{n \times k}}\text{-almost every } X \in \mathfrak{X}_0;
\end{equation}
if $k = 1$ we have in particular $g^*_{\kappa, M, p}(., e_+, R) \not \equiv 0$.
\item If $k \geq 2$ and $k(p+1) + p^* +   \mathbf{1}_{\mathbb{M}_{AM}}(M) \leq n$, where $p^* = p + (p \bmod 2)$, then 
\begin{equation}
g^*_{\kappa, M, p}(., (e_{+}, \tilde{X}), R) \not \equiv 0 \text{ for } \lambda_{\Rel^{n \times (k-1)}}\text{-almost every } \tilde{X} \in \mathfrak{\tilde{X}}_0.
\end{equation} 
\end{enumerate}
\end{lemma}
\begin{remark} \label{Rgenericityomega}
(1) Part 1 demonstrates that if $n$ is too small in the sense that $n < k(p+1) + p$, then for every $X \in \mathfrak{X}_0$ the test statistic in Equation \eqref{tslrv} with $\hat{\Omega} = \hat{\Omega}_{\kappa, M, p}$ vanishes identically if $q = k$ holds, because the estimator $\hat{\Omega}_{\kappa, M, p}$ is either not well defined or singular at every observation $y$. This shows that one can in general not expect that $N^*(\hat{\Omega}_{\kappa, M, p})$ is generically a $\lambda_{\Rel^n}$-null set in case $n < k(p+1) + p$. 

(2) Under the assumption that $k(p+1) + p + \mathbf{1}_{\mathbb{M}_{AM}}(M) \leq n$ holds, Part 2 establishes genericity of $g^*_{\kappa, M, p}(., X, R) \not \equiv 0$ in that it shows that the statement holds for $\lambda_{\Rel^{n \times k}}$- almost every $X \in \mathfrak{X}_0$. This notion of genericity is obviously related to situations, where the data-generating process underlying the design matrix $X$ is assumed to be absolutely continuous w.r.t. $\lambda_{\Rel^{n \times k}}$. In this situation, a bandwidth parameter $M \in \mathbb{M}_{AM} \cup \mathbb{M}_{NW}$ would typically be based on the weights vector $\omega = (1, \hdots, 1)' \in \Rel^k$. As a specific result of independent interest it is also shown that if $k = 1$ then $g^*_{\kappa, M, p}(., e_+, R) \not \equiv 0$ holds, which means that in the location model the set $N^*(\hat{\Omega}_{\kappa, M, p})$ is a $\lambda_{\Rel^n}$- null set. 

(3) Under the assumption that $k \geq 2$ and $k(p+1) + p^* +   \mathbf{1}_{\mathbb{M}_{AM}}(M) \leq n$ holds, Part 3 establishes genericity of $g^*_{\kappa, M, p}(., (e_+, \tilde{X}), R) \not \equiv 0$ by showing that the statement holds for $\lambda_{\Rel^{n \times (k-1)}}$ almost every $\tilde{X} \in \tilde{\mathfrak{X}}_0$. This is a genericity statement concerning design matrices the first column of which is the intercept. In contrast to (2) this notion of genericity is related to situations, where the first column of the design matrix is fixed and the data-generating process underlying the remaining columns is absolutely continuous w.r.t. $\lambda_{\Rel^{n \times (k-1)}}$. In such a setup the construction of a bandwidth parameter $M \in \mathbb{M}_{AM} \cup \mathbb{M}_{NW}$ would typically be based on the weights vector $\omega = (0, 1 \hdots, 1)' \in \Rel^k$. 
\end{remark}

\section{A negative result and its generic applicability}\label{neg}

In the first part of this section we obtain our main negative result concerning finite sample properties of tests based on prewhitened nonparametric covariance estimators. For this result to hold, we have to impose a richness assumption on the covariance model $\mathfrak{C}$. Let $\mathfrak{C}_{AR(1)}$ denote the set of all correlation matrices corresponding to stationary autoregressive processes of order one, i.e., $\mathfrak{C}_{AR(1)} = \geschw{\Lambda(\rho): \rho \in (-1, 1)}$, where $\Lambda(\rho)_{ij} = \rho^{|i-j|}$ for $1 \leq i,j \leq n$. The assumption is as follows.
\begin{assumption}\label{AAR(1)}
$ \mathfrak{C}_{AR(1)} \subseteq \mathfrak{C}$.
\end{assumption}
\begin{remark}\label{singboundary}
Assumption \ref{AAR(1)} implies in particular that the singular boundary of $\mathfrak{C} \subseteq \Rel^{n \times n}$, i.e., the set of singular matrices in $\boundary{\mathfrak{C}}$, contains at least the two elements $e_+ e_+'$ and $e_- e_-'$, where $e_+ = (1, \hdots, 1)'$ and $e_- = (-1, 1, \hdots, (-1)^n)'$. We note that these two singular matrices can be approximated by sequences $\Lambda(\rho_m) \in \mathfrak{C}$ with $\rho_m \rightarrow 1$ and $\rho_m \rightarrow -1$, respectively, where $\rho_m \in (-1,1)$. 
\end{remark}
Since the procedures we study in the present paper are geared towards situations such as $\mathfrak{C} \supseteq \mathfrak{C}_{\xi}$ for some $\xi \geq 0$ (cf. Remark \ref{stationary}), covariance models which clearly satisfy the above assumption, Assumption \ref{AAR(1)} is mild in our context (cf. also the discussion in Section 3.2.2 of \cite{PP13}). Under this assumption and given a hypothesis $(R,r)$, the subsequent theorem provides four sufficient conditions on the design matrix under which a test based on a test statistic as in Equation \eqref{tslrv} with $\hat{\Omega} = \hat{\Omega}_{\kappa, M, p}$, together with an arbitrary (but data-independent) critical value $0 < C < \infty$, breaks down in terms of its finite sample size and/or power properties. More precisely, Conditions (1) and (4) imply that the test has size equal to one, Condition (3) implies that the test has size not smaller than $1/2$, and Condition (2) implies that the nuisance-minimal rejection probability equals zero at every point $\mu_1 \in \mathfrak{M}_1$.
\begin{theorem}\label{thmlrvPW}
Suppose that the triple $\kappa$, $M$, $p$ satisfies Assumption \ref{weightsPWRB} and that $\mathfrak{C}$ satisfies Assumption \ref{AAR(1)}. Let $T$ be the test statistic defined in \eqref{tslrv} with $\hat{\Omega} = \hat{\Omega}_{\kappa, M, p}$. Let $W(C)=\left\{y\in \mathbb{R}^{n}:T(y)\geq C\right\}$ be the rejection region, where $C$ is a real number satisfying $0<C<\infty $. Then the following holds:
\begin{enumerate}
\item Suppose $g^*_{\kappa, M, p}(e_{+}, X, R) \neq 0$ and $T(e_{+}+\mu_{0}^{\ast })>C$ holds for some (and hence all) $\mu _{0}^{\ast }\in \mathfrak{M}_{0}$, or $g^*_{\kappa, M, p}(e_{-}, X, R) \neq 0$ and $T(e_{-}+\mu _{0}^{\ast })>C$ holds for some (and hence all) $\mu _{0}^{\ast}\in \mathfrak{M}_{0}$. Then 
\begin{equation*}
\sup\limits_{\Sigma \in \mathfrak{C}}P_{\mu _{0},\sigma ^{2}\Sigma }\left(W \left( C\right) \right) =1
\end{equation*}
holds for every $\mu _{0}\in \mathfrak{M}_{0}$ and every $0<\sigma^{2}<\infty $. In particular, the size of the test is equal to one.
\item Suppose $g^*_{\kappa, M, p}(e_{+}, X, R) \neq 0$ and $T(e_{+}+\mu_{0}^{\ast })<C$ holds for some (and hence all) $\mu _{0}^{\ast }\in \mathfrak{M}_{0}$, or $g^*_{\kappa, M, p}(e_{-}, X, R) \neq 0$ and $T(e_{-}+\mu _{0}^{\ast })<C$ holds for some (and hence all) $\mu _{0}^{\ast}\in \mathfrak{M}_{0}$. Then 
\begin{equation*}
\inf_{\Sigma \in \mathfrak{C}}P_{\mu _{0},\sigma ^{2}\Sigma }\left(W\left(C\right) \right) =0
\end{equation*}
holds for every $\mu _{0}\in \mathfrak{M}_{0}$ and every $0<\sigma^{2}<\infty $, and hence 
\begin{equation*}
\inf_{\mu _{1}\in \mathfrak{M}_{1}}\inf_{\Sigma \in \mathfrak{C}}P_{\mu_{1},\sigma ^{2}\Sigma }\left( W\left( C\right) \right) =0
\end{equation*}
holds for every $0<\sigma ^{2}<\infty $. In particular, the test is biased. Furthermore, the nuisance-infimal rejection probability at every point $\mu_{1}\in \mathfrak{M}_{1}$ is zero, i.e., 
\begin{equation*}
\inf\limits_{0<\sigma ^{2}<\infty }\inf\limits_{\Sigma \in \mathfrak{C}}P_{\mu _{1},\sigma ^{2}\Sigma }(W\left( C\right) )=0.
\end{equation*}
In particular, the infimal power of the test is equal to zero.
\item Suppose $g^*_{\kappa, M, p}(e_{+}, X, R) \neq 0$, $T(e_{+}+\mu_{0}^{\ast }) = C$ and $\grad T(e_{+} + \mu_0^{\ast})$ exists for some (and hence all) $\mu _{0}^{\ast }\in \mathfrak{M}_{0}$, or $g^*_{\kappa, M, p}(e_{-}, X, R) \neq 0$, $T(e_{-}+\mu _{0}^{\ast })= C$ and $\grad T(e_{-} + \mu_0^{\ast})$ exists for some (and hence all) $\mu _{0}^{\ast}\in \mathfrak{M}_{0}$. Then 
\begin{equation*}
\sup\limits_{\Sigma \in \mathfrak{C}}P_{\mu _{0},\sigma ^{2}\Sigma }\left(W \left( C\right) \right) \geq 1/2
\end{equation*}
holds for every $\mu _{0}\in \mathfrak{M}_{0}$ and every $0<\sigma^{2}<\infty $. In particular, the size of the test is at least $1/2$.
\item Suppose that $g^*_{\kappa, M, p}(., X, R) \not \equiv 0$. Suppose further that $e_{+} \in \mathfrak{M}$ and $R\hat{\beta}(e_{+})\neq 0$ holds, or $e_{-} \in \mathfrak{M}$ and $R\hat{\beta}(e_{-})\neq 0$ holds. Then 
\begin{equation*}
\sup\limits_{\Sigma \in \mathfrak{C}}P_{\mu _{0},\sigma ^{2}\Sigma }\left(
W\left( C\right) \right) =1
\end{equation*}
holds for every $\mu _{0}\in \mathfrak{M}_{0}$ and every $0<\sigma^{2}<\infty $. In particular, the size of the test is equal to one.
\end{enumerate}
\end{theorem}
\begin{remark}\label{remarkthmlrvPW}
(i) Lemma \ref{A567PW} in Appendix \ref{Aneg} shows that the rejection probabilities $P_{\mu, \sigma^2\Sigma}(W(C))$ depend on $(\mu, \sigma^2, \Sigma)$ only through $(\langle (R\beta - r)/\sigma \rangle, \Sigma)$, where $\beta$ is uniquely determined by $X\beta = \mu$. \\
(ii) Obviously, the conclusions of the preceding theorem also apply to any rejection region $W^* \in \mathcal{B}(\Rel^n)$ which differs from $W(C)$ only by a $\lambda_{\Rel^n}$-null set. \\
(iii) In Part 1 of the theorem the condition $g^*_{\kappa, M, p}(e_{+}, X, R) \neq 0$ ($g^*_{\kappa, M, p}(e_{-}, X, R) \neq 0$) is superfluous, because it is already implicit in $T(e_{+}+\mu_{0}^{\ast })> C> 0$ ($T(e_{-}+\mu _{0}^{\ast })>C > 0$), which is readily seen from the definition of $T$ in Equation \eqref{tslrv}. A similar comment applies to Part 3 of the theorem, where the condition $g^*_{\kappa, M, p}(e_{+}, X, R) \neq 0$ ($g^*_{\kappa, M, p}(e_{-}, X, R) \neq 0$) is already implicit in $T(e_{+}+\mu_{0}^{\ast }) = C> 0$ ($T(e_{-}+\mu _{0}^{\ast })=C>0$). The conditions are included for the sake of comparability with Part 2 of the theorem. \\
(iv) In case $M \in \mathbb{M}_{KV}$, the assumption concerning the existence of the gradient can be dropped in Part 3 of the theorem. This follows from Lemma \ref{gradientT} in Appendix \ref{Aneg}, where it is shown that if $M \in \mathbb{M}_{KV}$, then the existence of $\grad T(e_+ + \mu_0^*)$ and $\grad T(e_- + \mu_0^*)$ is already implied by $g^*_{\kappa, M, p}(e_{+}, X, R) \neq 0$ and $g^*_{\kappa, M, p}(e_{-}, X, R) \neq 0$, respectively. \\
(v) Throughout the theorem, Assumption \ref{AAR(1)} can be replaced by the weaker assumption that there exist two sequences $\Lambda(\rho_m^{(1)})$ and $\Lambda(\rho_m^{(2)})$ of AR(1) correlation matrices in $\mathfrak{C}$, such that $\rho_m^{(1)} \rightarrow -1$ and $\rho_m^{(2)} \rightarrow 1$. In Parts 1 and 2 of the theorem it is even enough to assume that there exist sequences $\Sigma_m^{(i)} \in \mathfrak{C}$ for $i = 1, 2$ with $\Sigma_m^{(1)} \rightarrow e_+e_+'$ and $\Sigma_m^{(2)} \rightarrow e_-e_-'$. Therefore, in these parts it is only important that - and not how - these singular matrices can be approximated from within $\mathfrak{C}$. 
\end{remark}
We shall now provide some intuition for Theorem \ref{thmlrvPW} (cf. also the discussion preceding Theorem 5.7 in \cite{PP13}). The repeated appearance of the vectors $e_+$ and $e_-$ in the theorem stems from the fact that both $e_+e_+'$ and $e_-e_-'$ are elements of the singular boundary of $\mathfrak{C} \supseteq \mathfrak{C}_{AR(1)}$ (cf. Remark \ref{singboundary}). Furthermore, for every $\mu_0^* \in \mathfrak{M}_0$ and every $0 < \sigma^2 < \infty$ we have that $P_{\mu_0^*, \sigma^2 \Sigma} \rightarrow P_{\mu_0^*, \sigma^2 e_+ e_+'}$ weakly as $\Sigma \rightarrow e_+e_+'$ with $\Sigma \in \mathfrak{C}$, and similarly that $P_{\mu_0^*, \sigma^2 \Sigma} \rightarrow P_{\mu_0^*, \sigma^2 e_- e_-'}$ weakly as $\Sigma \rightarrow e_-e_-'$ with $\Sigma \in \mathfrak{C}$. These limiting measures are absolutely continuous w.r.t. $\lambda_{\mu_0^* +  \lspan(e_+)}$ and $\lambda_{\mu_0^* +  \lspan(e_-)}$, respectively. As a consequence we see that the mass of $P_{\mu_0^*, \sigma^2 \Sigma} \in \mathfrak{P}$ concentrates on `neighborhoods' of certain one-dimensional affine spaces as $\Sigma$ approximates $e_+e_+'$ or $e_-e_-'$ from within $\mathfrak{C}$. From that it is intuitively clear that size and power properties crucially depend on the behavior of the tests on `neighborhoods' of these spaces. The first and second part of the theorem provide sufficient conditions under which these spaces are almost surely (w.r.t. $\lambda_{\mu_0^* +  \lspan(e_+)}$ and $\lambda_{\mu_0^* +  \lspan(e_-)}$) contained in the interior or exterior of the rejection region, respectively. The former case then leads to size distortions, the latter to power deficiencies. The situation in the third part of the theorem is quite different and more complex. In this case the one-dimensional affine space supporting the respective limiting measure is neither almost surely contained in the interior, nor almost surely contained in the exterior of the rejection region. Rather it is almost surely contained in the boundary of the rejection region. Therefore, in contrast to Parts 1 and 2, it is not only important that the measures concentrate on the respective one-dimensional space, but also \textit{how} they concentrate (cf. Remark \ref{remarkthmlrvPW} (v)). The concentration turns out to be such that eventually the measures put roughly equal weight onto the rejection region and onto its complement, resulting in rejection probabilities as large as $1/2$ under the null. We point out that the proof idea used to establish Part 3 is inspired by the proof of Theorem 2.20 in \cite{PP14}. The last part of the theorem considers the case where one of the vectors $e_+$ or $e_-$ is an element of $\mathfrak{M}$ that is also `involved' in the hypothesis. It is then shown that the size of the test is one if the global condition $g^*_{\kappa, M, p}(., X, R) \not \equiv 0$ is satisfied. We recall that if this condition fails to hold, then the test $T$ based on $\hat{\Omega}_{\kappa, M, p}$ breaks down in a trivial way, because $T$ is then zero everywhere. Therefore we see that under Assumption \ref{AAR(1)} one simply can not test a hypothesis involving $e_+ \in \mathfrak{M}$ or $e_-\in \mathfrak{M}$ by means of a test $T$ based on $\hat{\Omega}_{\kappa, M, p}$ with $\kappa$, $M$, $p$ satisfying Assumption \ref{weightsPWRB} (this in particular covers the location model where $X = e_+$, cf. also Lemma \ref{nullN*PW}, Part 2).

\begin{remark}\label{boundaway}
Suppose that it is \textit{known a priori} that for some (fixed) $\varepsilon \in (0, 1]$ the covariance model $\mathfrak{C}$ does not contain AR(1) correlation matrices $\Lambda(\rho)$ with $\rho \leq -1+\varepsilon$; i.e., instead of Assumption \ref{AAR(1)} the covariance model $\mathfrak{C}$ satisfies
\begin{equation}\label{alternativecov}
\mathfrak{C}_{AR(1)}(\varepsilon) = \geschw{\Lambda(\rho): \rho \in (-1+\varepsilon, 1)} \subseteq \mathfrak{C}. 
\end{equation}
Inspection of the proof of Theorem \ref{thmlrvPW} then shows that a version of Theorem \ref{thmlrvPW} holds, in which all references to $e_-$ are deleted in Parts 1-4. For example, Part 4 of this version of Theorem \ref{thmlrvPW} reads as follows: 
\begin{quote}
``Suppose that $g^*_{\kappa, M, p}(., X, R) \not \equiv 0$. Suppose further that $e_{+} \in \mathfrak{M}$ and $R\hat{\beta}(e_{+})\neq 0$ holds. Then 
\begin{equation*}
\sup\limits_{\Sigma \in \mathfrak{C}}P_{\mu _{0},\sigma ^{2}\Sigma }\left(
W\left( C\right) \right) =1
\end{equation*}
holds for every $\mu _{0}\in \mathfrak{M}_{0}$ and every $0<\sigma^{2}<\infty $. In particular, the size of the test is equal to one.''
\end{quote}
This statement covers (in particular) the important special case of testing a restriction on the mean in a location model. We make the following observations concerning this version of Theorem \ref{thmlrvPW}:
\begin{itemize}
	\item Since $e_-e_-'$ is not necessarily an element of the singular boundary of the covariance model considered here, the result just described does not contain ``size equal to one''- or ``nuisance-minimal-power equal to zero''-statements that arise from covariance matrices approaching $e_-e_-'$. Note, however, that the original Theorem \ref{thmlrvPW} implies by a continuity argument that if $\varepsilon$ is small (compared to sample size), then considerable size distortions or power deficiencies will nevertheless be present for covariance matrices in $\mathfrak{C}$ that are close to $e_-e_-'$. 
  \item Consider the case where $e_+ \in \mathfrak{M}$, i.e., the regression contains an intercept, and where the hypothesis does not involve the intercept, i.e., $R\hat{\beta}(e_+) = 0$: Then we see that Parts 1-4 of the version of Theorem \ref{thmlrvPW} just obtained do not apply. In fact, in this case we can establish a positive result concerning a test based on $T$ with $\hat{\Omega} = \hat{\Omega}_{\kappa, M, p}$, and based on a non-standard critical value that depends on $\varepsilon$. This positive result, together with its restrictions, is discussed in Remark \ref{Rposbound}.
\end{itemize}
\end{remark}

Given a hypothesis $(R,r)$ the four sufficient conditions provided in the preceding theorem are conditions on the design matrix $X$. They depend on observable quantities only. How these conditions can be checked is discussed in the subsequent paragraph: The first three parts of the theorem operate under the local assumption that the multivariate polynomial $g^*_{\kappa, M, p}(., X, R)$ does not vanish at the point $e_+$ or $e_-$, respectively. The multivariate polynomial $g^*_{\kappa, M, p}(., X, R)$ is explicitly constructed in the proof of Lemma \ref{N*PW}. Therefore, the condition that it does not vanish at specific data points can readily be checked.  Some additional conditions needed in Parts 1-3 of the theorem are formulated in terms of $T(e_+ + \mu_0^*)$ and $T(e_- + \mu_0^*)$, which are in fact independent of the specific $\mu_0^* \in \mathfrak{M}_0$ chosen and therefore easy to calculate. Part 3 of the theorem requires the existence of $\grad T(e_+ + \mu_0^*)$ or $\grad T(e_- + \mu_0^*)$ (which is immaterial if $M \in \mathbb{M}_{KV}$ as discussed in the preceding Remark). Again the existence of the gradients is independent of the specific choice of $\mu_0^* \in \mathfrak{M}_0$. Sufficient conditions for the existence of the gradient, under the assumption that $\kappa$ is continuously differentiable on the complement of a finite number of points, are provided in Lemma \ref{gradientT} in Appendix \ref{Aneg}. These conditions amount to checking whether or not $M(e_+)$ or $M(e_-)$, respectively, is an element of a certain set determined by $\kappa$ consisting of finitely many points. In contrast to Parts 1-3, the fourth part of the theorem operates under the global assumption that the multivariate polynomial $g^*_{\kappa, M, p}(., X, R)$ is not the zero polynomial. Since the polynomial $g^*_{\kappa, M, p}(., X, R)$ is explicitly constructed in the proof of Lemma \ref{N*PW}, the global assumption $g^*_{\kappa, M, p}(., X, R) \not \equiv 0$ can either be checked analytically, or by using standard algorithms for polynomial identity testing. In addition to this global assumption, the fourth part needs additional assumptions on the structure of $\mathfrak{M}$ and the hypothesis $(R,r)$ which can of course be easily checked by the user.

\smallskip

The preceding theorem has given sufficient conditions on the design matrix, under which the test considered breaks down in terms of its size and/or power behavior. However, for a given hypothesis $(R,r)$ there exist elements of $\mathfrak{X}_0 \subseteq \Rel^{n \times k}$ to which the theorem is not applicable. As a consequence, the question remains to `how many' elements of $\mathfrak{X}_0$ the theorem can be applied once $(R,r)$ has been fixed. This question is studied subsequently. It is shown that generically in the space of all design matrices at least one of the four conditions of Theorem \ref{thmlrvPW} applies. The first part of the proposition establishes this genericity result in the class of all design matrices of full column rank, i.e., $\mathfrak{X}_0$. The remaining parts establish the genericity result in case $k \geq 2$ and the first column of $X$ is the intercept, i.e., $X = (e_+, \tilde{X})$ with $\tilde{X} \in \tilde{\mathfrak{X}}_0$. Before we state the proposition, we introduce two assumptions on the kernel $\kappa$. The first assumption is satisfied by all kernels typically used in practice.
\begin{assumption} \label{CD}
The kernel $\kappa$ is continuously differentiable on the complement of $\Delta(\kappa) \subseteq \Rel$, a set consisting of finitely many elements.
\end{assumption}
The second assumption, which is used in some statements of the second part of the genericity result, imposes compactness of the support of the kernel. This is satisfied by many kernels used in practice, e.g., the Bartlett kernel or the Parzen kernel, but is not satisfied by the Quadratic-Spectral kernel.
\begin{assumption} \label{CDL}
The support of $\kappa$ is compact.
\end{assumption}
The genericity result is now as follows, where several quantities are equipped with the additional subindex $X$ to stress their dependence on the design matrix.
\begin{proposition}\label{generic}
Fix a hypothesis $(R,r)$ such that $\rank(R) = q$. Let $\kappa$, $M$, $p$ satisfy Assumption \ref{weightsPWRB}. For $X \in \mathfrak{X}_0$ let $T_X$ be the test statistic defined in \eqref{tslrv} with $\hat{\Omega} = \hat{\Omega}_{\kappa, M, p, X}$ and let $\mu _{0, X}^{\ast } \in \mathfrak{M}_{0, X} = \geschw{\mu \in \lspan(X): \mu = X\beta, R\beta = r}$ be arbitrary (the sets defined below do not depend on the  choice of $\mu _{0, X}^{\ast }$). Fix a critical value $C$ such that $0 < C < \infty$. Then, the following holds.
\begin{enumerate}
\item Suppose that $k(p+1) + p + \mathbf{1}_{\mathbb{M}_{AM}}(M) \leq n$, define 
\begin{align}
\mathfrak{X}_{1}\left( e_{+}\right) &= \geschw{X\in \mathfrak{X}_0: g_{\kappa, M, p}^*(e_{+}, X, R) = 0} \\ 
\mathfrak{X}_{2}\left( e_{+} \right) &= \geschw{ X\in \mathfrak{X}_{0}\backslash 
\mathfrak{X}_{1}\left( e_{+}\right) :  \nexists ~ (\grad T_X(.))|_{e_{+}+\mu _{0, X}^{\ast }} \text{ and } T_X(e_{+}+\mu _{0, X}^{\ast }) = C},
\end{align}
and similarly define $\mathfrak{X}_{1}\left( e_{-}\right) $ and $\mathfrak{X}_{2}\left( e_{-}\right)$. Then, $\mathfrak{X}_{1}\left( e_{+}\right)$ and  $\mathfrak{X}_{1}\left( e_{-}\right) $ are $\lambda _{\mathbb{R}^{n\times k}}$-null sets. If $M \in \mathbb{M}_{KV}$ or if $\kappa$ satisfies Assumption \ref{CD}, then $\mathfrak{X}_{2}\left( e_{+}\right)$ and $\mathfrak{X}_{2}\left( e_{-}\right)$ are $\lambda_{\mathbb{R}^{n\times k}}$-null sets. If Assumption \ref{AAR(1)} holds, then the set of all design matrices $X\in \mathfrak{X}_{0}$ for which the first three parts of Theorem \ref{thmlrvPW} do not apply is a subset of $\left( \mathfrak{X}_{1}\left( e_{+}\right) \cup \mathfrak{X}_{2}\left( e_{+}\right) \right)\cap \left( \mathfrak{X}_{1}\left( e_{-}\right) \cup \mathfrak{X}_{2}\left(e_{-}\right) \right) $ and hence is a $\lambda _{\mathbb{R}^{n\times k}}$-null set if $M \in \mathbb{M}_{KV}$ or if $\kappa$ satisfies Assumption \ref{CD}; it thus is a `negligible' subset of $\mathfrak{X}_{0}$ in view of the fact that $\mathfrak{X}_{0}$ differs from $\mathbb{R}^{n\times k}$ only by a $\lambda _{\mathbb{R}^{n\times k}}$-null set. 
\item Let $k \geq 2$ and assume further that $k(p+1) + p^* + \mathbf{1}_{\mathbb{M}_{AM}}(M) \leq n$, where $p^* = p + (p\bmod2)$. Define
\begin{align}
\mathfrak{\tilde{X}}_{1}\left( e_{-}\right) &=\left\{ \tilde{X}\in \mathfrak{\tilde{X}}_{0}: g_{\kappa, M, p}^*(e_{-}, (e_{+}, \tilde{X}), R) = 0 \right\} , \\
\mathfrak{\tilde{X}}_{2}\left( e_{-}\right) &= \left\{ 
\begin{matrix}
\tilde{X} \in 
\mathfrak{\tilde{X}}_{0}\backslash \mathfrak{\tilde{X}}_{1}\left(
e_{-}\right) : & \nexists ~ (\grad T_{(e_{+}, \tilde{X})}(.))|_{e_{-}+\mu _{0, (e_+, \tilde{X})}^{\ast }} \\
& \text{ and } T_{(e_{+}, \tilde{X})}(e_{-}+\mu _{0, (e_+, \tilde{X})}^{\ast }) = C 
\end{matrix}
\right\}.
\end{align}
Then, $\mathfrak{\tilde{X}}_{1}\left(e_{-}\right) $ is a $\lambda _{\mathbb{R}^{n\times \left( k-1\right) }}$-null set. Furthermore, $\mathfrak{\tilde{X}}_{2}\left( e_{-}\right)$ is a $\lambda _{\mathbb{R}^{n\times \left( k-1\right) }}$-null set under each of the following conditions:
\begin{enumerate}
	\item $M \in \mathbb{M}_{KV}$.
	\item $M \in \mathbb{M}_{AM}$ and $\kappa$ satisfies Assumptions \ref{CD} and \ref{CDL}.	
	\item $M \in \mathbb{M}_{NW}$, $p$ is odd, $\omega_i > 0$ for some $i > 1$ and $\kappa$ satisfies Assumptions \ref{CD} and \ref{CDL}.
	\item $\kappa$ satisfies Assumption \ref{CD} and $\tilde{X} \mapsto T_{(e_+, \tilde{X})}(e_- + \mu _{0, (e_+, \tilde{X})}^{\ast }) \not \equiv C$ on $\mathfrak{\tilde{X}}_{0}\backslash \mathfrak{\tilde{X}}_{1}\left(e_{-}\right)$.
\end{enumerate}
Suppose that the first column of $R$ consists of zeros and that Assumption \ref{AAR(1)} holds. Then, the set of all matrices $\tilde{X} \in \mathfrak{\tilde{X}}_{0}$ such that the first three parts of Theorem \ref{thmlrvPW} do not apply to the design matrix $X=( e_{+},\tilde{X}) $ is a subset of $\mathfrak{\tilde{X}}_{1}\left( e_{-}\right) \cup \mathfrak{\tilde{X}}_{2}\left( e_{-}\right)$ and hence is a $\lambda_{\mathbb{R}^{n\times\left( k-1\right) }}$-null set if one of the conditions in (a)-(d) holds; it thus is a `negligible' subset of $\mathfrak{\tilde{X}}_{0}$ in view of the fact that $\mathfrak{\tilde{X}}_{0}$ differs from $\mathbb{R}^{n\times \left( k-1\right) }$ only by a $\lambda_{\mathbb{R}^{n\times \left( k-1\right) }}$-null set. 
\item Suppose $k\geq 2$, that the first column of $R$ is nonzero and that Assumption \ref{AAR(1)} holds. Then Theorem \ref{thmlrvPW} (Part 4) applies to the design matrix $X=\left( e_{+},\tilde{X}\right) $ for every $\tilde{X}\in 
\mathfrak{\tilde{X}}_{0}$ satisfying $g_{\kappa, M, p}^*(., X, R) \not \equiv 0$. 
\end{enumerate}
\end{proposition}

\begin{remark}\label{Rgeneric}
(i) If $n < k(p+1) + p$ and $q = k$ holds, the first part of Lemma \ref{nullN*PW} shows that the test trivially breaks down, since for every element $X$ of $\mathfrak{X}_0$ the test statistic $T_X$ is then constant on $\Rel^n$. Therefore, the assumption on $n$ in the first two parts of the proposition can in general not be substantially improved. 

(ii) In the second part of the proposition, the analogously defined sets $\mathfrak{\tilde{X}}_{1}\left( e_{+}\right) $ and $\mathfrak{\tilde{X}}_{2}\left( e_{+}\right) $ clearly satisfy $\mathfrak{\tilde{X}}_{1}\left( e_{+}\right) =\mathfrak{\tilde{X}}_{0}$ and $\mathfrak{\tilde{X}}_{2}\left( e_{+}\right) =\emptyset $. 

(iii) In the third part of the proposition, if $X = (e_+, \tilde{X})$ does not satisfy $g_{\kappa, M, p}^*(., X, R) \not \equiv 0$, then the test breaks down in a trivial way, since $T_X$ is then constant.
\end{remark}

The first part of the preceding genericity result shows that if $M \in \mathbb{M}_{KV}$, or if the kernel satisfies Assumption \ref{CD},  then Theorem \ref{thmlrvPW} can be applied to generic elements of $\mathfrak{X}_0$, i.e., to all elements besides a $\lambda_{\Rel^{n \times k}}$-null set. Since all kernels used in practice, in particular the kernels emphasized in \cite{A92} and \cite{NW94}, i.e., the Quadratic-Spectral kernel and the Bartlett kernel, respectively, satisfy Assumption \ref{CD}, this additional restriction on $\kappa$ is practically immaterial. The second part of the proposition considers the situation where the first column of the design matrix is the intercept, which in addition is assumed not to be involved in the hypothesis in the sense that the first column of $R$ is zero. In this situation it is shown that Theorem \ref{thmlrvPW} can generically be applied to design matrices of the form $(e_+, \tilde{X})$ with $\tilde{X} \in \tilde{\mathfrak{X}}_0$, under certain sets of conditions on the triple $\kappa$, $M$, $p$. We first discuss Conditions (a)-(c):
\begin{enumerate}[label=(\alph*)]
\item In case $M \in \mathbb{M}_{KV}$ no additional condition is needed for establishing generic applicability of Theorem \ref{thmlrvPW}. 
\item If $M \in \mathbb{M}_{AM}$, generic applicability of the negative result follows if the kernel satisfies Assumptions \ref{CD} and \ref{CDL}, which applies to many kernels used in practice, but not to the Quadratic-Spectral kernel which is emphasized in \cite{A92}. 
\item In case $M \in \mathbb{M}_{NW}$, the result shows that the procedure breaks down generically if $p$ is odd, $\omega_i > 0$ for some $i > 1$ and $\kappa$ satisfies Assumptions \ref{CD} and \ref{CDL}. This seems to be restrictive. However, the recommended procedure in \cite{NW94} is obtained by choosing $\kappa$ the Bartlett kernel, $p = 1$ and $\omega = (0, 1, \hdots, 1)'$, because in Part 2 the first column of $X$ is the intercept. Therefore, we see that the recommended procedure in \cite{NW94} satisfies this condition. 
\end{enumerate}
Summarizing, we see that the Conditions (a)-(c) in the proposition cover the recommended choices of $\kappa$, $M$ and $p$ in \cite{NW94} and \cite{Rho2013}. For all procedures that are not covered by Conditions (a)-(c), e.g., the procedure in \cite{A92} based on the Quadratic-Spectral kernel, one can typically obtain the genericity result by applying Condition (d), which (under Assumption \ref{CD}) is always satisfied apart from at most one exceptional critical value $C^*$. This is seen as follows: Clearly, Condition (d) depends on the critical value $C$. We see that if Assumption \ref{CD} is satisfied, then (d) can be violated for at most a single $0 < C^* < \infty$. If this $C^*$ happens to coincide with $C$, the condition is not satisfied and we can not draw the desired conclusion for this specific value of $C$. Moreover, we immediately see that the condition must then be satisfied for any other choice $C'$, say. Therefore, generic applicability of the negative result follows for any value $C' \neq C$ in that case. This shows that even if one chooses a triple $\kappa$, $M$, $p$ that does not allow for an application of (a)-(c), one can not expect to obtain a procedure that has good finite sample size and power properties, because for all but at most one exceptional critical value the corresponding test is guaranteed to break down generically. The third part of the proposition considers the case where the first column of the design matrix is the intercept, and where the coefficient corresponding to the intercept is restricted by the hypothesis. In this case it follows that one can either apply Part 4 of Theorem \ref{thmlrvPW}, or the test statistic is constant and hence the test breaks down in a trivial way (cf. Remark \ref{Rgeneric}).

\section{A positive result, an adjustment procedure and its generic applicability}\label{pos}

In the previous section we have established a (generically applicable) negative result concerning tests as in \eqref{tslrv} based on a prewhitened covariance estimator $\hat{\Omega}_{\kappa, M, p}$. In the present section we first present a positive result concerning these tests under a \textit{non-generic} condition on the design matrix. Then we introduce an adjustment procedure and establish a condition on the design matrix under which the adjustment procedure leads to improved tests. Finally we prove that this condition holds generically in the set of all design matrices. Both, the positive result concerning tests as in \eqref{tslrv} based on a prewhitened covariance estimator $\hat{\Omega}_{\kappa, M, p}$, and the results concerning the adjustment procedure are established under the following assumption on the covariance model $\mathfrak{C}$. 
\begin{assumption} \label{approxAAR(1)}
The set $\mathfrak{C} \subseteq \Rel^{n \times n}$ is norm-bounded and satisfies $\mathfrak{C}_{AR(1)} \subseteq \mathfrak{C}$. Furthermore, for every sequence $\Sigma_m \in \mathfrak{C}$ that converges to $\bar{\Sigma} \in \boundary(\mathfrak{C})$ satisfying $\rank(\bar{\Sigma}) < n$ there exists a corresponding sequence $\rho_m \in (-1, 1)$ such that $\Lambda(\rho_m)^{-1/2} \Sigma_m \Lambda(\rho_m)^{-1/2} \rightarrow I_n$ as $m \rightarrow \infty$.
\end{assumption}
\begin{remark}\label{RapproxAAR(1)}
(i) We first note that Assumption \ref{approxAAR(1)} is stronger than Assumption \ref{AAR(1)}. Therefore, under the former assumption the negative result established in Section \ref{neg} concerning tests as in \eqref{tslrv} based on a prewhitened covariance estimator $\hat{\Omega}_{\kappa, M, p}$ \textit{does} apply a fortiori. As a consequence, if $\mathfrak{C}$ satisfies Assumption \ref{approxAAR(1)}, then positive results concerning size and power properties of tests of the form \eqref{tslrv} can only be established under non-generic assumptions on the design matrix. However, as we shall show, positive results can generically be established for an \textit{adjusted} version of such tests.

(ii) Boundedness of $\mathfrak{C}$ is typically satisfied in our setup, as it is always satisfied if $\mathfrak{C}$ consists only of correlation matrices. 

(iii) The last part of the assumption states that elements of $\mathfrak{C}$ that are `close' to being singular can be well approximated by AR(1) correlation matrices. This, together with $\mathfrak{C}_{AR(1)}$ being a subset of $\mathfrak{C}$, readily implies that the singular boundary of $\mathfrak{C}$ must coincide with $\geschw{e_+ e_+', e_- e_-'}$. Therefore, we see that the assumption rules out the existence of rank deficient elements of $\boundary(\mathfrak{C})$ with rank strictly greater than one. As an example, this rules out the case where $\mathfrak{C}$ is the correlation model corresponding to all stationary autoregressive processes of order less than or equal to two (cf. Lemma G.2 in \cite{PP13}). If this is not ruled out, however, further obstructions to good size and power properties can arise along suitable sequences approximating these boundary points (cf. Section 3.2.3 in \cite{PP13}). The possibility of establishing positive results in settings like that is beyond the scope of the present paper and will be discussed elsewhere. 

(iv) We note that Assumption \ref{approxAAR(1)} is clearly satisfied for every covariance model of the form $\mathfrak{C} = \mathfrak{C}_{AR(1)} \cup \mathfrak{C}^{\sharp}$, where $\mathfrak{C}^{\sharp} \subseteq \Rel^{n \times n}$ is a closed set consisting of positive definite correlation matrices. As an example, let $d \in \Nat$ be fixed and let $\mathfrak{C}_{MA(d)}$ denote the set of all correlation matrices corresponding to stationary moving average processes of an order not exceeding $d$, i.e., 
\begin{equation}
\mathfrak{C}_{MA(d)} = \geschw{\Sigma(f_{\alpha, \delta}): \alpha = (1, \alpha_1, \hdots, \alpha_{d})' \in \Rel^{d+1}, \delta > 0},
\end{equation}
where $\Sigma(f_{\alpha, \delta})$ denotes the $n \times n$-dimensional correlation matrix corresponding to the spectral density $f_{\alpha, \delta}(\lambda) = \frac{\delta^2}{2\pi}|\sum_{j = 0}^{d} \alpha_j \exp(-\iota \lambda j)|^2$ (cf. Equation \eqref{covmat}). Then $\mathfrak{C} = \mathfrak{C}_{AR(1)} \cup \closure(\mathfrak{C}_{MA(d)})$ satisfies Assumption \ref{approxAAR(1)}, because every element of the closure of $\mathfrak{C}_{MA(d)}$ is a positive definite correlation matrix (the latter statement follows from Equation \eqref{covmat}, compactness of the unit sphere in $\Rel^{d+1}$, and the Dominated Convergence Theorem).
\end{remark}
Under Assumption \ref{approxAAR(1)} we shall subsequently establish a positive result concerning tests based on a test statistic $T$ as in \eqref{tslrv} with $\hat{\Omega} = \hat{\Omega}_{\kappa, M, p}$. In light of Part (i) of the preceding remark we already know that such a positive result can only be established under non-generic conditions on the design matrix. In particular, the subsequent positive result considers the non-generic case where - besides $g^*_{\kappa, M, p}(., X, R)  \not \equiv 0$, a condition that is generically satisfied under a mild constraint on $n$ (cf. Lemma \ref{nullN*PW}) - the column span of the design matrix includes the vectors $e_+$ and $e_-$ and where $R\hat{\beta}(e_+) = R\hat{\beta}(e_-) = 0$ holds.  
\begin{proposition}\label{excPW} 
Suppose that the triple $\kappa$, $M$, $p$ satisfies Assumption \ref{weightsPWRB}, and that $\mathfrak{C}$ satisfies Assumption \ref{approxAAR(1)}. Let $T$ be the test statistic defined in Equation \eqref{tslrv} with $\hat{\Omega} = \hat{\Omega}_{\kappa, M, p}$. Let $W(C)=\geschw{y\in \mathbb{R}^{n}:T(y)\geq C}$ be the rejection region, where $C$ is a real number satisfying $0<C<\infty$. Suppose further that $e_{+}, e_{-} \in \mathfrak{M}$, $R\hat{\beta}(e_{+})=R\hat{\beta}(e_{-})=0$ and $g^*_{\kappa, M, p}(., X, R)  \not \equiv 0$. Then, the following holds:
\begin{enumerate}
\item The size of the rejection region $W(C)$ is strictly less than $1$, i.e.,
\begin{equation*}
\sup\limits_{\mu _{0}\in \mathfrak{M}_{0}}\sup\limits_{0<\sigma ^{2}<\infty}\sup\limits_{\Sigma \in \mathfrak{C}}P_{\mu_{0},\sigma ^{2}\Sigma }\left(W(C)\right) <1.
\end{equation*}
Furthermore,
\begin{equation*}
\inf_{\mu _{0}\in \mathfrak{M}_{0}}\inf_{0<\sigma ^{2}<\infty }\inf_{\Sigma\in \mathfrak{C}} P_{\mu _{0},\sigma ^{2}\Sigma }\left( W(C)\right) >0.
\end{equation*}
\item The infimal power is bounded away from zero, i.e., 
\begin{equation*}
\inf_{\mu _{1}\in \mathfrak{M}_{1}}\inf\limits_{0<\sigma ^{2}<\infty}\inf\limits_{\Sigma \in \mathfrak{C}}P_{\mu _{1},\sigma ^{2}\Sigma}(W(C))>0.
\end{equation*}
\item For every $0<c<\infty $
\begin{equation*}
\inf_{\substack{ \mu _{1}\in \mathfrak{M}_{1},0<\sigma ^{2}<\infty  \\ 
d\left( \mu _{1},\mathfrak{M}_{0}\right) /\sigma \geq c}}P_{\mu _{1},\sigma ^{2}\Sigma _{m}}(W(C))\rightarrow 1
\end{equation*}
holds for $m\rightarrow \infty $ and for any sequence $\Sigma _{m}\in\mathfrak{C}$ satisfying $\Sigma _{m}\rightarrow \bar{\Sigma}$ with $\bar{\Sigma}$ a singular matrix. Furthermore, for every sequence $0<c_{m}<\infty $
\begin{equation*}
\inf_{\substack{ \mu _{1}\in \mathfrak{M}_{1},  \\ d\left( \mu _{1},\mathfrak{M}_{0}\right) \geq c_{m}}} \inf_{\Sigma \in \mathfrak{C}^*} P_{\mu _{1},\sigma _{m}^{2}\Sigma
}(W(C))\rightarrow 1
\end{equation*}
holds for $m\rightarrow \infty $ whenever $0<\sigma _{m}^{2}<\infty $, $c_{m}/\sigma _{m}\rightarrow \infty $, and $\mathfrak{C}^*$ is a closed subset of $\mathfrak{C}$. [The very last statement even holds if one of the conditions $e_+, e_- \in \mathfrak{M}$ and $R\hat{\beta}(e_+) = R\hat{\beta}(e_-) = 0$ is violated.]
\item For every $\delta $, $0<\delta <1$, there exists a $C(\delta )$, $0<C(\delta )<\infty $, such that
\begin{equation*}
\sup\limits_{\mu _{0}\in \mathfrak{M}_{0}}\sup\limits_{0<\sigma ^{2}<\infty}\sup\limits_{\Sigma \in \mathfrak{C}}P_{\mu _{0},\sigma ^{2}\Sigma}(W(C(\delta )))\leq \delta.
\end{equation*}
\end{enumerate}
\end{proposition}
Under the maintained assumptions on the hypothesis and the design Proposition \ref{excPW} shows that given any level of significance $0 < \delta < 1$, a critical value can be chosen in such a way that the test obtained holds its size, while its nuisance-minimal power at every point $\mu_1$ in the alternative is bounded away from zero. As Theorem \ref{thmlrvPW} in combination with Proposition \ref{generic} shows, this is impossible for generic elements of the space of all design matrices. Additionally, Part 3 of the proposition shows that the power approaches one in certain parts of the parameter space corresponding to the alternative hypothesis. These parts are characterized by $\norm{(R \beta^{(1)} - r)/\sigma}$ being bounded away from zero and $\Sigma \rightarrow \bar{\Sigma}$ with $\bar{\Sigma}$ being singular, or $\norm{(R \beta^{(1)} - r)/\sigma} \rightarrow \infty$ and $\Sigma \rightarrow \bar{\Sigma}$ with $\bar{\Sigma}$ positive definite, and where in both cases $\beta^{(1)}$ is the parameter vector corresponding to $\mu_1$ (note that $d(\mu_1, \mathfrak{M}_0)$ is bounded from above and below by multiples of $\norm{R\beta^{(1)} -r}$, where the constants involved are positive and depend only on $X$, $R$ and $r$).

\begin{remark}\label{Rposbound}
Suppose that instead of Assumption \ref{approxAAR(1)} it is known that the covariance model satisfies the following variant of Assumption \ref{approxAAR(1)} that rules out AR(1) correlation matrices $\Lambda(\rho)$ with $\rho$ arbitrarily close to $-1$: 
\begin{quote}
The covariance model $\mathfrak{C} \subseteq \Rel^{n \times n}$ is norm-bounded and there exists an $\varepsilon \in (0, 1]$ such that $\mathfrak{C}_{AR(1)}(\varepsilon) \subseteq \mathfrak{C}$ (cf. Remark \ref{boundaway}). Furthermore, for every sequence $\Sigma_m \in \mathfrak{C}$ that converges to $\bar{\Sigma} \in \boundary(\mathfrak{C})$ satisfying $\rank(\bar{\Sigma}) < n$ there exists a corresponding sequence $\rho_m \in (-1+\varepsilon, 1)$ such that $\Lambda(\rho_m)^{-1/2} \Sigma_m \Lambda(\rho_m)^{-1/2} \rightarrow I_n$ as $m \rightarrow \infty$.
\end{quote}
Let $T$ and $W(C)$ be defined as in Proposition \ref{excPW} above, and suppose further that $e_+ \in \mathfrak{M}$, i.e., the regression contains an intercept, that $R\hat{\beta}(e_+) = 0$, i.e., the hypothesis does not involve the intercept, and that $g^*_{\kappa, M, p}(., X, R) \not \equiv 0$. Then one can show (using essentially the same argument as in the proof of Proposition \ref{excPW}) that the Conclusions 1-4 of Proposition \ref{excPW} hold in this setup. In particular, for any given $\delta \in (0, 1)$ there exists a critical value $C(\delta) = C(\delta, \varepsilon)$ such that the test with rejection region $W(C(\delta, \varepsilon))$ has size not greater than $\delta$. This establishes a \textit{positive result} concerning a test based on the test statistic $T$ with $\hat{\Omega} = \hat{\Omega}_{\kappa, M, p}$, and based on the non-standard critical value $C(\delta, \varepsilon)$, which in practice can be obtained as explained in the discussion following Theorem \ref{TU_3}. However, the test obtained critically depends on $\varepsilon$, which in practice will typically be difficult to choose. If one uses a test with critical region $W(C(\delta, \varepsilon^*))$ where $\varepsilon^* > \varepsilon$, then the size of this test might exceed $\delta$. In light of this drawback, it is important to stress that autocorrelation-robust-testing is possible \textit{without} imposing such an artificial condition that rules out AR(1) correlation matrices $\Lambda(\rho)$ with $\rho$ arbitrarily close to $-1$: Theorems \ref{TU_3} and \ref{genericADJ} in the present section show that (at the small cost of including the artificial regressor $e_-$) a generic positive result can be obtained under the more natural Assumption \ref{approxAAR(1)}. Theorem \ref{TU_3} furthermore shows that including the artificial regressor $e_-$ leads to good power properties of the resulting test for covariance matrices close to $\Lambda(-1)$.
\end{remark}

\smallskip

Proposition \ref{excPW} assumes, among others, that $X$ satisfies $e_+, e_- \in \lspan(X)$, an assumption which is satisfied only by non-generic elements of the set of all design matrices. In the following we shall now consider the (generic) situation where $e_+, e_- \in \lspan(X)$ is \textit{violated}. Proposition \ref{excPW} is then clearly not applicable. However, as we shall see, a positive result similar to Proposition \ref{excPW} can be established for an \textit{adjusted} test statistic. To explain how the adjustment procedure works, suppose that we have a triple $\kappa$, $M$, $p$ satisfying Assumption \ref{weightsPWRB} which we want to use for covariance estimation. Suppose further that $1 \leq p \leq \frac{n}{k+3}$ holds, which is typically satisfied. The following theorem now shows that under certain conditions on the design matrix - which are shown to be generically satisfied in Proposition \ref{genericADJ} below, and which in particular require $e_+, e_- \in \lspan(X)$ to be violated - one can work with an \textit{adjusted} test statistic that has improved size and power properties, and which is constructed as follows: instead of basing the construction of the test statistic on the true design matrix $X$ and on $R$, we first construct an artificial design matrix $\bar{X}$ of full column rank satisfying $\lspan(\bar{X}) = \lspan(X, e_+, e_-)$ by adding the vectors $e_+$ and/or $e_-$ to $X$. Furthermore, we construct a corresponding matrix $\bar{R}$, where zero columns are added to $R$ such that $\bar{R}$ and $\bar{X}$ have the same number of columns. Then we construct a test statistic $\bar{T}$ as in Equation \eqref{tslrv}, but with $X$ and $R$ replaced by $\bar{X}$ and $\bar{R}$, respectively, and where the covariance estimator is based on $\kappa$, $M$ and $p$ as above besides some minor updates in the construction of $M$ described below. The subsequent theorem shows that if $e_+, e_- \in \lspan(X)$ is violated, if every $e \in \geschw{e_+, e_-} \cap \lspan(X)$ satisfies $R\hat{\beta}(e) = 0$, if $\rank(\bar{X})< n$, and if an assumption on $\bar{X}$ and $\bar{R}$ analogous to the assumption $g^*_{\kappa, M, p}(., X, R)  \not \equiv 0$ in Proposition \ref{excPW} is satisfied, then for every critical value $0 < C < \infty$ the test with critical region $\bar{W}(C) = \geschw{y \in \Rel^n : \bar{T}(y) \geq C}$ has the same Properties (1)-(4) as $W(C)$ in Proposition \ref{excPW}. 
\begin{theorem}\label{TU_3}
Suppose that the triple $\kappa$, $M$, $p$ satisfies Assumption \ref{weightsPWRB}, that $p$ additionally satisfies $1 \leq p \leq \frac{n}{k+3}$, and that $\mathfrak{C}$ satisfies Assumption \ref{approxAAR(1)}. Suppose that one of the following (mutually exclusive) scenarios applies:
\begin{enumerate}
\item $e_{+}\in \mathfrak{M}$ with $R\hat{\beta}_X(e_{+})=0$, $e_{-}\notin 
\mathfrak{M}$ and $\bar{k} = k + 1 < n$. Let $\bar{X}=\left( X,e_{-}\right) $ and define $\bar{R}=\left( R,0\right) \in \Rel^{q \times \bar{k}}$.
\item $e_{+}\notin \mathfrak{M}$,  $e_{-}\in \mathfrak{M}$ with $R\hat{%
\beta}_X(e_{-})=0$ and $\bar{k} = k + 1 < n$. Let $\bar{X}=\left( X,e_{+}\right) $ and define $\bar{R}=\left( R,0\right) \in \Rel^{q \times \bar{k}}$.
\item $e_{+}\notin \mathfrak{M}$, $e_{-}\notin \mathfrak{M}$ with $%
\rank\left( X,e_{+},e_{-}\right) =k+2$ and $\bar{k} = k + 2 < n$. Let $\bar{X}=\left(
X,e_{+},e_{-}\right) $ and define $\bar{R}=\left( R,0,0\right) \in \Rel^{q \times\bar{k}}$.
\item $e_{+}\notin \mathfrak{M}$, $e_{-}\notin \mathfrak{M}$ with $%
\rank\left( X,e_{+},e_{-}\right) =k+1$ and $\bar{k} = k + 1 < n$. Let $\bar{X}=\left( X,e_{+}\right) $ and define $\bar{R}=\left( R,0\right) \in \Rel^{q \times \bar{k}}$.
%
\end{enumerate}
Then in all cases $\bar{X}$ is a matrix of full column rank. Define
\begin{equation}
\bar{T}(y)= \begin{cases}
(\bar{R}\hat{\beta}_{\bar{X}}(y)-r)^{\prime }\hat{\Omega}_{\kappa, \bar{M}, p, \bar{X}}^{-1}(y)(\bar{R}\hat{\beta}_{\bar{X}}(y)-r) & 
\text{if } y \notin N^*(\hat{\Omega}_{\kappa, \bar{M}, p, \bar{X}}), \\[7pt]
0 & \text{else},%
\end{cases}%
\end{equation}
where $\hat{\Omega}_{\kappa, \bar{M}, p, \bar{X}} \left( y\right)$ is the estimator one would obtain following Steps 1-3 in Section \ref{tpwc} if $\bar{X}$ was the underlying design matrix, $(\bar{R}, r)$ was the hypothesis to be tested and where $\bar{M}$ is defined as follows: in case $M \in \mathbb{M}_{KV}$ we set $\bar{M} \equiv M$; in case $M \in \mathbb{M}_{AM}$ we compute $\bar{M}$ as outlined in Section \ref{MA92} (using as input $\hat{Z}_{\bar{X}}(y)$ as opposed to $\hat{Z}_{X}(y)$, and replacing $k$ by $\bar{k}$), with the same constants $c_1$ and $c_2$ and $j$ as used in the construction of $M$, but with $\omega$ replaced by $\bar{\omega} = (\omega, 0)' \in \Rel^{\bar{k}}$; in case $M \in \mathbb{M}_{NW}$ we compute $\bar{M}$ as outlined in Section \ref{MNW94} (using as input $\hat{Z}_{\bar{X}}(y)$ as opposed to $\hat{Z}_{X}(y)$), with the same constants $\bar{c}_i$ for $i = 1, 2, 3$ and the same weights $w$ as used in the construction of $M$, but with $\omega$ replaced by $\bar{\omega} = (\omega, 0)' \in \Rel^{\bar{k}}$. Let $\bar{W}(C)=\left\{ y\in \mathbb{R}^{n}:\bar{T}(y)\geq C\right\} $ be the rejection region where $C$ is a real number satisfying $0<C<\infty $. If $g_{\kappa, \bar{M}, p}^*(. , \bar{X}, \bar{R}) \not \equiv 0$ (where $g_{\kappa, \bar{M}, p}^*$ is the function obtained from Lemma \ref{N*PW} applied to $\kappa$, $\bar{M}$ and $p$ and acting as if $\bar{X}$ was the underlying design matrix and $(\bar{R}, r)$ was the hypothesis to be tested), or equivalently if $ N^*(\hat{\Omega}_{\kappa, \bar{M}, p, \bar{X}}) \neq \Rel^n$, then in each of the four scenarios above the Conclusions (1)-(4) of Proposition \ref{excPW} hold with $W(C)$ replaced by $\bar{W}(C)$.
\end{theorem}
The procedure outlined in the preceding theorem is based on an artificial design matrix $\bar{X}$ which is obtained from $X$ by adding either one or both elements of the set $\geschw{e_+, e_-}$ to the columns of $X$. If the so-obtained matrix $\bar{X}$ satisfies $g^*_{\kappa, \bar{M}, p}(., \bar{X}, \bar{R}) \not \equiv 0$, then the results from Proposition \ref{excPW} carry over to the rejection regions derived from $\bar{T}$. In particular the adjusted test statistic $\bar{T}$ leads to rejection regions the size of which is bounded away from one and such that the nuisance minimal power is bounded away from zero. Besides these improvements, the adjustment procedure is extremely convenient from a computational perspective, as the adjusted test statistic $\bar{T}$ does not require any additional implementations. It is based on the same algorithm as the calculation of $T$, only with a different design matrix. 

The theorem shows that for every level of significance $0 < \delta < 1$, there exists a critical value $C(\delta)$ such that the rejection region $\bar{W}(C(\delta))$ has size smaller than $\delta$. The critical value can be determined as follows: First of all, due to certain invariance properties of $\bar{T}$ (cf. the proof of Theorem \ref{TU_3}), the probabilities $P_{\mu_0, \sigma^2 \Sigma}(\bar{W}(C))$ do not depend on $\mu_0$ and $\sigma^2$. Hence, for any fixed $0 < C < \infty$, the maximal rejection probability under the null can be approximated numerically by simulating the rejection probabilities from a finite subset of $\mathfrak{C}$, and then doing a grid search. In a second step $C(\delta)$ can be approximated by a line search exploiting monotonicity of $P_{\mu_0, \sigma^2 \Sigma}(\bar{W}(C))$ in the critical value. 

\smallskip

The adjustment procedure described in Theorem \ref{TU_3} is applicable and yields an improved test under the assumption that $e_+, e_- \in \lspan(X)$ is violated (and hence the positive result in Proposition \ref{excPW} concerning the \textit{unadjusted} test does not apply), that every $e \in \geschw{e_+, e_-} \cap \lspan(X)$ satisfies $R\hat{\beta}(e) = 0$, that $\bar{k} < n$ and that $g^*_{\kappa, \bar{M}, p}(., \bar{X}, \bar{R}) \not \equiv 0$. Given a hypothesis $(R, r)$ these are conditions on the design matrix $X$. Our final result now shows (under mild constraints on $n$) that these conditions are generically satisfied in the set of all design matrices $\mathfrak{X}_0$; and also in $\tilde{\mathfrak{X}}_0$, the set of all design matrices the first column of which is the intercept, under the additional condition that the first column of $R$ is zero. Under Assumption \ref{approxAAR(1)} we hence see that although rejection regions based on $T$ generically break down as a consequence of Proposition \ref{generic}, this problem can generically be resolved by using rejection regions based on the adjusted test statistic $\bar{T}$, unless the regression includes an intercept and the first column of $R$ is nonzero.
\begin{proposition}\label{genericADJ}
Fix a hypothesis $(R,r)$ with $\rank(R) = q$, suppose that the triple $\kappa$, $M$, $p$ satisfies Assumption \ref{weightsPWRB}, that $p$ additionally satisfies $1 \leq p \leq \frac{n}{k+3}$ and that $\mathfrak{C}$ satisfies Assumption \ref{approxAAR(1)}. Then the following holds, where $p^* = p + (p\bmod2)$.
\begin{enumerate}
\item If $(k+3)(p^* + 2) + p-1 + \mathbf{1}_{\mathbb{M}_{AM}}(M) \leq n$, then for $\lambda_{\Rel^{n \times k}}$-almost every design matrix $X \in \mathfrak{X}_0 \subseteq \Rel^{n \times k}$ Scenario 3 in Theorem \ref{TU_3} applies, and the Conclusions (1)-(4) of Proposition \ref{excPW} hold for any critical value $0 < C < \infty$ with $W(C)$ replaced by $\bar{W}(C) = \geschw{y \in \Rel^n: \bar{T}(y) \geq C}$, where $\bar{T}$ is constructed as outlined in Theorem \ref{TU_3}.
\item Suppose that the first column of $R$ is zero, that $k \geq 2$ and assume that $(k+2)(p^* + 2) + p-1 + \mathbf{1}_{\mathbb{M}_{AM}}(M) \leq n$ holds. Then for $\lambda_{\Rel^{n \times (k-1)}}$-almost every $\tilde{X} \in \tilde{\mathfrak{X}}_0$ Scenario 1 of Theorem \ref{TU_3} applies to $X = (e_+, \tilde{X})$, and the Conclusions (1)-(4) of Proposition \ref{excPW} hold for any critical value $0 < C < \infty$ with $W(C)$ replaced by $\bar{W}(C) = \geschw{y \in \Rel^n: \bar{T}(y) \geq C}$, where $\bar{T}$ is constructed as outlined in Theorem \ref{TU_3}.
\end{enumerate}
\end{proposition}

\section{Conclusion}\label{concl}

We have shown that tests for \eqref{testing problem} based on prewhitened covariance estimators and possibly data-dependent bandwidth parameters break down in finite samples in terms of their size or power properties. This breakdown arises already for comparably simple covariance models such as $\mathfrak{C} = \mathfrak{C}_{AR(1)}$. We have also shown how a simple adjustment procedure can generically solve this problem in many cases. The test statistic obtained by applying the adjustment procedure is of the same structural form as the test statistic based on estimators suggested by \cite{A92} and \cite{NW94} and the test statistic in \cite{Rho2013}, but it is based on an artificial design matrix. Therefore, the adjustment procedure does not only lead to improved size and power properties, but is also convenient from a computational point of view. For the adjustment procedure to work, Assumption \ref{approxAAR(1)} has to be satisfied, which requires that elements of the covariance model $\mathfrak{C}$ that are close to being singular are well approximated by AR(1) correlation matrices. If and how the adjustment procedure can be extended to settings where this approximation condition is not satisfied is currently under investigation.

\newpage

\begin{large}
\begin{center}
\textsc{Appendices}
\end{center}
\end{large}

\bigskip

\appendix

\textbf{Additional notation:} For the sake of clarity we shall repeatedly stress the dependence of $\hat{V}$, $\hat{V}_1$, $\hat{V}_p$, $\hat{Z}$, $\hat{A}^{(p)}$, $\hat{u}$, $\hat{\Omega}_{\kappa, M, p}$ and $B_p$ on the design matrix $X$ by writing $\hat{V}_X$, $\hat{V}_{1, X}$, $\hat{V}_{p,X}$, $\hat{Z}_X$, $\hat{A}^{(p)}_X$, $\hat{u}_X$, $\hat{\Omega}_{\kappa, M, p, X}$ and $B_{p, X}$ in the following proofs. At various places we shall use the following notation: Given a matrix $M \in \Rel^{m_1 \times m_2}$ and indices $1\leq i \leq m_1$ and $1 \leq j \leq m_2$ we denote by $[M]_{ij} = M_{ij}$ the $ij$-th coordinate of $M$, by $[M]_{\cdot j} = M_{\cdot j}$ the $j-th$ column of $M$ and by $[M]_{i\cdot} = M_{i\cdot}$  the $i-th$ row of $M$. In case $m_2 = 1$ we write  $[M]_{i} = M_i$ instead of $[M]_{i1}$.

\section{Proofs of Results in Section \ref{struct}}\label{AA}

\begin{proof}[Proof of Lemma \ref{NPWRB}]
Since $X$ is a matrix of full column rank by assumption, we clearly have $\det(X'X) \neq 0$. From the definition of $\hat{\Omega}_{\kappa, M, p, X}$ we see that $y \in N(\hat{\Omega}_{\kappa, M, p, X})$, i.e., $\hat{\Omega}_{\kappa, M, p, X}(y)$ is not well defined, if and only if one of the following conditions is satisfied (cf. Remark \ref{WDPSI}): 
\begin{enumerate}
\item[(I)]  $\det\left(\hat{V}_{1, X}(y)\hat{V}'_{1, X}(y)\right) = 0$; 
\item[(II)] $\det\left(\hat{V}_{1, X}(y)\hat{V}'_{1, X}(y)\right) \neq 0$ and $\det\left(I_k-\sum_{l = 1}^p \hat{A}^{(p)}_{l, X}(y)\right) = 0$; 
\item[(III)] $\det\left(\hat{V}_{1, X}(y)\hat{V}'_{1, X}(y)\right) \neq 0$ and $\det\left(I_k-\sum_{l = 1}^p \hat{A}^{(p)}_{l, X}(y)\right) \neq 0$ and $M(y)$ is not well defined. 
\end{enumerate}
Using $\hat{u}_X(y) = (I- \det(X'X)^{-1} X' \adj(X'X) X') y$ we see that the coordinates of $\bar{V}_{1, X}(y) := \det(X'X) \hat{V}_{1, X}(y)$ and of $\bar{V}_{p, X}(y) := \det(X'X) \hat{V}_{p, X}(y)$ are values of certain multivariate polynomials defined on $\Rel^n \times \Rel^{n \times k}$ evaluated at the point $(y, X)$. Since (I) is equivalent to
\begin{equation}
\det(\det(X'X)^{2} \hat{V}_{1, X}(y)\hat{V}'_{1, X}(y)) = \det(\bar{V}_{1, X}(y)\bar{V}'_{1, X}(y)) = 0,
\end{equation}
this shows that (I) is equivalent to $g_1(y, X)=0$, say, where $g_1: \Rel^n \times \Rel^{n \times k} \rightarrow \Rel$ is a multivariate polynomial which is clearly independent of $(R,r)$. Using this equivalence, Condition (II) is seen to be equivalent to $g_1(y, X) \neq 0$ and $\det\left(I_k-\sum_{l = 1}^p \hat{A}^{(p)}_{l, X}(y)\right) = 0$. Because of $g_1(y, X) \neq 0$ we have
\begin{align}
I_k-\sum_{l = 1}^p \hat{A}^{(p)}_{l, X}(y) &= I_k - \hat{V}_{p, X}(y) \hat{V}'_{1, X}(y)\left(\hat{V}_{1, X}(y)\hat{V}'_{1, X}(y)\right)^{-1}  D(p), \\
&= I_k - \bar{V}_{p, X}(y) \bar{V}'_{1, X}(y)\left(\bar{V}_{1, X}(y)\bar{V}'_{1, X}(y)\right)^{-1}  D(p), \\
&= I_k - \det\left(\bar{V}_{1, X}(y)\bar{V}'_{1, X}(y)\right)^{-1} \bar{V}_{p, X}(y) \bar{V}'_{1, X}(y)\adj\left(\bar{V}_{1, X}(y)\bar{V}'_{1, X}(y)\right) D(p),
\end{align}
where $D(p) = (I_k, \hdots , I_k)' \in \Rel^{kp \times k}$. Using this together with similar arguments as above we see that pre-multiplying $I_k-\sum_{l = 1}^p \hat{A}^{(p)}_{l, X}(y)$ by $\det\left(\bar{V}_{1, X}(y)\bar{V}'_{1, X}(y)\right)$ results in a matrix, the entries of which are values of certain multivariate polynomials, defined on $\Rel^n \times \Rel^{n \times k}$, evaluated at the point $(y, X)$. It follows that the second equation in (II) can be replaced by 
\begin{equation}
g_{2}(y, X) := \left[\det\left(\bar{V}_{1, X}(y)\bar{V}'_{1, X}(y)\right) \right]^k \det\left(I_k-\sum_{l = 1}^p \hat{A}^{(p)}_{l, X}(y)\right) = 0,
\end{equation}
where $g_{2}: \Rel^{n} \times \Rel^{n \times k} \rightarrow \Rel$ is a multivariate polynomial which is independent of $(R,r)$ either. Summarizing our observations concerning (I) and (II) we see that
\begin{align}\label{NSET}
N(\hat{\Omega}_{\kappa, M, p, X}) &= \left\{y \in \Rel^n: g_1(y, X)g_2(y, X) = 0 \right\} \\[6pt] & \hspace{1.5cm} \cup \geschw{y \in \Rel^n: g_1(y, X)g_2(y, X) \neq 0 \text{ and } M(y) \text{ not w.d.}}.
\end{align}
The set in the second line of the previous display depends on the specific bandwidth $M$. Hence, we have to distinguish three cases: Suppose first that $M \equiv M_{KV} \in \mathbb{M}_{KV}$, i.e., $M$ is a constant which is functionally independent of $y$, $X$ and thus everywhere well defined on $\Rel^n$. Define $g_{\kappa, M_{KV}, p} \equiv g_1g_2$, so that $g_{\kappa, M_{KV}, p}: \Rel^n \times \Rel^{n \times k} \rightarrow \Rel$ is a multivariate polynomial. Noting that
\begin{equation}
N(\hat{\Omega}_{\kappa, M_{KV}, p, X}) = \geschw{y \in \Rel^n: g_{\kappa, M_{KV}, p}(y, X) = 0}
\end{equation}
then proves the statement in case $M \in \mathbb{M}_{KV}$, because $g_{\kappa, M_{KV}, p}$ is independent of $(R,r)$. Next we consider the case $M = M_{AM, \omega} \in \mathbb{M}_{AM}$, where we write $M_{AM, \omega}$ instead of $M_{AM, j, \omega, c}$, because the argument and the resulting function $g_{\kappa, M_{AM, \omega}, p}$ do not depend on $j$ and $c$. We partition 
\begin{equation}\label{MSET}
\geschw{y \in \Rel^n: g_1(y, X)g_2(y, X) \neq 0 \text{ and } M_{AM, \omega}(y) \text{ not w.d.}} = D_1 \cup D_2,
\end{equation}
where $D_1$ and $D_2$ are disjoint and defined as
\begin{align}
D_1 &= \geschw{y \in \Rel^n: g_1(y, X)g_2(y, X) \neq 0, \exists i^*: \hat{\rho}_{i^*}(y) \text{ not w.d. or } \hat{\rho}_{i^*}(y)^2 = 1}, \\
D_2 &= \geschw{y \in \Rel^n \backslash D_1:  g_1(y, X)g_2(y, X) \neq 0, \forall i \text{ s.t. } \omega_i \neq 0: \hat{\sigma}_i^2(y) = 0}. 
\end{align}
The equality in \eqref{MSET} is readily seen from the definition of $M_{AM, \omega}$. We want to obtain more suitable characterizations of $D_1$ and $D_2$ and proceed in two steps: (i) First, we claim that $y \in D_1$ if and only if 
\begin{equation} \label{A1}
g_1(y, X)g_2(y, X) \neq 0 \text{  and  } \prod_{i=1}^k \left( \left[\sum_{j = 2}^{n-p} [\hat{Z}_X(y)]_{ij} [\hat{Z}_X(y)]_{i(j-1)}\right]^2 - \left[ \sum_{j = 1}^{n-p-1} [\hat{Z}_X(y)]_{i j}^2\right]^2 \right) = 0.
\end{equation}
To see this assume that $g_1(y, X)g_2(y, X) \neq 0$ holds: Suppose that  $\hat{\rho}_{i^*}(y)$ is not well defined. The latter occurs if and only if $\sum_{j = 1}^{n-p-1} [\hat{Z}_X(y)]_{{i^*}j}^2=0$, i.e., all summands are zero, which immediately implies $\sum_{j = 2}^{n-p} [\hat{Z}_X(y)]_{{i^*}j} [\hat{Z}_X(y)]_{{i^*}(j-1)} = 0$. Therefore, the factor corresponding to index $i^*$ vanishes and thus the product defining the second equation in \eqref{A1} vanishes. That $\hat{\rho}^2_{i^*}(y) = 1$ implies that the product vanishes is obvious. To prove the other direction assume that $g_1(y, X)g_2(y, X) \neq 0$ and that the product vanishes. This implies that at least one factor with index $i^*$, say, equals zero, which implies that either $\hat{\rho}_{i^*}(y)$ is not well defined or $\hat{\rho}^2_{i^*}(y) = 1$ holds. This proves the claim. Secondly, we recall that if $g_1(y, X)g_2(y, X) \neq 0$, then $\hat{Z}_X(y) = \hat{V}_{p, X}(y) - \hat{A}^{(p)}_X(y) \hat{V}_{1, X}(y)$. Using an argument as above it is then easy to see that $ \hat{Z}_X(y)$ pre-multiplied by 
\begin{equation}\label{factor}
\det\left(\bar{V}_{1, X}(y)\bar{V}'_{1, X}(y)\right)\det(X'X)
\end{equation}
gives a matrix, the entries of which are values of certain multivariate polynomials defined on $\Rel^n \times \Rel^{n \times k}$ evaluated at the point $(y, X)$. Thus, if we multiply the second equation in \eqref{A1} by the $(4k)$-th power of the expression in the previous display we see that Equation \eqref{A1} can equivalently be written as
\begin{equation}
g_1(y, X)g_2(y, X) \neq 0 \text{  and  } g_{AM, 1}(y,X) = 0,
\end{equation}
where $g_{AM, 1}: \Rel^n \times \Rel^{n \times k} \rightarrow \Rel$ is a multivariate polynomial that is independent of $(R,r)$. Summarizing, we have shown that
\begin{equation}
D_1 = \geschw{y \in \Rel^n: g_1(y, X) g_2(y, X) \neq 0, g_{AM, 1}(y, X) = 0}.
\end{equation}
(ii) First we observe that $y \in D_2$ if and only if
\begin{equation}\label{A3}
g_1(y, X) g_2(y, X)g_{AM, 1}(y, X) \neq  0 \text{ and } \sum_{i=1}^k \omega_i
\sum_{j = 2}^{n-p}\left([\hat{Z}_X(y)]_{ij} - \hat{\rho}_i(y) [\hat{Z}_X(y)]_{i(j-1)}\right)^2 = 0,
\end{equation}
where we recall that by assumption $\omega$ is functionally independent of $y$ and $X$. Because $g_{AM,1}(y, X) \neq  0$ implies that $\hat{\rho}_i(y)$ is well defined for $i = 1, \hdots, k$, which is equivalent to $\sum_{j = 1}^{n-p-1} [\hat{Z}_X(y)]_{{i}j}^2\neq0$ for $i = 1, \hdots, k$, the second equation in the previous display can be replaced by
\begin{equation}
\sum_{i=1}^k \omega_i
\sum_{j = 2}^{n-p}\left(\left[\sum_{j = 1}^{n-p-1} [\hat{Z}_X(y)]_{{i}j}^2\right] [\hat{Z}_X(y)]_{ij} - \left[\sum_{j = 2}^{n-p} [\hat{Z}_X(y)]_{{i}j} [\hat{Z}_X(y)]_{{i}(j-1)} \right] [\hat{Z}_X(y)]_{i(j-1)}\right)^2 = 0.
\end{equation}
We now multiply the function defining this equation by the $6$-th power of the expression in Equation \eqref{factor} and denote the resulting function by $g_{AM, \omega,2}(y, X)$. The statement in Equation \eqref{A3} is then seen to be equivalent to 
\begin{equation}
g_1(y, X) g_2(y, X)g_{AM,1}(y, X) \neq  0 \text{ and } g_{AM, \omega, 2}(y, X) = 0,
\end{equation}
where $g_{AM, \omega, 2}: \Rel^n \times \Rel^{n \times k} \rightarrow \Rel$ is a multivariate polynomial. We also see that $g_{AM, \omega, 2}$ is independent of $(R,r)$. We conclude that 
\begin{equation}
D_2 = \geschw{y \in \Rel^n: g_1(y, X) g_2(y, X)g_{AM, 1}(y, X) \neq 0, g_{AM, \omega, 2}(y, X) = 0}.
\end{equation}
Now let $g_{\kappa, M_{AM, \omega},p} \equiv g_1g_2 g_{AM,1}g_{AM, \omega, 2}$. By what has been shown above $g_{\kappa, M_{AM, \omega},p}: \Rel^n \times \Rel^{n \times k} \rightarrow \Rel$ is a multivariate polynomial. Furthermore, $g_{\kappa, M_{AM, \omega},p}$ does not depend on $(R,r)$. We observe that
\begin{equation}
N(\hat{\Omega}_{\kappa, M_{AM, \omega}, p, X}) = \geschw{y \in \Rel^n: g_{\kappa, M_{AM, \omega}, p}(y,X) = 0}.
\end{equation}
This proves the lemma in case $M \in \mathbb{M}_{AM}$. Finally, we consider $M \equiv M_{NW,\omega, w} \in \mathbb{M}_{NW}$, where we write $M_{NW,\omega, w}$ instead of $M_{NW,\omega, w, \bar{c}}$, because the argument and the resulting polynomial are independent of $\bar{c}$. We use a similar argument as in the previous case. We observe that if $g_1(y, X)g_2(y, X) \neq 0$, then the function $M_{NW,\omega, w}$ is not well defined if and only if 
\begin{equation}\label{N1}
\sum_{i = -(n-p-1)}^{n-p-1} w(i) \bar{\sigma}_i(y) = 0,
\end{equation}
where
\begin{equation}
\bar{\sigma}_i(y) = (n-p)^{-1}\sum_{j=|i| + 1}^{n-p} \omega'[\hat{Z}_X(y)]_{\cdot j} \left([\hat{Z}_X(y)]_{\cdot (j-|i|)}\right)'\omega \text{   for   } |i| = 0, \hdots n-p-1.
\end{equation}
Since $\omega$ and $w$ are both functionally independent of $X$ and $y$, we can pre-multiply Equation \eqref{N1} by the square of the expression in Equation \eqref{factor} to see that the statement $g_1(y, X)g_2(y, X) \neq 0$ and $M_{NW, \omega, w}$ not being well defined is equivalent to 
\begin{equation}
g_1(y, X)g_2(y, X) \neq 0 \text{ and } g_{NW, \omega, w}(y, X) = 0,
\end{equation}
where $g_{NW, \omega, w}: \Rel^n \times \Rel^{n \times k} \rightarrow \Rel$ is a multivariate polynomial. The function $g_{NW, \omega, w}$ is independent of $(R,r)$. Using these properties, defining $g_{\kappa, M_{NW, \omega, w}, p} = g_1g_2g_{NW, \omega, w}$, a function which does not depend on $(R,r)$, and noting that
\begin{equation}
N(\hat{\Omega}_{\kappa, M_{NW,\omega, w}, p, X}) = \geschw{y \in \Rel^n: g_{\kappa, M_{NW, \omega, w}, p}(y,X) = 0},
\end{equation}
then proves the claim in case $M \in \mathbb{M}_{NW}$.
\end{proof}

\begin{proof}[Proof of Lemma \ref{N*PW}]
To establish Parts 1-4 of the lemma we apply a similar argument as in the proof of Lemma 3.1 in \cite{PP13}. We observe that if $y \notin N(\hat{\Omega}_{\kappa, M, p, X})$, or equivalently $g_{\kappa, M, p}(y, X) \neq 0$, we can write $\hat{\Omega}_{\kappa, M, p, X}(y)$ as
\begin{equation}\label{QREPr}
\hat{\Omega}_{\kappa, M, p, X}(y) = \frac{n}{n-p}B_{p,X}(y) \mathcal{W}_{n-p}(y) B'_{p,X}(y),
\end{equation}
where $\mathcal{W}_{n-p}(y) \in \Rel^{(n-p)\times (n-p)}$ is the symmetric Toeplitz matrix with ones on the main diagonal, and where for $i \neq j$ its $ij$-th coordinate is given by $\kappa((i-j)/M(y))$ whenever $M(y) \neq 0$, and by $0$ else. Recall that $M(y) \geq 0$. If $M(y) = 0$ we have $\mathcal{W}_{n-p}(y) = I_n$. If $M(y) > 0$ the matrix $\mathcal{W}_{n-p}(y)$ is positive definite by Assumption \ref{weightsPWRB}. Therefore, in both cases the matrix $\mathcal{W}_{n-p}(y)$ is positive definite. This immediately establishes Parts 1-4, where $\rank(R) = q$ is used in proving Part 4 (we emphasize that $\hat{\Omega}_{\kappa, M, p, X}(y)$ can be nonnegative definite, singular, zero or positive definite only if it is well defined, i.e., only if $g_{\kappa, M, p}(y, X) \neq 0$ holds). It remains to prove Part 5. We recall that
\begin{equation}\label{repN*}
N^*(\hat{\Omega}_{\kappa, M, p, X}) = N(\hat{\Omega}_{\kappa, M, p, X}) \cup \geschw{y \in \Rel^n \backslash N(\hat{\Omega}_{\kappa, M, p, X}): \det\left[\hat{\Omega}_{\kappa, M, p, X}(y)\right] = 0}.
\end{equation}
From Part 2 of the present lemma we know that we can rewrite the second set to the right as
\begin{equation}\label{secondN*PW}
\geschw{y \in \Rel^n \backslash N(\hat{\Omega}_{\kappa, M, p, X}): \det\left[B_{p,X}(y)B'_{p,X}(y)\right] = 0}.
\end{equation}
For every $y \in \Rel^n \backslash N(\hat{\Omega}_{\kappa, M, p, X})$ we have with $D(p) = (I_k, \hdots , I_k)' \in \Rel^{(kp) \times k}$ that (using the same notation as in the proof of Lemma \ref{NPWRB}) $B_{p,X}(y)$ can be written as
\begin{equation}
\begin{aligned}
& R(X'X)^{-1}\left(I_k - \bar{V}_{p, X}(y) \bar{V}'_{1, X}(y)\left[\bar{V}_{1, X}(y)\bar{V}'_{1, X}(y)\right]^{-1} D(p) \right)^{-1} \\ & \hspace{1cm}\times \left(\hat{V}_{p, X}(y) - \bar{V}_{p, X}(y) \bar{V}'_{1, X}(y)\left[\bar{V}_{1, X}(y)\bar{V}'_{1, X}(y)\right]^{-1} \hat{V}_{1, X}(y)\right) \\[7pt]
=& R(X'X)^{-1} \det(X'X)^{-1} \left(\det(\left[\bar{V}_{1, X}(y)\bar{V}'_{1, X}(y)\right]) I_k - \bar{V}_{p, X}(y) \bar{V}'_{1, X}(y)\adj \left[\bar{V}_{1, X}(y)\bar{V}'_{1, X}(y)\right] D(p) \right)^{-1} \\ 
& \hspace{1cm} \times \left(\det\left[\bar{V}_{1, X}(y)\bar{V}'_{1, X}(y)\right] \bar{V}_{p, X}(y) - \bar{V}_{p, X}(y) \bar{V}'_{1, X}(y)\adj\left[\bar{V}_{1, X}(y)\bar{V}'_{1, X}(y)\right] \bar{V}_{1, X}(y)\right) \\[7pt]
=& \det(X'X)^{-2} \det \left(\det(\left[\bar{V}_{1, X}(y)\bar{V}'_{1, X}(y)\right]) I_k - \bar{V}_{p, X}(y) \bar{V}'_{1, X}(y)\adj \left[\bar{V}_{1, X}(y)\bar{V}'_{1, X}(y)\right] D(p) \right)^{-1} \\ 
& \hspace{1cm}\times R\adj(X'X)  \adj\left(\det(\left[\bar{V}_{1, X}(y)\bar{V}'_{1, X}(y)\right]) I_k - \bar{V}_{p, X}(y) \bar{V}'_{1, X}(y)\adj \left[\bar{V}_{1, X}(y)\bar{V}'_{1, X}(y)\right] D(p) \right) \\ 
& \hspace{1cm}\times \left(\det\left[\bar{V}_{1, X}(y)\bar{V}'_{1, X}(y)\right] \bar{V}_{p, X}(y) - \bar{V}_{p, X}(y) \bar{V}'_{1, X}(y)\adj\left[\bar{V}_{1, X}(y)\bar{V}'_{1, X}(y)\right] \bar{V}_{1, X}(y)\right)
\end{aligned}
\end{equation}
We therefore see that the coordinates of the matrix $\bar{B}_{p,X}(y)$, say, which is obtained by pre-multiplying $B_{p,X}(y)$ by the factor
\begin{equation}
F_p(y, X) = \det(X'X)^{2} \det \left(\det(\left[\bar{V}_{1, X}(y)\bar{V}'_{1, X}(y)\right]) I_k - \bar{V}_{p, X}(y) \bar{V}'_{1, X}(y)\adj \left[\bar{V}_{1, X}(y)\bar{V}'_{1, X}(y)\right] D(p) \right) 
\end{equation}
(for later reference we note that $F_p: \Rel^n \times \Rel^{n \times k} \rightarrow \Rel$ is a multivariate polynomial), are values of certain multivariate polynomials defined on $\Rel^n \times \Rel^{n \times k}$  evaluated at $(y, X)$. Furthermore, we can replace ${B}_{p,X}(y)$ in Equation \eqref{secondN*PW} by $\bar{B}_{p,X}(y)$ without changing the set. This follows because $y \notin N(\hat{\Omega}_{\kappa, M, p})$ implies
\begin{equation}
F_p(y, X) = \det(X'X)^{2} \det(\bar{V}_{1, X}(y)\bar{V}'_{1, X}(y))^k \det \left(I_k - \sum_{l = 1}^p \hat{A}_{l,X}^{(p)}(y) \right) \neq 0.
\end{equation}
If we combine this equivalent expression for \eqref{secondN*PW} with \eqref{repN*} and Lemma \ref{NPWRB} we obtain
\begin{equation}
N^*(\hat{\Omega}_{\kappa, M, p, X}) = \geschw{y \in \Rel^n: g_{\kappa, M,p}(y, X) \det\left[\bar{B}_{p,X}(y)\bar{B}'_{p,X}(y)\right] = 0}.
\end{equation}
We next define $g^*_{\kappa, M,p}(y, X, R) \equiv g_{\kappa, M,p}(y, X) \det[\bar{B}_{p,X}(y)\bar{B}'_{p,X}(y)]$. By Lemma \ref{NPWRB} we see that $g^*_{\kappa, M,p}: \Rel^n \times \Rel^{n \times k} \times \Rel^{q \times k} \rightarrow \Rel$ is a multivariate polynomial that does not depend on $r$. 
\end{proof}

The subsequent technical lemma plays a key role in several constructions in the proofs of the genericity results. 

\begin{lemma}\label{AUXCONSTR}
Let $1 \leq k < n$, $n > 2$ and let $(R,r)$ be a hypothesis. Suppose that the triple $\kappa$, $M$, $p$ satisfies Assumption \ref{weightsPWRB}. Assume that the tuple $(y, X) \in \Rel^n \times \mathfrak{X}_0$ satisfies for some $t \geq k$:
\begin{enumerate}
  \item[(A1)] $\hat{V}_X(y)$ has exactly $t+1$ nonzero columns with indices $1 = j_1 < j_2 < \hdots < j_{t+1} \leq n$.  
  \item[(A2)] $j_{i+1} - j_{i} \geq p+1$ for $i = 1, \hdots, t$, and $n- j_{t+1} \geq p-1$.
  \item[(A3)] If $t = k$, then $\rank(\hat{V}_X(y)) = k$. Otherwise, 
	\begin{equation}
	  \lspan(\geschw{[\hat{V}_X(y)]_{\cdot j_i}: i = 1, \hdots, t}) = \lspan(\geschw{[\hat{V}_X(y)]_{\cdot j_i}: i = 2, \hdots, t+1}) = \Rel^k.
	\end{equation}
\end{enumerate}
Then, the following holds:
\begin{enumerate}
	\item $\hat{A}^{(p)}_{X}(y) = 0$.
	\item Under each of the following three conditions it follows that $g^*_{\kappa, M, p}(y,X, R) \neq 0$ (or equivalently $y \notin N^*(\hat{\Omega}_{\kappa, M, p, X})$):
\begin{enumerate}
	\item[(CKV)] $M \in \mathbb{M}_{KV}$; 
	\item[(CAM)] $M \in \mathbb{M}_{AM}$, and every row vector of the matrix obtained from $\hat{Z}_X(y)$ by deleting its last column is nonzero [this is in particular satisfied if $n-j_{t+1} > p-1$];
	\item[(CNW)] $M \in \mathbb{M}_{NW}$, and either each coordinate of $\omega' \hat{Z}_X(y)$ is non-negative, or each coordinate of $\omega' \hat{Z}_X(y)$ is non-positive.
\end{enumerate}
 \item For every $Q \in \Rel^{k \times k}$ such that $\rank(Q) = k$, the tuple $(y, XQ)$ is an element of $\Rel^n \times \mathfrak{X}_0$ that satisfies (A1), (A2) and (A3).
 \item If $k \geq 2$ and either $[\hat{V}_{X}(y)]_{{1j_{i}}}> 0$ for $i = 2, \hdots, t+1$ or $[\hat{V}_{X}(y)]_{{1j_{i}}}< 0$ for $i = 2, \hdots, t+1$ holds, then there exists a regular matrix $Q \in \Rel^{k\times k}$ such that the first columns of $X$ and $XQ$, respectively, coincide and such that $g^*_{\kappa, M, p}(y,XQ, R) \neq 0$ (or equivalently $y \notin N^*(\hat{\Omega}_{\kappa, M, p, XQ})$).
\end{enumerate}
\end{lemma}
\begin{proof}
Denote the column vectors of $\hat{V}_X(y)$ by $v_i$ for $i = 1, \hdots, n$. If $t > k$, then by (A3) the set $\geschw{v_{j_1}, \hdots, v_{j_t}}$ and $\geschw{v_{j_2}, \hdots, v_{j_{t+1}}}$, respectively, spans $\Rel^k$. Using (A3), we now show that this is automatically satisfied in case $t = k$. To prove this claim, we first recall that $v_j = [\hat{u}_X(y)]_j X_{j \cdot}'$. We see that $\hat{u}_X(y) \bot \lspan(X)$ implies
\begin{equation}
0 = \sum_{j = 1}^n [\hat{u}_{X}(y)]_j X_{j \cdot}' = \sum_{j = 1}^n v_j = \sum_{i = 1}^{k+1} v_{j_i},
\end{equation}
where the third equality follows from (A1) ($t = k$). This shows that
\begin{equation}
\begin{aligned}
&v_{j_1} &= & ~~~~ -\sum_{i = 2}^{k+1} v_{j_i} \\
&v_{j_{k+1}} &= & ~~~~ -\sum_{i = 1}^{k} v_{j_i} .
\end{aligned}
\end{equation}
By (A3) $\rank(\hat{V}_X(y)) = k$ , which together with (A1) implies that $\lspan(\geschw{v_{j_i}: i = 1, \hdots, k+1}) = \Rel^k$. Therefore, it follows from the two equations in the previous display that $\geschw{v_{j_i}: i = 1, \hdots, k}$ and $\geschw{v_{j_i}: i = 2, \hdots, k+1}$, respectively, spans $\Rel^k$. Hence the claim follows. We next show that $\hat{A}^{(p)}_{X}(y)$ is well defined. For this we have to verify that $\rank(\hat{V}_{1, X}(y)) = kp$ (cf. Remark \ref{WDPSI}). The $j$-th column ($j = 1, \hdots, n-p$) of $\hat{V}_{1, X}(y)$ is given by
\begin{equation}\label{Vcol}
(v'_{j+p-1}, \hdots, v'_{j+1}, v'_j)' \in \Rel^{kp},
\end{equation}
which is to be interpreted as $v_j$ if $p = 1$, as $(v'_{j+1}, v'_j)'$ if $p = 2$ etc. For $l = 1, \hdots, p$ we define the $(kp) \times k$ dimensional auxiliary matrix $D_l = e_{l}(p) \otimes I_k$, where $e_{l}(p)$ denotes the $l$-th element of the canonical basis of $\Rel^p$ (and $\otimes$ denotes the Kronecker product). The following claims are immediate consequences of the structure of $\hat{V}_X(y)$ implied by (A1) - (A2): 
\begin{enumerate}
	\item[(I)] $D_p v_{j_1}$ is the first column of $\hat{V}_{1, X}(y)$;
	\item[(II)]If $t \geq 2$, then $D_l v_{j_i}$ for $i = 2, \hdots, t$ and $l = 1, \hdots, p$ are columns of $\hat{V}_{1, X}(y)$;
	\item[(III)] If $p\geq 2$, then $D_l v_{j_{t+1}}$ for $l = 1, \hdots, (p-1)$ are columns of $\hat{V}_{1, X}(y)$.
\end{enumerate}%
To see Parts (I) and (II), we observe that (A1) and (A2) imply that for $i = 1, \hdots, t$, there are at least $p$ zero columns between the columns $v_{j_i}$ and $v_{j_{i+1}}$ of $\hat{V}_X(y)$. Equation \eqref{Vcol} together with $j_1 = 1$ then immediately implies Parts (I) and (II). Now we consider Part (III) and hence assume that $p \geq 2$. We start with the case $l = (p-1)$. Every column of $\hat{V}_X(y)$ with index greater than $j_{t+1}$ is zero by Assumption (A1). By Assumption (A2) we have $n - j_{t+1} \geq p-1$. Together, this implies that the column $v_{j_{t+1}}$ is followed by at least $p-1$ zero columns. Since $j_{t+1} - j_{t} \geq p+1$ by Assumption (A2), the column $v_{j_{t+1}}$ is preceded by at least $p$ zero columns. The assumption $n - j_{t+1} \geq p-1$ is equivalent to $n - p \geq j_{t+1} -1$. Hence, denoting the $m_1 \times m_2$-dimensional zero matrix by $0_{m1, m2}$, we can use Equation \eqref{Vcol} with $j = j_{t+1} -1$ to see that 
\begin{equation}
(v'_{j_{t+1}+p-2}, \hdots, v'_{j_{t+1}}, v'_{j_{t+1} -1})' = \begin{cases} (0_{1,(p-2)k}, v'_{j_{t+1}}, 0_{1,k})' & \text{ if } p > 2 \\
(v'_{j_{t+1}}, 0_{1,k})' & \text{ if } p = 2,
\end{cases}
\end{equation}
is a column of $\hat{V}_{1, X}(y)$, where in deriving the equality we made use of the already established fact that the column  $v_{j_{t+1}}$ of $\hat{V}_X(y)$ is preceded by at least $p > 1$ zero columns (which implies that $v'_{j_{t+1} -1}$ is the zero vector), and  followed by at least $p-1$ zero columns (which is used in case $p > 2$). This proves the statement concerning $D_{p-1} v_{j_{t+1}}$. In case $p = 2$ we are done. If $p > 2$, then the statements concerning $D_l v_{j_{t+1}}$ for $l = 1, \hdots, p-2$ follow from Equation \eqref{Vcol} together with the equation in the previous display and the fact that $v_{j_{t+1}}$ is preceded by at least $p$ zero columns. 

We now use (I)-(III) together with $\lspan(\geschw{v_{j_1}, \hdots, v_{j_t}}) = \lspan(\geschw{v_{j_2}, \hdots, v_{j_{t+1}}}) = \Rel^k$ to show that $\rank(\hat{V}_{1, X}(y)) = kp$: The matrix $\hat{V}_{1, X}(y)$ is $kp \times (n-p)$ dimensional. Therefore, we must show that it has full row rank. Assume existence of a row vector $\xi = (\xi_1, \hdots, \xi_p)$, where $\xi_i' \in \Rel^k$ for $i = 1,\hdots, p$, such that $\xi \hat{V}_{1, X}(y) = 0$ holds. Part (I) shows that $0 = \xi D_p v_{j_1} = \xi_p v_{j_1}$. If $t \geq 2$, then Part (II) applied with $l = p$ shows that $0 = \xi D_pv_{j_i} = \xi_p v_{j_i}$ for $i = 2, \hdots, t$. Summarizing the cases $t = 1$ and $t \geq 2$ we obtain $\xi_p v_{j_i} = 0$ for $i = 1, \hdots, t$. Because $\geschw{v_{j_1}, \hdots, v_{j_t}}$ spans $\Rel^k$, it follows that $\xi_p = 0$. If $p \geq 2$, Part (III) implies that $0 = \xi D_l v_{j_{t+1}} = \xi_l v_{j_{t+1}}$ for $l = 1, \hdots, (p-1)$. If $t \geq 2$ Part (II) implies $0 = \xi D_l v_{j_i} = \xi_l v_{j_i}$ for $l = 1, \hdots, (p-1)$ and $i = 2, \hdots, t$. Summarizing again the cases $t = 1$ and $t \geq 2$ we obtain that for $l = 1, \hdots, (p-1)$ we have
\begin{equation}
\xi_l v_{j_i} = 0 \text{ for } i = 2, \hdots, t+1.
\end{equation}
Because $\geschw{v_{j_2}, \hdots, v_{j_{t+1}}}$ spans $\Rel^k$, it follows from the previous display that $\xi_l = 0$ for $l = 1, \hdots, p-1$. Since we already know that $\xi_p = 0$, we obtain $\xi = 0$ and thus $\rank(\hat{V}_{1, X}(y)) = kp$. Therefore $\hat{A}^{(p)}_X(y)$ is well defined. To see that $\hat{A}^{(p)}_X(y) = 0$ we observe that every nonzero column of $\hat{V}_X(y)$ besides the first one is preceded by at least $p$ zero columns. The matrix $\hat{V}_{p, X}(y)$ is obtained from $\hat{V}_X(y)$ by deleting the first $p \geq 1$ columns. This together with Equation \eqref{Vcol} immediately implies $\hat{V}_{p,X}(y) \hat{V}_{1,X}'(y) = 0$ and thus $\hat{A}^{(p)}_X(y) = 0$. 

To show that $y \notin N^*(\hat{\Omega}_{\kappa, M, p, X})$ under the conditions (CKV), (CAM) and (CNW), respectively, we first note that $\rank(\hat{Z}_X(y)) = k$. This follows, because $\hat{A}^{(p)}_X(y) = 0$ implies $\hat{Z}_X(y) = \hat{V}_{p, X}(y)$, which together with $j_2-j_1 \geq p+1$ shows that the vectors $v_{j_i}$ for $i = 2, \hdots, t+1$ (which span $\Rel^k$) are column vectors of $\hat{Z}_X(y)$. As a consequence of Part 4 of Lemma \ref{N*PW} positive definiteness of $\hat{\Omega}_{\kappa, M, p, X}(y)$ and hence $y \notin N^*(\hat{\Omega}_{\kappa, M, p, X})$ follows if we can show that $\hat{\Omega}_{\kappa, M, p, X}(y)$ is well defined. Since $\hat{A}^{(p)}_X(y) = 0$ implies invertibility of $I_k - \sum_{l = 1}^p \hat{A}_{i, X}^{(p)}(y)$, it remains to show that $M$ is well defined at $y$ (cf. Remark \ref{WDPSI}). This is trivially satisfied under Condition (CKV) because $M \in \mathbb{M}_{KV}$ is everywhere well defined. Suppose that Condition (CAM) holds, i.e., $M \in \mathbb{M}_{AM}$ and every row vector of the matrix obtained from $\hat{Z}_X(y)$ by deleting the last column is nonzero. That the latter condition is satisfied if $n-j_{t+1} > p-1$ follows because in that case the last column of $\hat{Z}_X(y)$ is the zero vector and $\rank(\hat{Z}_X(y)) = k$. Under the assumption that every row vector of the matrix obtained from $\hat{Z}_X(y)$ by deleting the last column is nonzero it is obvious that the denominators in the definition of $\hat{\rho}_i(y)$ for $i = 1, \hdots, k$, i.e., 
\begin{equation}
\sum_{j = 1}^{n-p-1} [\hat{Z}_X(y)]_{ij}^2 ~~ \text{ for } ~~ i = 1, \hdots, k,
\end{equation}
do not vanish. Therefore, $\hat{\rho}_i(y)$ for $i = 1, \hdots, k$ are well defined. Using Assumptions (A1) and (A2) together with $p \geq 1$ and $\hat{Z}_X(y) = \hat{V}_{p, X}(y)$, it follows that there is always at least one zero column between two nonzero columns of $\hat{Z}_X(y)$. Therefore, it is clear that the numerators appearing in the definition of $\hat{\rho}_i(y)$ for $i = 1, \hdots, k$, i.e., 
\begin{equation}
\sum_{j = 2}^{n-p} [\hat{Z}_X(y)]_{ij}[\hat{Z}_X(y)]_{i(j-1)} ~~ \text{ for } ~~ i = 1, \hdots, k,
\end{equation}
must vanish. It follows that $\hat{\rho}_i(y) = 0$ for $i = 1, \hdots, k$. We finally show that $\hat{\sigma}_i(y) > 0$ for $i = 1, \hdots, k$. We note that $\hat{\rho}_i(y) = 0$ for $i = 1, \hdots, k$ implies
\begin{equation}
\hat{\sigma}_i(y) = (n-p-1)^{-1}\sum_{j = 2}^{n-p}[\hat{Z}_X(y)]^2_{ij} \text{ for } i = 1, \hdots, k,
\end{equation}
Because of $\hat{Z}_X(y) = \hat{V}_{p, X}(y)$ it follows from Assumptions (A1) and (A2) that the first column of $\hat{Z}_X(y)$ must be zero. Furthermore, we already know that $\rank(\hat{Z}_X(y)) = k$. This implies that the matrix $Z_*$, say, which is obtained from $\hat{Z}_X(y)$ by deleting the first column, must be of full row rank $k$. Consequently all rows of $Z_*$ must be non-zero. The previous display shows that $\hat{\sigma}_i(y)$ for $i = 1, \hdots, k$ is, up to a positive factor, the squared Euclidean norm of the $i$-th row of $Z_*$. Therefore $\hat{\sigma}_i(y) > 0$ for $i = 1, \hdots, k$ must hold. Therefore, we have shown that $M(y)$ is well defined (we even see that $M(y) = 0$ holds). Now we consider the case where Condition (CNW) holds, i.e., $M \in \mathbb{M}_{NW}$ and every coordinate of $\omega' \hat{Z}_X(y)$ is non-negative (non-positive). We have to show that 
\begin{equation}
\sum_{i = -(n-p-1)}^{n-p-1} w(i) \bar{\sigma}_i(y) \neq 0.
\end{equation}
The non-negativity (non-positivity) condition immediately implies $\bar{\sigma}_i(y) \geq 0$ for $|i| = 0, 1, \hdots, n-p-1$. Furthermore, since $\rank(\hat{Z}_X(y)) = k$ and $\omega \neq 0$, by the definition of the weights vector, we have $\omega' \hat{Z}_X(y) \neq 0$. This implies $\bar{\sigma}_0(y) = (n-p)^{-1} \norm{\omega' \hat{Z}_X(y)}^2 > 0$. By assumption $w(0) = 1$ and $w(i) \geq 0$ for $|i| = 1, \hdots, n-p-1$. Therefore, the quantity in the previous display does not vanish and thus $M(y)$ is well defined. This proves the second part of the lemma.

We next prove Part 3. Let $Q$ be a regular $k \times k$ dimensional matrix. First, we obviously have $X Q \in \mathfrak{X}_0$, because $X \in \mathfrak{X}_0$ and $Q$ is regular. Secondly, since $\lspan(X) = \lspan(XQ)$, using regularity of $Q$, we have that $\hat{u}_{X}(y) = \hat{u}_{XQ}(y)$. This immediately entails 
\begin{equation}
\hat{V}_{XQ}(y) = (XQ)'\diag(\hat{u}_{XQ}(y)) = (XQ)'\diag(\hat{u}_{X}(y)) = Q' X' \diag(\hat{u}_{X}(y)) = Q'\hat{V}_{X}(y).
\end{equation}
Therefore, the tuple $(y, XQ) \in \Rel^n \times \mathfrak{X}_0$ satisfies (A1), (A2) and (A3), because $(y, X)$ does so and $Q$ is regular.

It remains to prove Part 4. We do this by constructing a $Q$ as in Part 3 such that the tuple $(y, XQ) \in \Rel^n \times \mathfrak{X}_0$ satisfies Condition (CKV), (CAM) or (CNW), respectively, if $M \in \mathbb{M}_{KV}$, $M \in \mathbb{M}_{AM}$ or $M \in \mathbb{M}_{NW}$, respectively. If $M \in \mathbb{M}_{KV}$ we can obviously choose $Q = I_k$. Suppose that $M \notin \mathbb{M}_{KV}$. Let $Q \in \Rel^{k \times k}$ be such that $\rank(Q) = k$. We specify this matrix later on. From Part 2 of the present lemma we see that the tuple $(y, XQ)$ satisfies (A1), (A2) and (A3) and therefore we can conclude from Part 1 of the present lemma that $\hat{A}^{(p)}_{XQ}(y) = 0$, which implies $\hat{Z}_{XQ}(y) = \hat{V}_{p, XQ}(y)$. Together with the equation in the previous display, we see that
\begin{equation} \label{ZT}
\hat{Z}_{XQ}(y) = \hat{V}_{p, XQ}(y) = Q'\hat{V}_{p, X}(y).
\end{equation}
We now want to choose $Q \in \Rel^{k \times k}$ (regular) such that 
\begin{enumerate}
	\item[(i)] every row vector of the matrix obtained from $\hat{Z}_{XQ}(y)$ by deleting the last column is nonzero, and
	\item[(ii)] either every coordinate of $\omega' \hat{Z}_{XQ}(y)$ is non-negative, or every coordinate of $\omega' \hat{Z}_{XQ}(y)$ is non-positive,
\end{enumerate}
holds, and that additionally the first columns of $X$ and $XQ$, respectively, coincide. By assumption we either have $[\hat{V}_{X}(y)]_{{1j_{i}}}> 0$ for $i = 2, \hdots, t+1$, or we have $[\hat{V}_{X}(y)]_{{1j_{i}}}< 0$ for $i = 2, \hdots, t+1$. Consider the former (latter) case: Let 
\begin{equation}
Q(\gamma) = \begin{pmatrix}
1 & \gamma & \gamma & \gamma & \hdots & \gamma\\
0 & 1 & 0 & 0 & \hdots & 0 \\
0 & 0 & 1 & 0 & \hdots & 0\\
\vdots & \vdots & \vdots & \vdots& \vdots & \vdots \\
0 & 0 & 0 &0 &\hdots & 1
\end{pmatrix},
\end{equation}
which is a regular matrix for every $\gamma \in \Rel$. Post-multiplying $X$ by $Q(\gamma)$ has the same effect as adding $\gamma$-times the first column of $X$ to all other columns, without changing the first column. We observe that since $\hat{V}_{p,X}(y)$ is obtained from $\hat{V}_X(y)$ by deleting the first $p$ columns, the nonzero columns of $\hat{V}_{p, X}(y)$ are precisely the vectors $[\hat{V}_{X}(y)]_{{\cdot j_{i}}}$ for $i = 2, \hdots, t+1$ because $(y, X)$ satisfies Assumptions (A1) and (A2). Therefore, it is obvious from Equation \eqref{ZT}, together with the assumed $[\hat{V}_{X}(y)]_{{1j_{i}}}> 0$ for $i = 2, \hdots, t+1$  ($[\hat{V}_{X}(y)]_{{1j_{i}}}<  0$ for $i = 2, \hdots, t+1$), that by choosing $\gamma^* > 0$ large enough, we can enforce that all nonzero columns of $\hat{Z}_{XQ(\gamma^*)}(y)$ are coordinate-wise positive (negative). Consider (i): Using Equation \eqref{ZT} we see that $Q(\gamma^*)' [\hat{V}_X(y)]_{\cdot j_2}$ is a column of the matrix obtained by deleting the last column of $\hat{Z}_{XQ(\gamma^*)}(y)$. This follows from  $j_2 \geq p+2$ (a consequence of the assumptions $j_2 - j_1 \geq p+1$ and $j_1 = 1$), which shows that $Q(\gamma^*)' [\hat{V}_X(y)]_{\cdot j_2}$ is a column of $\hat{Z}_{XQ(\gamma^*)}(y)$, together with $j_2 < j_3 \leq n$ (a consequence of $k \geq 2$), which shows that it is not the last column of $\hat{Z}_{XQ(\gamma^*)}(y)$. Since all coordinates of $Q(\gamma^*)' [\hat{V}_X(y)]_{\cdot j_2}$ are positive (negative) by construction of $Q(\gamma^*)$, it follows that $Q(\gamma^*)$ satisfies (i) above. To show (ii) we recall that $\omega$ is nonzero and coordinate-wise nonnegative. Since all nonzero columns of $\hat{Z}_{XQ(\gamma^*)}(y)$ are coordinate-wise positive (negative), it immediately follows that every coordinate of $\omega' \hat{Z}_{XQ(\gamma^*)}(y)$ is non-negative (non-positive). By construction the first columns of $X$ and $XQ(\gamma^*)$ coincide. This proves the claim.
\end{proof}

\begin{proof}[Proof of Lemma \ref{nullN*PW}]
We start with the first part. Let $X \in \mathfrak{X}_0 \subseteq \Rel^{n \times k}$ and $y \in \Rel^n$ be arbitrary but fixed. We show that $y \in N^{*}(\hat{\Omega}_{\kappa, M, p, X})$, which is equivalent to $g_{\kappa, M, p}^*(y, X, R) = 0$ by Part 5 of Lemma \ref{N*PW}. If $y \in N(\hat{\Omega}_{\kappa, M, p, X}) \subseteq N^{*}(\hat{\Omega}_{\kappa, M, p, X})$ we are done. Suppose $y \notin N(\hat{\Omega}_{\kappa, M, p, X})$ which is equivalent to $g_{\kappa, M, p}(y, X) \neq 0$ by Lemma \ref{NPWRB}. We claim that $\rank(\hat{Z}_X(y)) < k$ must hold. Assuming this claim and using $\rank(R) = q = k$, $X \in \mathfrak{X}_0$ which implies $\rank(X) = k$, and $y \notin N(\hat{\Omega}_{\kappa, M, p, X})$ which implies $\rank(I_k - \sum_{l = 1}^p \hat{A}_l^{(p)}(y)) = k$, it then follows from the definition of $B_{p, X}(y)$ in Equation \eqref{defB} that $\rank(B_{p, X}(y)) < q$. As a consequence, Part 2 of Lemma \ref{N*PW} then shows that $\hat{\Omega}_{\kappa, M, p, X}(y)$ is singular, which implies $y \in N^{*}(\hat{\Omega}_{\kappa, M, p, X})$. To prove $\rank(\hat{Z}_X(y)) < k$ we note that
\begin{equation}
\hat{Z}_X(y) = \hat{V}_{p, X}(y) \left[I_{n-p} - \hat{V}'_{1, X}(y)\left(\hat{V}_{1, X}(y)\hat{V}'_{1, X}(y)\right)^{-1} \hat{V}_{1, X}(y) \right] = \hat{V}_{p, X}(y)\Pi_{\lspan(\hat{V}'_{1, X}(y))^{\bot}}.
\end{equation}
We see from the previous display that $\rank(\hat{Z}_X(y)) = k$, i.e., $\hat{Z}_X(y)$ having full row rank, is equivalent to $\rank(\hat{V}_{p, X}(y)) = k$ and $\lspan(\hat{V}'_{1, X}(y)) \cap \lspan(\hat{V}'_{p, X}(y)) = \geschw{0}$. Using $\rank(\hat{V}_{1, X}(y)) = kp$, a consequence of $y \notin N(\hat{\Omega}_{\kappa, M, p, X})$ (cf. Remark \ref{WDPSI}), this implies
\begin{equation}
\rank\left((\hat{V}'_{1, X}(y):\hat{V}'_{p, X}(y))\right) = (p+1)k.
\end{equation}
But this is impossible, because the matrix $(\hat{V}'_{1, X}(y):\hat{V}'_{p, X}(y))$ is $(n-p) \times ((p+1)k)$ dimensional, which together with $n < (p+1)k + p$ implies $\rank(\hat{V}'_{1, X}(y):\hat{V}'_{p, X}(y)) \leq n-p < (p+1)k$. 

Next, we prove Part 2 of the lemma. Under the present assumptions it is shown in Part 1 of Proposition \ref{generic} that for $\lambda_{\Rel^{n \times k}}$ almost every $X \in \mathfrak{X}_0$ we have $g^*_{\kappa, M, p}(e_+, X, R) \neq 0$ and therefore in particular $g^*_{\kappa, M, p}(., X, R) \not \equiv 0$. This proves the first statement. To show the remaining statement we construct a $y$ such that $g^*_{\kappa, M, p}(y, e_+, R) \neq 0$. Note first that  $2p+1+ \mathbf{1}_{\mathbb{M}_{AM}}(M) \leq n$ obviously implies $2p+1 \leq n$. Let $y \in \Rel^n$ satisfy $y_1 = -1$, $y_{p+2} = 1$ and $y_i = 0$ else. [Note that this is feasible, i.e., $p+2 \leq n$ holds, because of $2p+1 \leq n$ and $p \geq 1$.] We intend to apply Part 2 of Lemma \ref{AUXCONSTR} with $t = k = 1$ to the tuple $(y, e_+)$. We first have to show that the tuple $(y, e_+)$, which is clearly an element of $\Rel^n \times \mathfrak{X}_0$, satisfies Assumptions (A1), (A2) and (A3). For this we observe that $e_+ \bot y$, which implies $\hat{u}_{e_{+}}(y) = y$ and therefore
\begin{equation}
\hat{V}_{e_+}(y) = \hat{u}'_{e_{+}}(y) = y'.
\end{equation}
Hence (A1) is satisfied, because $y_1 = -1 \neq 0$ and $y$ has only two nonzero coordinates. The corresponding indices are $j_1 = 1$ and $j_{t+1} = p+2$. The first part of Assumption (A2) is therefore obviously satisfied. The second part, i.e., $n- j_{t+1} = n- (p+2) \geq p-1$, follows immediately from $2p+1 \leq n$. Assumption (A3) follows from $y \neq 0$ together with the previous display and $t = k$. Therefore, $\hat{Z}_{e_+}(y) = \hat{V}_{p, e_+}(y) = (0, 1, 0, \hdots 0) \in \Rel^{n-p}$ follows as an application of Part 1 of Lemma \ref{AUXCONSTR}. Obviously (CKV) holds if $M \in \mathbb{M}_{KV}$. Since $\hat{Z}_{e_+}(y) =  (0, 1, 0, \hdots 0)$ it is also obvious that (CNW) holds if $M \in \mathbb{M}_{NW}$. Suppose now that $M \in \mathbb{M}_{AM}$ holds. In this case $\mathbf{1}_{\mathbb{M}_{AM}}(M) = 1$ and therefore $2(p+1) \leq n$ holds. The latter implies $n-(p+2) = n-j_{t+1} > p-1$. Consequently the statement in brackets in (CAM) shows that the condition is satisfied. 

It remains to prove Part 3. Under the present assumptions it is shown in Part 2 of Proposition \ref{generic} that for $\lambda_{\Rel^{n \times (k-1)}}$ almost every $\tilde{X} \in \mathfrak{\tilde{X}}_0$ we have $g^*_{\kappa, M, p}(e_-, (e_+, \tilde{X}), R) \neq 0$ and therefore $g^*_{\kappa, M, p}(., (e_+, \tilde{X}), R) \not \equiv 0$. 
\end{proof}

\section{Proofs of Results in Section \ref{neg}} \label{Aneg}

For a definition of the group $G\left(\mathfrak{M}_{0}\right)$ appearing in the following lemma we refer the reader to \cite{PP13} Section 5.1.

\begin{lemma}\label{A567PW}
Assume that the triple $\kappa$, $M$, $p$ satisfies Assumption \ref{weightsPWRB}. Assume further that $g^*_{\kappa, M, p}(., X, R) \not \equiv 0$. Then, $\hat{\beta}$ and $\hat{\Omega}_{\kappa, M, p}$ satisfy Assumptions 5, 6, and 7 in \cite{PP13} with $N=N(\hat{\Omega}_{\kappa, M, p})$. In fact, $\hat{\Omega}_{\kappa, M, p} \left( y\right) $ is nonnegative definite for every $y \in \mathbb{R}^{n}\backslash N(\hat{\Omega}_{\kappa, M, p})$, and is positive definite $\lambda_{\mathbb{R}^{n}}$-almost everywhere. The test statistic $T$ defined in Equation \eqref{tslrv}, with $\hat{\Omega} = \hat{\Omega}_{\kappa, M, p}$, is invariant under the group $G\left(\mathfrak{M}_{0}\right)$ and the rejection probabilities $P_{\mu ,\sigma ^{2}\Sigma }(T\geq C)$ depend on $\left( \mu ,\sigma ^{2},\Sigma \right) \in \mathfrak{M}\times (0,\infty )\times \mathfrak{C}$ only through $\left( \left( R\beta -r\right)/\sigma ,\Sigma \right)$ (in fact, only through $\left( \left\langle \left(R\beta -r\right) /\sigma \right\rangle ,\Sigma \right)$), where $\beta $ corresponds to $\mu $ via $\mu =X\beta $.
\end{lemma}

\begin{proof}[Proof of Lemma \ref{A567PW}]
The assumption $g^*_{\kappa, M, p}(., X, R) \not \equiv 0$ together with Part 5 of Lemma \ref{N*PW} implies that the algebraic set $N^*(\hat{\Omega}_{\kappa, M, p})$ is a closed $\lambda_{\Rel^n}$-null set. Therefore, by Lemma \ref{NPWRB}, it follows that the algebraic set $N(\hat{\Omega}_{\kappa, M, p}) \subseteq N^*(\hat{\Omega}_{\kappa, M, p})$ is a closed $\lambda_{\Rel^n}$-null set as well. We claim that $\hat{\Omega}_{\kappa, M, p}$ is continuous (and obviously well defined by definition) on $\Rel^n \backslash N(\hat{\Omega}_{\kappa, M, p})$, because it can be written as a composition of continuous functions on this set: We recall from Equation \eqref{QREPr} the representation 
\begin{equation}
\hat{\Omega}_{\kappa, M, p}(y) = \frac{n}{n-p} B_p(y)\mathcal{W}_{n-p}(y) B_p'(y) \text{ for every } y \in \Rel^n \backslash N(\hat{\Omega}_{\kappa, M, p}).
\end{equation}
We first observe that $B_p(.)$ (which was defined in Equation \eqref{defB}) is continuous on $\Rel^n \backslash N(\hat{\Omega}_{\kappa, M, p})$, because $\hat{V}(.)$ and hence $\hat{A}^{(p)}(.)$ and $\hat{Z}(.)$ are continuous on $\Rel^n \backslash N(\hat{\Omega}_{\kappa, M, p})$. By considering each of the cases $M \in \mathbb{M}_{KV}$, $M \in \mathbb{M}_{NW}$ and $M \in \mathbb{M}_{AM}$ separately, it is easy to see that $M(.)$ is continuous on $\Rel^n \backslash N(\hat{\Omega}_{\kappa, M, p})$. The main diagonal entries of $\mathcal{W}_n(.)$ are by definition constant on $\Rel^n \backslash N(\hat{\Omega}_{\kappa, M, p})$. Therefore, it remains to show that all off-diagonal entries are continuous on $\Rel^n \backslash N(\hat{\Omega}_{\kappa, M, p})$. Each of them is of the form $\kappa(i/M(y))$ for some fixed $|i| = 1, \hdots, n-p-1$ if $M(y) \neq 0$, and $0$ if $M(y) = 0$. Since $\kappa$ is a continuous function by Assumption \ref{weightsPWRB}, and $M(.)$ is continuous on $\Rel^n \backslash N(\hat{\Omega}_{\kappa, M, p})$ and satisfies $M(y) \geq 0$, it remains to check that $\kappa(x) \rightarrow 0$ as $|x| \rightarrow \infty$, which is a part of Assumption \ref{weightsPWRB}. This proves the claim. Since $\hat{\beta}$ is well defined and continuous everywhere on $\Rel^n$, it follows that both $\hat{\beta}$ and $\hat{\Omega}_{\kappa, M, p}$ are well-defined and continuous on $\Rel^n \backslash N(\hat{\Omega}_{\kappa, M, p})$. Clearly, $\hat{\Omega}_{\kappa, M, p}$ is symmetric on $\Rel^n \backslash N(\hat{\Omega}_{\kappa, M, p})$. This proves Part (i) of Assumption 5 in \cite{PP13}. To prove the second part let $y \in \Rel^n \backslash N(\hat{\Omega}_{\kappa, M, p})$, $\alpha \neq 0$ and $\gamma \in \Rel^k$. We have to show that $\alpha y + X \gamma \in \Rel^n \backslash N(\hat{\Omega}_{\kappa, M, p})$. Note that $\hat{V}(\alpha y + X\gamma) = X'\diag(\hat{u}(\alpha y + X\gamma)) = \alpha \hat{V}(\hat{u}(y))$, which implies $\hat{A}^{(p)}(\alpha y + X \gamma) = \hat{A}^{(p)}(\hat{u}(y))$ and $\hat{Z}(\alpha y + X \gamma) = \alpha \hat{Z}(y)$. The latter immediately leads (considering each of the cases $M \in \mathbb{M}_{KV}$, $M \in \mathbb{M}_{NW}$ and $M \in \mathbb{M}_{AM}$ separately) to $M(\alpha y + X \gamma) = M(y)$, which in turn implies $\mathcal{W}_{n-p}(\alpha y + X \gamma) = \mathcal{W}_{n-p}(y)$. It then follows from the previous display and the definition of $B_p(y)$ that $\hat{\Omega}_{\kappa, M, p}(\alpha y + X\gamma) =  \alpha^2 \hat{\Omega}_{\kappa, M, p}(y)$. Therefore, we clearly have $\alpha y + X\gamma \in \Rel^n \backslash N(\hat{\Omega}_{\kappa, M, p})$, which proves Part (ii) of Assumption 5 in \cite{PP13}, and where we have also established the equivariance property of $\hat{\Omega}_{\kappa, M, p}$ required in Part (iii) of Assumption 5 in \cite{PP13}. That $\hat{\beta}$ satisfies the equivariance property in Part (iii) of Assumption 5 in \cite{PP13} is obvious. It remains to show that $\hat{\Omega}_{\kappa, M, p}$ is $\lambda_{\Rel^n}$-almost everywhere nonsingular on $\Rel^n \backslash N(\hat{\Omega}_{\kappa, M, p})$. This is equivalent to 
\begin{equation}
\geschw{y \in \Rel^n \backslash N(\hat{\Omega}_{\kappa, M, p}): \det(\hat{\Omega}_{\kappa, M, p}(y)) = 0} = N^*(\hat{\Omega}_{\kappa, M, p}) \backslash N(\hat{\Omega}_{\kappa, M, p})
\end{equation}
being a $\lambda_{\Rel^n}$-null set. This is obvious, since we have already observed that $N^*(\hat{\Omega}_{\kappa, M, p})$ is a $\lambda_{\Rel^n}$-null set under the maintained assumptions. This proves the claim concerning Assumption 5. That $\hat{\Omega}_{\kappa, M, p}(y)$ is nonnegative definite for every $y \notin N(\hat{\Omega}_{\kappa, M, p})$ (which is equivalent to $g_{\kappa, M, p}(y, X) \neq 0$ by Lemma \ref{NPWRB}) has been shown in Part 1 of Lemma \ref{N*PW}. It follows that $\hat{\Omega}_{\kappa, M, p}$ is positive definite on the complement of $N^*(\hat{\Omega}_{\kappa, M, p})$. Hence $\hat{\Omega}_{\kappa, M, p}$ is $\lambda_{\Rel^n}$- almost everywhere positive definite. This immediately shows that Assumptions 6 and 7 in \cite{PP13} are satisfied. The remaining two claims in the lemma now follow immediately from what has been established together with Lemma 5.15 Part 3 and Proposition 5.4 in \cite{PP13}.
\end{proof}
\begin{proof}[Proof of Theorem \ref{thmlrvPW}]
In each part of the theorem we have $g^*_{\kappa, M, p}(., X, R) \not \equiv 0$. In Part 4 this is an explicit assumption. In the other parts this is implied by the assumption that $g^*_{\kappa, M, p}(., X, R) $ does not vanish at a specific point. As a consequence Lemma \ref{A567PW} is applicable in all parts of the theorem. We shall now apply the first two parts of Corollary 5.17 in \cite{PP13} to prove the first two parts of the present theorem. Lemma \ref{A567PW} shows that $\hat{\beta}$ and $\hat{\Omega}_{\kappa, M, p}$ satisfy Assumption 5 in \cite{PP13} with $N = N(\hat{\Omega}_{\kappa, M, p})$. Furthermore, note that the set $N^*$ figuring Corollary 5.17 of \cite{PP13} coincides with $N^*(\hat{\Omega}_{\kappa, M, p})$. By Assumption \ref{AAR(1)} the spaces $\mathcal{Z}_{+}=\text{span}(e_{+})$ and $\mathcal{Z}_{-}=\text{span}(e_{-})$ are concentration spaces of $\mathfrak{C}$ (cf. Lemma G.1 in \cite{PP13}). Hence, Parts 1 and 2 of the present theorem now follow by applying the first two parts of Corollary 5.17 in \cite{PP13} and Remark 5.18(i) in \cite{PP13} to $\mathcal{Z}_{+}$ as well as to $\mathcal{Z}_{-}$, and by noting that Part 5 of Lemma \ref{N*PW} shows that the statement $e_{+} \in \mathbb{R}^{n}\backslash N^*(\hat{\Omega}_{\kappa, M, p})$ translates into $g_{\kappa, M, p}^*(e_{+}, X, R) \neq 0$, with a similar translation if $e_{+} $ is replaced by $e_{-}$. That the test is biased in Part 2 of the theorem follows immediately from Part 5 of Lemma 5.15 in \cite{PP13} (note that Assumptions 5 and 6 in \cite{PP13} are satisfied by Lemma \ref{A567PW}) showing that $W(C)$ contains a non-empty open set. 
To prove Part 4 we apply Theorem 5.19 in \cite{PP13}. Lemma \ref{A567PW} shows that $\hat{\beta}$ and $\hat{\Omega}_{\kappa, M, p}$ also satisfy Assumption 7 in \cite{PP13}. We consider the case where $e_+ \in \mathfrak{M}$ and $R\hat{\beta}(e_+) \neq 0$. The other case can be handled analogously. From Remark 5.20 in \cite{PP13} we see that all conditions on the covariance model in Theorem 5.19 in \cite{PP13} are satisfied with $\bar{\Sigma} = e_+ e_+'$, $\lspan(\bar{\Sigma}) = \lspan(e_+)$ and $Z = e_+$. Clearly, $\lspan(\bar{\Sigma}) = \lspan(e_+) \subseteq \mathfrak{M}$ and $R\hat{\beta}(z) \neq 0$ holds $\lambda_{\lspan(\bar{\Sigma}))}$-a.e. This shows that Equation (33) in \cite{PP13} holds in the present setup. To conclude, it remains to observe that $K_2$ in this equation equals one. This follows from the discussion preceding Theorem 5.19 in \cite{PP13}, because $\hat{\Omega}_{\kappa, M, p}$ is almost everywhere positive definite by Lemma \ref{A567PW}.

\smallskip

Now we consider Part 3 of the theorem. We prove the case where $g^*_{\kappa,M, p}(e_{+}, X, R) \neq 0$, $T(e_{+}+\mu_{0}^{\ast }) = C$ and $\grad T(e_{+} + \mu_0^{\ast})$ exists for some $\mu _{0}^{\ast }\in \mathfrak{M}_{0}$. The other case works analogously. The statement in the theorem saying that if $\grad T(e_{+} + \mu_0^{\ast})$ exists and $T(e_{+}+\mu_{0}^{\ast }) = C$ holds for \textit{some} $\mu _{0}^{\ast }\in \mathfrak{M}_{0}$, then $\grad T(e_{+} + \mu_0^{\ast})$ exists and $T(e_{+}+\mu_{0}^{\ast }) = C$ holds for \textit{all} $\mu _{0}^{\ast }\in \mathfrak{M}_{0}$ follows at once from invariance of $T$ w.r.t. $G(\mathfrak{M}_0)$, which holds as a consequence of Lemma \ref{A567PW}. In a first step we now show that the linear functional on $\Rel^n$ corresponding to the row vector $\grad T (\mu_0^{\ast} + e_{+})$ does not vanish everywhere on $\lspan(e_{+})^{\bot}$: Arguing by contradiction, assume that $\grad T (\mu_0^{\ast} + e_{+})w = 0$ for every $w \in \lspan(e_{+})^{\bot}$, which is equivalent to $\grad T (\mu_0^{\ast} + e_{+})' \bot \lspan(e_{+})^{\bot}$ and therefore $\grad T (\mu_0^{\ast} + e_{+})' \in \lspan(e_+)$ holds, i.e., $\grad T (\mu_0^{\ast} + e_{+}) = ce_+'$ for some $c \in \Rel$. Since $T$ is $G(\mathfrak{M}_0)$ invariant, it holds for every $\gamma \neq 0$ that
\begin{equation}
T(\gamma e_+ + \mu_0^*) = T(\gamma(e_+ + \mu_0^* - \mu_0^*) + \mu_0^*)  = T(e_+ + \mu_0^*) = C.
\end{equation}
Hence on the set $\Rel \backslash \geschw{-1}$ the mapping
\begin{equation}
\alpha \mapsto T(e_+ + \mu_0^* + \alpha e_+) = T((1+\alpha)e_+ + \mu_0^*) =  C
\end{equation}
is constant, thus showing that the directional derivative of $T$ at the point $e_+ + \mu_0^*$ in direction $e_+$ is zero. The latter is equivalent to $\grad T (\mu_0^{\ast} + e_{+}) e_+ = c \norm{e_+}^2 = 0$, and hence $c = 0$ holds which implies $\grad T (\mu_0^{\ast} + e_{+}) = 0$. To arrive at a contradiction it remains to show that there is a vector $v$ such that the directional derivative of $T$ at $e_+ + \mu_0$ in direction $v$ does not vanish. To this end, recall that Assumption 5 in \cite{PP13} is satisfied, hence the discussion following that Assumption in \cite{PP13} shows that $N^*(\hat{\Omega}_{\kappa, M, p})$ is invariant w.r.t. $G(\mathfrak{M})$. Therefore, $e_+ \notin N^*(\hat{\Omega}_{\kappa, M, p})$ implies $e_{+} + \mu_0^* \notin N^*(\hat{\Omega}_{\kappa, M, p})$. Since $N^*(\hat{\Omega}_{\kappa, M, p})$ is closed by Lemma \ref{N*PW}, there exists an open ball $U_{\varepsilon}$ of radius $\varepsilon > 0$ centered at $e_{+} + \mu_0^*$ such that $U_{\varepsilon} \subseteq \Rel^n \backslash N^*(\hat{\Omega}_{\kappa, M, p})$. Additionally, we note that $e_{+} \notin \mathfrak{M}$, because $\mathfrak{M} \subseteq N^*(\hat{\Omega}_{\kappa, M, p})$ always holds (see the discussion in \cite{PP13} after Assumption 5). Therefore, $v = \Pi_{\mathfrak{M}^{\bot}} e_{+}/\norm{\Pi_{\mathfrak{M}^{\bot}} e_{+}}$ is well defined and for $0 \leq |\alpha| < \varepsilon$ we have $e_{+} + \mu_0^* + \alpha v \in U_{\varepsilon}$. Assume that $0 \leq |\alpha| < \varepsilon$. The OLS estimator $\hat{\beta}$ clearly satisfies
\begin{equation}
R\hat{\beta}(e_{+} + \mu_0^* + \alpha v) = R\hat{\beta}(e_{+} + \mu_0^*)  + \alpha R\hat{\beta}(v) =  R\hat{\beta}(e_{+} + \mu_0^*),
\end{equation}
where the second equality follows from $v \bot \mathfrak{M}$. Since $\hat{\Omega}_{\kappa, M, p}$ satisfies the equivariance condition in Assumption 5 of \cite{PP13} we can furthermore write
\begin{align}
\hat{\Omega}_{\kappa, M, p}(e_{+} + \mu_0^* + \alpha v) &= \hat{\Omega}_{\kappa, M, p}(e_{+} + \alpha v) \\
&= \hat{\Omega}_{\kappa, M, p}(e_+ - \Pi_{\mathfrak{M}}e_+ + \alpha v ) \\
&= \hat{\Omega}_{\kappa, M, p}(\Pi_{\mathfrak{M}^{\bot}}e_+ + \alpha \Pi_{\mathfrak{M}^{\bot}} e_{+}/\norm{\Pi_{\mathfrak{M}^{\bot}} e_{+}} ) \\
&= \hat{\Omega}_{\kappa, M, p}\left((1 + \frac{\alpha}{\norm{\Pi_{\mathfrak{M}^{\bot}} e_{+}}}) \Pi_{\mathfrak{M}^{\bot}}e_+ \right) \\
&= (1 + \frac{\alpha}{\norm{\Pi_{\mathfrak{M}^{\bot}} e_{+}}})^2 \hat{\Omega}_{\kappa, M, p}\left( \Pi_{\mathfrak{M}^{\bot}}e_+ \right) \\
&= (1 + \frac{\alpha}{\norm{\Pi_{\mathfrak{M}^{\bot}} e_{+}}})^2 \hat{\Omega}_{\kappa, M, p}\left( e_+ + \mu_0^* \right).
\end{align}
By definition of $T$ (cf. Equation \eqref{tslrv} and recall that $e_{+} + \mu_{0}^{\ast } + \alpha v \in U_{\varepsilon} \subseteq \Rel^n \backslash N^*(\hat{\Omega}_{\kappa, M, p})$) and the assumed equality $T(e_{+}+\mu_{0}^{\ast }) = C$, the relations derived above allow us to show that
\begin{align}
T(e_{+} + \mu_{0}^{\ast } + \alpha v) &= \left( R\hat{\beta}(e_{+} + \mu_{0}^{\ast } + \alpha v) - r\right)' \hat{\Omega}^{-1}_{\kappa, M, p}(e_{+} + \mu_{0}^{\ast } + \alpha v) \left( R\hat{\beta}(e_{+} + \mu_{0}^{\ast } + \alpha v) - r\right) \\
&= (1 + \frac{\alpha}{\norm{\Pi_{\mathfrak{M}^{\bot}} e_{+}}})^{-2} \left( R\hat{\beta}(e_{+} + \mu_{0}^{\ast }) - r\right)' \hat{\Omega}^{-1}_{\kappa, M, p}(e_{+} + \mu_0^*) \left( R\hat{\beta}(e_{+} + \mu_{0}^{\ast}) - r \right) \\
&= (1 + \frac{\alpha}{\norm{\Pi_{\mathfrak{M}^{\bot}} e_{+}}})^{-2} T(e_{+} + \mu_{0}^{\ast}) \\
&= (1 + \frac{\alpha}{\norm{\Pi_{\mathfrak{M}^{\bot}} e_{+}}})^{-2} C
\end{align} 
holds for every $0 \leq |\alpha| < \varepsilon$. This implies that the directional derivative of $T$ in direction $v$ at the point $e_{+} + \mu_{0}^{\ast }$ equals $-2C/\norm{\Pi_{\mathfrak{M}^{\bot}} e_{+}}$, which is nonzero as a consequence of $C > 0$. In a second step we shall now derive an expansion of $T$ at points of the form $y + \mu_0^*$ for $y$ satisfying $e_+'y \neq 0$: For every $h \in \Rel^n$ we have
\begin{equation}\label{Taylor}
T(e_{+} + \mu_0^* + h) = T(e_{+} + \mu_0^*) + \grad T (e_{+} + \mu_0^*) h + Q(h)
\end{equation}
where $Q(h)/\norm{h} \rightarrow 0$ as $h \rightarrow 0$ and $h \neq 0$. Recall that $T$ is invariant under the group $G(\mathfrak{M}_0)$. In particular for every $y$ such that $e_{+}'y \neq 0$ we have
\begin{equation}
T(y + \mu_0^*) = T(\frac{e_{+}'y}{n} e_{+} + \mu_0^*  + \Pi_{\lspan(e_{+})^{\bot}}y) = T(e_{+} + \mu_0^* + \frac{n}{e_{+}'y} \Pi_{\lspan(e_{+})^{\bot}}y), 
\end{equation}
where the first equality holds because of $y = \Pi_{\lspan(e_{+})}y + \Pi_{\lspan(e_{+})^{\bot}}y$ and the second follows from invariance of $T$ w.r.t. $G(\mathfrak{M}_0)$. This means that whenever $e_+'y \neq 0$ holds, we can combine the equation in the previous display and Equation \eqref{Taylor} with $h = \frac{n}{e_{+}'y} \Pi_{\lspan(e_{+})^{\bot}}y$ to see that
\begin{equation}\label{Taylor2}
T(y + \mu_0^*) = T(e_{+} + \mu_0^*) + \frac{n}{e_{+}'y} \grad T (e_{+} + \mu_0^*) \Pi_{\lspan(e_{+})^{\bot}}y + Q(\frac{n}{e_{+}'y} \Pi_{\lspan(e_{+})^{\bot}}y)
\end{equation}
holds and that 
\begin{equation}\label{Taylor3}
Q(\frac{n}{e_{+}'y_m} \Pi_{\lspan(e_{+})^{\bot}}y_m)/\norm{Q(\frac{n}{e_{+}'y_m} \Pi_{\lspan(e_{+})^{\bot}}y_m)} \rightarrow 0,
\end{equation}
for any sequence $y_m$ satisfying $e_{+}'y_m \neq 0$, $\frac{n}{e_{+}'y_m} \Pi_{\lspan(e_{+})^{\bot}}y_m  \rightarrow 0$ and $\frac{n}{e_{+}'y_m} \Pi_{\lspan(e_{+})^{\bot}}y_m \neq 0$. Now, we choose a sequence $\rho_m \in (-1, 1)$ such that $\rho_m \rightarrow 1$ and apply Assumption \ref{AAR(1)} to obtain $\Lambda(\rho_m) \in \mathfrak{C}$ for every $m$. We intend to show that $P_{\mu_0^*, \Lambda(\rho_m)}(W(C)) \rightarrow 1/2$ along a subsequence. The last statement in Lemma \ref{A567PW} then implies $P_{\mu_0, \sigma^2 \Lambda(\rho_m)}(W(C)) \rightarrow 1/2$ for every pair $\mu_0 \in \mathfrak{M}_0$ and $0 < \sigma^2 < \infty$ along this subsequence. In Part 3 of Lemma G.1 in \cite{PP13} it is shown that $\Lambda(\rho_m) \rightarrow e_{+} e_{+}'$, $D_m = \Pi_{\lspan(e_{+})^{\bot}} \Lambda(\rho_m) \Pi_{\lspan(e_{+})^{\bot}}/s_m \rightarrow D$, where $s_m$ is a sequence of numbers such that $s_m > 0$, $s_m \rightarrow 0$ and $D$ is regular on $\lspan(e_{+})^{\bot}$. Furthermore, it is shown that $\Pi_{\lspan(e_{+})^{\bot}} \Lambda(\rho_m) \Pi_{\lspan(e_{+})}/s_m^{1/2} \rightarrow 0$. We can use these relations to derive three useful facts: (i) Observe that the matrix $s_m^{-1/2} \Pi_{\lspan(e_{+})^{\bot}} \Lambda(\rho_m)^{1/2}$ is an $n \times n$-dimensional nonnegative square root of the symmetric matrix $D_m$. Therefore, we can find an orthogonal matrix $U_m$ such that
\begin{equation}
s_m^{-1/2} \Pi_{\lspan(e_{+})^{\bot}} \Lambda(\rho_m)^{1/2} U_m = D_m^{1/2}
\end{equation}
holds. The sequence $D_m^{1/2}$ converges to $D^{1/2}$ as a consequence of $D_m \rightarrow D$ together with the continuity of taking the nonnegative definite symmetric matrix square root of a symmetric and nonnegative definite matrix. Since $U_m$ is orthogonal we can choose a subsequence $m'$ along which $U_m$ converges to $U$, say. Without loss of generality we henceforth assume $m' \equiv m$. Using the relation in the previous display we find that $s_m^{-1/2} \Pi_{\lspan(e_{+})^{\bot}} \Lambda(\rho_m)^{1/2}$ converges to 
\begin{equation}\label{Drep}
D^*  = D^{1/2}U',
\end{equation}
and we recall from above that $D^{1/2} \in \Rel^{n \times n}$ is regular on $\lspan(e_{+})^{\bot}$. (ii) We note that $\Lambda(\rho_m) \rightarrow e_+ e_+'$ implies 
\begin{equation}
\Lambda(\rho_m)^{1/2} \rightarrow n^{-1/2} e_{+}e_{+}'.
\end{equation}
(iii) We show that $D^* e_+ = 0$ must hold: Note that $\Pi_{\lspan(e_{+})^{\bot}} \Lambda(\rho_m) \Pi_{\lspan(e_{+})}/s_m^{1/2} \rightarrow 0$ can be rewritten as 
\begin{equation}
\left(s_m^{-1/2}\Pi_{\lspan(e_{+})^{\bot}} \Lambda(\rho_m)^{1/2}\right) \Lambda(\rho_m)^{1/2} \Pi_{\lspan(e_{+})} \rightarrow 0,
\end{equation}
that the term in brackets converges to $D^*$ by (i) and that the other term converges to $n^{-1/2}e_{+}e_{+}'$ by (ii). Therefore,
\begin{equation}
D^* n^{-1/2}e_{+}e_{+}' = 0,
\end{equation}
or equivalently $D^* \Pi_{\lspan(e_{+})} = 0$. But this implies $D^* e_{+} = 0$. Now, we are ready to show that $P_{\mu_0^*, \Lambda(\rho_m)}(W(C)) \rightarrow 1/2$. Let $\mathbf{G}$ be a random n-vector defined on some underlying probability space such that the probability measure induced by $\mathbf{G}$ on $(\Rel^n, \mathcal{B}(\Rel^n))$ equals $P_{0, I_n}$. Consequently, the random vector $\Lambda(\rho_m)^{1/2} \mathbf{G} + \mu_0^*$ induces the distribution $P_{\mu_0^*, \Lambda(\rho_m)}$ on $(\Rel^n, \mathcal{B}(\Rel^n))$. For notational convenience we write $\mathbf{G}_m = \Lambda(\rho_m)^{1/2} \mathbf{G}$. Consequently, we have
\begin{align}\label{Prob12}
P_{\mu_0^*, \Lambda(\rho_m)}\left(W(C)\right) &= \Pr(T(\mathbf{G}_m+ \mu_0^*) \geq C)\\
&= \Pr\left(s_m^{-1/2} \left[T(\mathbf{G}_m + \mu_0^*) - T(e_{+} + \mu_0^*)\right] \geq 0\right),
\end{align}
where we used $s_m^{-1/2} > 0$ and  $T(e_{+} + \mu_0^*) = C$ in deriving the second equality. Note that $e_+'\mathbf{G}_m \neq 0$ and $\norm{\frac{n}{e_{+}'\mathbf{G}_m} \Pi_{\lspan(e_{+})^{\bot}}\mathbf{G}_m} > 0$ on an event of probability one. Using the expansion developed in Equation \eqref{Taylor2} we see that with probability one $s_m^{-1/2} \left[T(\mathbf{G}_m + \mu_0^*) - T(e_{+} + \mu_0^*) \right]$ can be written as
\begin{align}
&s_m^{-1/2} \left[ \frac{n}{e_{+}'\mathbf{G}_m} \grad T (e_{+} + \mu_0^*) \Pi_{\lspan(e_{+})^{\bot}}\mathbf{G}_m + Q\left(\frac{n}{e_{+}'\mathbf{G}_m} \Pi_{\lspan(e_{+})^{\bot}}\mathbf{G}_m\right)\right] \\
= & ~~ \frac{n}{e_{+}'\mathbf{G}_m} \grad T (e_{+} + \mu_0^*) s_m^{-1/2} \Pi_{\lspan(e_{+})^{\bot}}\mathbf{G}_m + \norm{\frac{n}{e_{+}'\mathbf{G}_m} s_m^{-1/2} \Pi_{\lspan(e_{+})^{\bot}}\mathbf{G}_m} \\ & ~~~~~~~ \times  ~~~~~ \norm{\left(\frac{n}{e_{+}'\mathbf{G}_m} \Pi_{\lspan(e_{+})^{\bot}}\mathbf{G}_m\right)}^{-1} Q\left(\frac{n}{e_{+}'\mathbf{G}_m} \Pi_{\lspan(e_{+})^{\bot}}\mathbf{G}_m\right).
\end{align}
To derive the almost sure limit as $m \rightarrow \infty$ of the expression in the previous display we first observe that $\mathbf{G}_m$ converges point-wise to $n^{-1/2}e_+ e_+' \mathbf{G}$ because of (ii). From that it follows that $e_+' \mathbf{G}_m$ converges point-wise to $\sqrt{n} e_+'\mathbf{G}$ and that $\Pi_{\lspan(e_+)^{\bot}} \mathbf{G}_m$ converges point-wise to zero. An application of the continuous mapping theorem hence shows that
\begin{equation}
\frac{n}{e_{+}'\mathbf{G}_m} \Pi_{\lspan(e_{+})^{\bot}}\mathbf{G}_m \rightarrow 0
\end{equation}
almost surely as $m \rightarrow \infty$, which immediately implies
\begin{equation}
\norm{\left(\frac{n}{e_{+}'\mathbf{G}_m} \Pi_{\lspan(e_{+})^{\bot}}\mathbf{G}_m\right)}^{-1} Q\left(\frac{n}{e_{+}'\mathbf{G}_m} \Pi_{\lspan(e_{+})^{\bot}}\mathbf{G}_m\right) \rightarrow 0
\end{equation}
almost surely as $m \rightarrow \infty$ as a consequence of Equation \eqref{Taylor3} together with $Q(0) = 0$. We also observe that (i) above implies
\begin{equation}
\Pi_{\lspan(e_{+})^{\bot}}s_m^{-1/2}\mathbf{G}_m \rightarrow  D^* \mathbf{G}
\end{equation}
point-wise and thus, using the continuous mapping theorem again, we see that
\begin{equation}
\frac{n}{e_{+}'\mathbf{G}_m} \Pi_{\lspan(e_{+})^{\bot}}s_m^{-1/2}\mathbf{G}_m \rightarrow \frac{\sqrt{n}}{e_{+}'\mathbf{G}} D^* \mathbf{G}
\end{equation}
almost surely as $m \rightarrow \infty$ (where the limiting random vector is well defined almost-surely). This finally shows that 
\begin{equation}
s_m^{-1/2} \left[T(\mathbf{G}_m + \mu_0^*) - T(e_{+} + \mu_0^*) \right] \rightarrow \frac{\sqrt{n}}{e_{+}' \mathbf{G}} \grad T (e_{+} + \mu_0^*) D^* \mathbf{G},
\end{equation}
almost surely. We already know from Equation \eqref{Drep} that $D^* = D^{1/2} U'$, where $U$ is an orthogonal matrix. Furthermore $D^{1/2}$ maps $\Rel^n$ onto $\lspan(e_{+})^{\bot}$, and $\grad T (e_{+} + \mu_0^*)$ does not vanish everywhere on $\lspan(e_{+})^{\bot}$. Hence, we see that the probability that the limiting random variable in the previous display takes on the value $0$ vanishes because $\grad T (e_{+} + \mu_0^*) D^* \mathbf{G}$ is a Gaussian random variable with mean zero and positive variance. Hence, Equation \eqref{Prob12} together with Portmanteau theorem shows that
\begin{equation}\label{problim}
P_{\mu_0^*, \Lambda(\rho_m)}\left(W(C)\right) \rightarrow \Pr(\frac{\sqrt{n}}{e_{+}' \mathbf{G}} \grad T (e_{+} + \mu_0^*) D^{*} \mathbf{G} \geq 0).
\end{equation}
The covariance between the Gaussian mean-zero random variables $\grad T (e_{+} + \mu_0^*) D^* \mathbf{G}$ and $e_{+}' \mathbf{G}$ is given by
\begin{equation}
\grad T (e_{+} + \mu_0^*) D^* e_+ = 0,
\end{equation}
where the equality follows from (iii). Therefore, $e_{+}' \mathbf{G}$ and $\grad T (e_{+} + \mu_0^*) D^* \mathbf{G}$ are independent. Since the probability to the right in Equation \eqref{problim} equals the probability that the random variables $e_{+}' \mathbf{G}$ and $\grad T (e_{+} + \mu_0^*) D^* \mathbf{G}$ have the same sign it is now obvious that the limit equals $1/2$.
\end{proof}
\begin{lemma}\label{gradientT}
Assume that the triple $\kappa$, $M$, $p$ satisfies Assumption \ref{weightsPWRB} and let $T$ be as in Equation \eqref{tslrv} with $\hat{\Omega} = \hat{\Omega}_{\kappa, M, p}$. Let $\mu_0 \in \mathfrak{M}_0$. Then the following holds.
\begin{enumerate}
\item If $M \in \mathbb{M}_{KV}$ and $g^*_{\kappa, M, p}(y, X, R) \neq 0$, then $\grad T(\mu_0 + y)$ exists. 
\item Suppose that $M \notin \mathbb{M}_{KV}$, that $\kappa$ is continuously differentiable on the non-void complement of a closed set $\Delta(\kappa) \subseteq \Rel$, and that $g^*_{\kappa, M, p}(y, X, R) \neq 0$. Assume further that one of the following conditions is satisfied: 
\begin{enumerate}
\item $M(y) \neq 0$ and $\frac{i}{M(y)} \notin  \Delta(\kappa)$ for $|i| = 1, \hdots, n-p-1$. 
\item $M(y) = 0$ and $\kappa$ has compact support. 
\end{enumerate}
Then $\grad T(\mu_0 + y)$ exists.
\item If $M \notin \mathbb{M}_{KV}$, then for every $\delta \geq 0$ we have
\begin{align}
&\geschw{y \in \Rel^n: g_{\kappa, M, p}^*(y, X, R) \neq 0 \text{ and } M(y) = \delta } ~ \\ 
= &\geschw{y \in \Rel^n: g^*_{\kappa, M, p}(y, X, R) \neq 0 \text{ and } g^{(\delta)}_{\kappa, M, p}(y, X) = 0},
\end{align}
where $g^{(\delta)}_{\kappa, M, p}: \Rel^n \times \Rel^{n \times k} \rightarrow \Rel$ is a multivariate polynomial (explicitly constructed in the proof) that does not depend on the hypothesis $(R,r)$. 
\end{enumerate}
\end{lemma}
\begin{proof}
We first verify Parts 1 and 2. Let us start by deriving a convenient expression for $T(\mu_0 + y)$ under the assumption $g^*_{\kappa, M, p}(y, X, R) \neq 0$. By Lemma \ref{N*PW} the assumption $g^*_{\kappa, M, p}(y, X, R) \neq 0$ is equivalent to $y \notin N^*(\hat{\Omega}_{\kappa, M, p})$. An application of Lemma \ref{A567PW} shows that $\hat{\Omega}_{\kappa, M, p}$ satisfies Assumption 5 in \cite{PP13}. An application of Part (ii) of this Assumption shows that $\mu_0 + y \notin N^*(\hat{\Omega}_{\kappa, M, p})$. We can therefore use Equation \eqref{tslrv} together with $R\hat{\beta}(y + \mu_0) - r = R\hat{\beta}(y)$ and $\hat{\Omega}_{\kappa, M, p}(\mu_0 + y) = \hat{\Omega}_{\kappa, M, p}(y)$ (both following from Assumption 5 in \cite{PP13}) to see that $y \notin N^*(\hat{\Omega}_{\kappa, M, p})$ implies
\begin{equation}
T(\mu_0 + y) =  \hat{\beta}(y)'R' \hat{\Omega}^{-1}_{\kappa, M, p}(y) R \hat{\beta}(y) 
= \hat{\beta}(y)'R' \left(\frac{n}{n-p} B_p(y)\mathcal{W}_{n-p}(y) B'_p(y)\right)^{-1} R \hat{\beta}(y),
\end{equation}
where in deriving the second equality we made use of the representation of $\hat{\Omega}_{\kappa, M, p}(y)$ developed in Equation \eqref{QREPr}. The function $\hat{\beta}(.)$ is linear and hence totally differentiable on $\Rel^n$. Furthermore, in the proof of Lemma \ref{N*PW} it is shown that the coordinates of the matrix $B_p(.)$ are multivariate rational functions (without singularities) on $\Rel^n \backslash N^*(\hat{\Omega}_{\kappa, M, p})$. In particular the coordinates of $B_p(.)$ are continuously partially differentiable on $\Rel^n \backslash N^*(\hat{\Omega}_{\kappa, M, p})$. To show that $\grad T(\mu_0 + y)$ exists at a given point $y \in  \Rel^n \backslash N^*(\hat{\Omega}_{\kappa, M, p})$ it is therefore sufficient to show that each off-diagonal element (recall that the diagonal is constant) of $\mathcal{W}_{n-p}(.)$ is continuously partially differentiable on an open neighborhood of $y$ in $\Rel^n \backslash N^*(\hat{\Omega}_{\kappa, M, p})$. Recall that the $i$-th off-diagonal element ($i  \in \geschw{1, \hdots, n-p-1}$) of $\mathcal{W}_{n-p}(.)$ evaluated at some $y \in  \Rel^n \backslash N^*(\hat{\Omega}_{\kappa, M, p})$ is given by
\begin{equation}
f_i(y) := \begin{cases}
\kappa(i/M(y)) & \text{ if } M(y) \neq 0 \\
0 & \text{ else}.
\end{cases}
\end{equation}
If $M \in \mathbb{M}_{KV}$ the sufficient condition above is obviously satisfied, because in this case $M > 0$ is constant and therefore $f_i(.)$ is constant on $\Rel^n \backslash N^*(\hat{\Omega}_{\kappa, M, p})$. This proves Part 1 of the lemma. Consider now Part 2. By considering separately the cases $M \in \mathbb{M}_{AM}$ and $M \in \mathbb{M}_{NW}$, we observe that $M(.)$ is continuously partially differentiable in an open neighborhood of any element $y$ of $\Rel^n \backslash N^*(\hat{\Omega}_{\kappa, M, p})$ satisfying $M(y) \neq 0$. We start with Condition (a). Let $y$ satisfy $y \notin N^*(\hat{\Omega}_{\kappa, M, p})$ and $M(y) \neq 0$. Fix an $i  \in \geschw{1, \hdots, n-p-1}$. By assumption $\kappa$ is continuously differentiable on an open neighborhood of $i/M(y) \notin \Delta(\kappa)$. Hence there exists an open neighborhood $U$ of $y$ in $\Rel^n \backslash N^*(\hat{\Omega}_{\kappa, M, p})$ on which $M(.)$ is strictly greater than zero and such that $\kappa(i/M(.))$ is continuously partially differentiable on $U$. It hence follows that $f_i(.)$ is continuously partially differentiable on $U$ because it coincides with $\kappa(i/M(.))$ on this set. To establish existence of the gradient under Condition (b) let $y$ satisfy $y \notin N^*(\hat{\Omega}_{\kappa, M, p})$ and $M(y) = 0$. Let $i$ be as before and recall that the support of $\kappa$ is compact by assumption. Since $M$ is continuous on $\Rel^n \backslash N^*(\hat{\Omega}_{\kappa, M, p})$, there exists an open neighborhood of $y$ in $\Rel^n \backslash N^*(\hat{\Omega}_{\kappa, M, p})$ such that for every point $y^*$ in this neighborhood we either have $M(y^*) = 0$ or that $i/M(y^*)$ is not contained in the support of $\kappa$. It follows that the function $f_i$ is constant equal to $0$, and thus is in particular continuously partially differentiable, on this neighborhood. 

To prove the third part of the lemma consider first the case $M \equiv M_{AM, 1, \omega} \in \mathbb{M}_{AM}$, where we dropped the index $c$ because the argument and the resulting polynomial do not depend on it. Suppose $y$ satisfies $g^*_{\kappa, M_{AM, 1, \omega}, p}(y, X, R) \neq 0$. Then $M_{AM, 1, \omega}(y)$ is well defined and by definition $M_{AM, 1, \omega}(y) = \delta$ if and only if $\hat{\alpha}_1(y) = n^{-1} (c_1^{-1} \delta)^{1/c_2} =: \delta^*$ holds, where $c_1$ and $c_2$ are positive constants. This can equivalently be written as
\begin{equation}
\sumab{i = 1}{k}\omega_i \frac{4 \hat{\rho}_i^2(y) \hat{\sigma}_i^4(y)}{(1-\hat{\rho}_i(y))^6(1+\hat{\rho}_i(y))^2} = \delta^* \sumab{i = 1}{k}\omega_i \frac{ \hat{\sigma}^4_i(y)}{(1-\hat{\rho}_i(y))^4}, 
\end{equation}
which, after multiplying both sides of the equation by $\prod_{j = 1}^k (1-\hat{\rho}_j(y))^6(1+\hat{\rho}_j(y))^2$ (which is nonzero), is seen to be equivalent to
\begin{align}
&\sumab{i = 1}{k} \omega_i \left[  4 \hat{\rho}_i^2(y) \hat{\sigma}_i^4(y) \prod_{j \neq i}^k (1-\hat{\rho}_j(y))^6(1+\hat{\rho}_j(y))^2 \right]  \\ 
& \hspace{2cm} - ~~ \delta^* \sumab{i = 1}{k}\omega_i  \left[\hat{\sigma}^4_i(y) (1+\hat{\rho}_i(y))^2 (1-\hat{\rho}_i(y))^2 \prod_{j \neq i}^k (1-\hat{\rho}_j(y))^6(1+\hat{\rho}_j(y))^2 \right] = 0 
\end{align} 
By multiplying both sides of this equation by a suitably large power of the products of the denominators of $\hat{\rho}_i(y)$ (which are nonzero), we can write the preceding equation equivalently as
\begin{equation}
\sum_{i = 1}^k \omega_i \bar{p}_i^{(\delta)}(\hat{Z}(y)) = 0
\end{equation}
where each $\bar{p}_i^{(\delta)}: \Rel^{k \times (n-p)} \rightarrow \Rel$ for $i = 1, \hdots, k$ is a multivariate polynomial. In a final step we multiply both sides of the equation by a suitably large power of the non-vanishing factor in Equation \eqref{factor} to obtain an equivalent equation of the form
\begin{equation}
g^{(\delta)}_{\kappa, M_{AM, 1, \omega}, p}(y, X) = \sum_{i = 1}^k \omega_i p_i^{(\delta)}(y, X) = 0,
\end{equation}
where each $p_i^{(\delta)}: \Rel^n \times \Rel^{n \times k} \rightarrow \Rel$ is a multivariate polynomial. Therefore, the condition 
\begin{equation}
g^*_{\kappa, M_{AM, 1, \omega}, p}(y, X) \neq 0 \text{ and } M_{AM, 1, \omega}(y) = \delta
\end{equation}
can be equivalently stated as 
\begin{equation}
g^*_{\kappa, M_{AM, 1, \omega}, p}(y, X, R) \neq 0 \text{ and } g^{(\delta)}_{\kappa, M_{AM, 1, \omega}, p}(y, X) = 0.
\end{equation}
Finally, we note that the multivariate polynomial $g^{(\delta)}_{\kappa, M_{AM, 1, \omega}, p}: \Rel^n \times \Rel^{n \times k} \rightarrow \Rel$ does not depend on the hypothesis $(R,r)$. This proves the last part of the lemma in case $M \equiv M_{AM, 1, \omega}  \in \mathbb{M}_{AM}$. The proof of the case $M \equiv M_{AM, 2, \omega} \in \mathbb{M}_{AM}$ is almost identical and therefore we omit it. We finally note that similar arguments can be used to prove the statement in case $M \in \mathbb{M}_{NW}$, but we omit details.
\end{proof}
\begin{proof}[Proof of Proposition \ref{generic}]
We first prove that the sets $\mathfrak{X}_2(e_+)$, $\mathfrak{X}_2(e_-)$, $\tilde{\mathfrak{X}}_2(e_+)$ and $\tilde{\mathfrak{X}}_2(e_-)$ do not depend on the specific choice of $\mu_{0, X}^* \in \mathfrak{M}_{0, X}$. This follows from an invariance argument. Consider for example the set $\mathfrak{X}_2(e_+)$, which is by definition a subset of $\mathfrak{X}_0 \backslash \mathfrak{X}_1(e_+)$. Every element $X$ of this superset satisfies $g_{\kappa, M, p}(., X, R) \not \equiv 0$. Hence, for every such $X$ the corresponding test statistic $T_X$ is invariant w.r.t. $G(\mathfrak{M}_{0, X})$ by Lemma \ref{A567PW}. It now immediately follows that $\mathfrak{X}_1(e_+)$ does not depend on the specific choice of $\mu_{0, X}^{\ast} \in \mathfrak{M}_{0, X}$. The same argument shows that the statement in Part 2 Condition (d) is independent of the specific choice of $\mu^*_{0, (e_+, \tilde{X})}$. We shall now prove the three main parts of the proposition and start with the first: 

1) We begin with the statement concerning $\mathfrak{X}_{1}\left( e_{+}\right)$. Under the maintained assumptions we know from Part 5 of Lemma \ref{N*PW} that $g_{\kappa, M,p}^*(.,.,.): \Rel^n \times \Rel^{n \times k} \times \Rel^{q \times k} \rightarrow \Rel$ is a multivariate polynomial. This immediately implies that $g_{\kappa, M,p}^*(e_+,.,R): \Rel^{n \times k} \rightarrow \Rel$ is a multivariate polynomial, showing that
\begin{equation}
\geschw{X \in \Rel^{n \times k}: g_{\kappa, M, p}^*(e_+, X,R) = 0}
\end{equation}
is an algebraic superset of $\mathfrak{X}_{1}\left( e_{+}\right)$. It hence suffices to show that the set in the previous display is a $\lambda_{\Rel^{n \times k}}$- null set, or equivalently to show that $g_{\kappa,M,p}^*(e_{+},.,R) \not \equiv 0$. To this end we shall use Lemma \ref{AUXCONSTR} (with $t = k$) to construct a matrix $X \in \mathfrak{X}_0$ such that $g_{\kappa, M, p}^*(e_{+}, X,R) \neq 0$. Let $H \in \Rel^{(k+1) \times k}$ be an auxiliary matrix the column vectors of which span $\lspan(\bar{e}_+)^{\bot}$, where $\bar{e}_+ = (1, \hdots, 1)' \in \Rel^{k+1}$ is the vector obtained from $e_+$ by selecting the coordinates with indices $j_i = 1 + (i-1)(p+1)$ for $i = 1, \hdots, k+1$ (we shall need a similar construction for $e_-$ later on). Note that this selection is feasible because $j_{k+1} = 1 + k(p+1) \leq n$ holds as a consequence of the assumption $n-[k(p+1) + p] - \mathbf{1}_{\mathbb{M}_{AM}}(M) \geq 0$ together with $p \geq 1$. We also note that $H$ does not contain a row consisting of zeros only. If the construction of $M$ involves a weights vector $\omega$ (which is assumed to be  functionally independent of the design) we choose the columns of $H$ in such a way that $H\omega = (-1, 1, 0, \hdots, 0)'$ which is possible because this vector is an element of $\lspan(\bar{e}_+)^{\bot}$ and $\omega \neq 0$ holds. We now let $X \in \Rel^{n \times k}$ be the matrix the non-zero rows of which are precisely $X_{j_i \cdot } = H_{i \cdot}$ for $i = 1, \hdots, k+1$, i.e.,
\begin{equation}\label{Xmatrix}
X = \left(H'_{1\cdot}, 0_{k,p}, H'_{2\cdot}, 0_{k,p}, H'_{3\cdot}, 0_{k,p}, \hdots, H'_{(k+1) \cdot}, 0_{k, n-k(p+1)-1}\right)' \in \Rel^{n \times k},
\end{equation}
where $0_{m_1, m_2}$ denotes the $m_1 \times m_2$-dimensional zero matrix (here $0_{k, n-k(p+1)-1}$ vanishes if $n-k(p+1) - 1 = 0$). Obviously $\rank(X) = \rank(H) = k$ holds, which implies $X \in \mathfrak{X}_0$. Furthermore, $e_+ \bot \lspan(X)$ holds, which follows immediately from $\bar{e}_+ \bot \lspan(H)$, because the non-zero columns of $X$ have column indices $j_i$ for $i = 1, \hdots, k+1$ by construction. Therefore, we see that $\hat{u}_X(e_+) = e_+$ showing that
\begin{equation}
\hat{V}_X(e_+) = X'\diag(\hat{u}_X(e_+)) = X'\diag(e_+).
\end{equation}
Now we apply Lemma \ref{AUXCONSTR} with $t = k$ to $(e_+, X) \in \Rel^n \times \mathfrak{X}_0$. That the tuple $(e_+, X)$ satisfies (A1) of that lemma is obvious from the preceding display and Equation \eqref{Xmatrix}. We also see that (A2) is satisfied because $j_{i+1} - j_i = p+1$ for $i = 1, \hdots k$, and $k(p+1) + p + \mathbf{1}_{\mathbb{M}_{AM}}(M) \leq n$ implies $n- j_{k+1} = n- k(p+1) - 1 \geq p-1$. That (A3) is satisfied follows from the preceding display together with $\rank(X) = \rank(H) = k$. To infer $g_{\kappa, M, p}^*(e_+, X, R) \neq 0$ from Part 2 of Lemma \ref{AUXCONSTR}, we consider three cases: First, if $M \in \mathbb{M}_{KV}$ (CKV) is obviously satisfied and we are done. Secondly, assume that $M \in \mathbb{M}_{AM}$. In this case we have by assumption $k(p+1) + p + 1 \leq n$ which implies $n- j_{k+1} = n- k(p+1) - 1 > p-1$. This shows that (CAM) is satisfied. Thirdly suppose that $M \in \mathbb{M}_{NW}$. Since $\hat{A}^{(p)}_{X}(e_+) = 0$ follows from Part 1 of Lemma  \ref{AUXCONSTR}, we see that  $\hat{Z}_X(e_+) = \hat{V}_{p, X}(e_+)$ and hence that the nonzero columns of $\hat{Z}_X(e_+)$ are precisely $H_{i \cdot}'$ for $i = 2, \hdots, k+1$. By construction we have $H_{2 \cdot} \omega \neq 0$ and $H_{i \cdot} \omega = 0$ for $i = 3, \hdots, k+1$. This shows that exactly one coordinate of $\omega'\hat{Z}_X(e_+)$ is non-zero which implies that (CNW) holds. To show that $\mathfrak{X}_1(e_-)$ is a $\lambda_{\Rel^{n \times k}}-$null set, we can use a similar construction: we replace $e_+$ by $e_-$ throughout. If $p$ is even we then have $\bar{e}_- = (-1, 1, -1, \hdots, (-1)^{k+1})'$. If $p$ is odd we then have $\bar{e}_- = (-1, -1, \hdots, -1)'$. Furthermore, if the construction of $M$ involves a weights vector we choose $H$ such that $H\omega = (1, 1, 0, \hdots, 0)'$ if $p$ is even, and $H\omega = (-1, 1, 0, \hdots, 0)'$ if $p$ is odd. The remaining arguments are identical.

Now consider $\mathfrak{X}_2(e_{+})$. Using Part 1 of Lemma \ref{gradientT} we see that in case $M \in \mathbb{M}_{KV}$ the set $\mathfrak{X}_2(e_+)$ is empty, because $(\grad T_X(.))|_{e_+ + \mu_{0, X}^*}$ exists whenever $X \in \mathfrak{X}_0 \backslash \mathfrak{X}_1(e_+)$. Next consider the cases where $M \notin \mathbb{M}_{KV}$ and where $\kappa$ satisfies Assumption \ref{CD}. From Part 2a of Lemma \ref{gradientT} we know that for $X \in \mathfrak{X}_0 \backslash \mathfrak{X}_1(e_+)$ the non-existence of $(\grad T_X(.))|_{e_+ + \mu_{0, X}^*}$ implies either $M(e_+) = 0$ or $i/M(e_+) \in \Delta(\kappa)$ for some $|i| = 1, \hdots, n-p-1$. The latter two cases can clearly be summarized as $M(e_+) \in \bar{\Delta}$, where $\bar{\Delta} = \geschw{\delta_0, \delta_1, \hdots, \delta_m}$ is a set consisting of finitely many elements. Therefore, 
\begin{align}
\mathfrak{X}_2(e_+) &\subseteq \geschw{X \in \mathfrak{X}_0 \backslash \mathfrak{X}_1(e_+):  (\grad T_X(.))|_{e_+ + \mu_0^*} \text{ does not exist}}  \\
&\subseteq \geschw{X \in \mathfrak{X}_0 \backslash \mathfrak{X}_1(e_+):  M(e_+) \in \bar{\Delta}}  \\
&= \bigcup_{i = 0}^m \geschw{X \in \mathfrak{X}_0 \backslash \mathfrak{X}_1(e_+): M(e_+) = \delta_i}.
\end{align}
We use Part 3 of Lemma \ref{gradientT} to rewrite the latter set as
\begin{align}
&\bigcup_{i = 0}^m \geschw{X \in \mathfrak{X}_0:  g^*_{\kappa, M, p}(e_+, X, R) \neq 0 \text{ and } g_{\kappa, M, p}^{(\delta_i)}(e_+, X) = 0} \\
= &\geschw{X \in \mathfrak{X}_0:  g^*_{\kappa, M, p}(e_+, X, R) \neq 0 \text{ and } \prod_{i = 0}^m g_{\kappa, M, p}^{(\delta_i)}(e_+, X) = 0} ,
\end{align}
which is clearly a subset of
\begin{equation}
\geschw{X \in \Rel^{n \times k}: \prod_{i = 0}^m g_{\kappa, M, p}^{(\delta_i)}(e_+, X) = 0}.
\end{equation}
Part 3 of Lemma \ref{gradientT} shows that $\prod_{i = 0}^m g_{\kappa, M, p}^{(\delta_i)}(e_+, .): \Rel^{n \times k} \rightarrow \Rel$ is a multivariate polynomial. We consider two cases: First assume that $\prod_{i = 0}^m g_{\kappa, M, p}^{(\delta_i)}(e_+, .) \not \equiv 0$. Consequently, the set in the previous display is a $\lambda_{\Rel^{n \times k}}$-null set. It hence follows that $\mathfrak{X}_2(e_+)$ is a $\lambda_{\Rel^{n \times k}}$-null set and we are done. Next, assume that $\prod_{i = 0}^m g_{\kappa, M, p}^{(\delta_i)}(e_+, .)  \equiv 0$. It follows that there must exist a single index $i$ such that $g_{\kappa, M, p}^{(\delta_i)}(e_+, .)  \equiv 0$ holds [this is easily shown by contradiction]. Part 3 of Lemma \ref{gradientT} hence shows that $X \in \mathfrak{X}_0$ and $g^*_{\kappa, M, p}(e_+, X, R) \neq 0$, i.e., $X \in \mathfrak{X}_0 \backslash \mathfrak{X}_1(e_+)$, implies $M(e_+) = \delta_i$. Clearly,
\begin{equation}
\mathfrak{X}_2(e_+) \subseteq \geschw{X \in \mathfrak{X}_0 \backslash \mathfrak{X}_1(e_+): T_X(e_+ + \mu_0^*) = C}.
\end{equation}
If $X \in \mathfrak{X}_0 \backslash \mathfrak{X}_1(e_+)$, then (cf. the argument in the beginning of the proof of Lemma \ref{gradientT} applied to $y = e_+$)
\begin{equation}\label{ts}
T_X(e_+ + \mu_{0, X}^*) = \hat{\beta}_X(e_+)'R'\left(\frac{n}{n-p} B_{p, X}(e_+)\mathcal{W}_{n-p}(e_+)B'_{p, X}((e_+)\right)^{-1} R \hat{\beta}_X(e_+).
\end{equation}
Furthermore, since $X \in \mathfrak{X}_0 \backslash \mathfrak{X}_1(e_+)$ implies $M(e_+) = \delta_i$, the matrix $\mathcal{W}_{n-p}(e_+)$ is constant $\bar{\mathcal{W}}_{n-p}$, say, on $\mathfrak{X}_0 \backslash \mathfrak{X}_1(e_+)$. Hence, for $X \in \mathfrak{X}_0 \backslash \mathfrak{X}_1(e_+)$, the statement $T_X(e_+ + \mu_{0, X}^*) = C$ is equivalent to
\begin{align}
&F^2_p(e_+, X) \left[\det(X'X) \hat{\beta}_X(e_+) \right]'R'\adj(\bar{B}_{p, X}(e_+)\bar{\mathcal{W}}_{n-p}\bar{B}'_{p, X}((e_+))) R \left[\det(X'X) \hat{\beta}_X(e_+) \right] \\[7pt]
& \hspace{3cm} = ~~ ~\frac{n}{n-p}\det(X'X)^2 \det(\bar{B}_{p, X}(e_+)\bar{\mathcal{W}}_{n-p}\bar{B}'_{p, X}((e_+))) C
\end{align}
where $\bar{B}_{p, X}(e_+)$ and $F_p(e_+, X)$ have been defined in the proof of Lemma \ref{N*PW}, where it is shown that  $g_{\kappa, M, p}(e_+, X) \neq 0$ and $X \in \mathfrak{X}_0$ (which is weaker than $X \in \mathfrak{X}_0 \backslash \mathfrak{X}_1(e_+)$) implies $F_p(e_+, X) \neq 0$. Furthermore, it is shown in the proof of Lemma \ref{N*PW} that $[\bar{B}_{p, .}(e_+)]_{ij}: \Rel^{n \times k} \rightarrow \Rel$ (for $1 \leq i \leq q$ and $1 \leq j \leq n-p$) is a multivariate polynomial, and that $F_p(e_+, .): \Rel^{n \times k}\rightarrow \Rel$ is a multivariate polynomial as well. It is easily seen that the coordinates of $\det(X'X) \hat{\beta}_X(e_+)$ as a function of $X$ are multivariate polynomials. Putting this together we have shown that
\begin{equation}
\mathfrak{X}_2(e_+) \subseteq \geschw{X \in \Rel^{n \times k}: p_{\kappa, M, p}(e_+, X, C) = 0},
\end{equation}
where $p_{\kappa, M, p}(e_+, ., C): \Rel^{n \times k} \rightarrow \Rel$ is a multivariate polynomial. Therefore, if we can show that $p_{\kappa, M, p}(e_+, ., C) \not \equiv 0$ we obtain that $\mathfrak{X}_2(e_+)$ is a $\lambda_{\Rel^{n \times k}}$-null set. In the proof of the part concerning $\mathfrak{X}_1(e_+)$ above we have already constructed an $X \in \mathfrak{X}_0 \backslash \mathfrak{X}_1$ that satisfies $e_+ \bot \lspan(X)$ which implies $\hat{\beta}_X(e_+) = 0$. Together with Equation \eqref{ts} this shows that $T_X(e_+ + \mu_{0, X}^*) = 0 < C$ holds for this specific $X$. But this immediately shows that $p_{\kappa, M, p}(e_+, X, C) \neq 0$ and we are done. Clearly, we can use an almost identical argument to prove the statement concerning $\mathfrak{X}_2(e_{-})$. Under Assumption \ref{AAR(1)} the set of matrices $X \in \mathfrak{X}_0$ for which the first three cases of Theorem \ref{thmlrvPW} do not apply is obviously a subset of $(\mathfrak{X}_1(e_+) \cup \mathfrak{X}_2(e_+)) \cap (\mathfrak{X}_1(e_-) \cup \mathfrak{X}_2(e_-))$. Hence the first part of the proposition follows.

\medskip

2) We start with the statement concerning $\tilde{\mathfrak{X}}_{1}\left( e_{-}\right)$. Under the maintained assumptions we know from Part 5 of Lemma \ref{N*PW} that $g_{\kappa, M,p}^*(.,., .): \Rel^n \times \Rel^{n \times k} \times \Rel^{q \times k} \rightarrow \Rel$ is a multivariate polynomial. This immediately implies that $g_{\kappa, M,p}^*(e_-,(e_+, .),R): \Rel^{n \times (k-1)} \rightarrow \Rel$ is a multivariate polynomial, which shows that
\begin{equation}
\geschw{\tilde{X} \in \Rel^{n \times (k-1)}: g_{\kappa, M, p}^*(e_-, (e_+, \tilde{X}), R) = 0}
\end{equation}
is an algebraic superset of $\tilde{\mathfrak{X}}_{1}\left( e_{-}\right)$. It hence suffices to show that the set in the previous display is a $\lambda_{\Rel^{n \times (k-1)}}$- null set, or equivalently to show that $g_{\kappa,M,p}^*(e_{-}, (e_+, .), R) \not \equiv 0$. Again, we shall use Lemma \ref{AUXCONSTR} with $t = k$ to construct a matrix $\tilde{X} \in \tilde{\mathfrak{X}}_0$ such that $g_{\kappa, M, p}^*(e_{-}, (e_+, \tilde{X}), R) \neq 0$. The situation here is more complicated than in the first part, because the first column of the design matrix we seek has to be the intercept. For our construction we need some additional ingredients: By definition $p^* = p+1$ if $p$ is odd, and $p^* = p$ if $p$ is even. If $p$ is odd set $v = \bar{e}_-^* = (-1, -1, \hdots, -1, 1)' \in \Rel^{k+1}$, where $\bar{e}_-^*$ is the vector obtained from $e_-$ by selecting the coordinates $j^*_i = j_i$ for $i = 1, \hdots, k$ and $j^*_{k+1} = j_k^* + p^* + 1$, where $j_i = 1 + (i-1)(p+1)$ for $i = 1, \hdots, k+1$ was defined in Part 1 above. This selection is feasible, because by assumption we have $k(p+1) + p^* \leq n$, which, since $p$ is odd, gives $k(p+1) + p +1 \leq n$, implying that $j^*_{k+1} = (k-1)(p+1) + p^* + 2 = k(p+1) + 2 \leq n$, because of $p \geq 1$. If $p$ is even set $v = \bar{e}_- = (-1, 1, -1, \hdots, (-1)^{k+1})'\in \Rel^{k+1}$, the vector obtained from $e_-$ by selecting the coordinates $j_i^* = j_i$ for $i = 1, \hdots, k+1$. Next, define 
\begin{equation}
z = (-1, k^{-1}, \hdots, k^{-1})' \in \Rel^{k+1}.
\end{equation}
We claim that $v$ and $z$ satisfy $u := \Pi_{\lspan(z)}v \neq 0$,  $x := \Pi_{\lspan(z)^{\bot}}v = v-u$ is linearly independent of $e := (1, \hdots, 1)' \in \Rel^{k+1}$ and $z$ is orthogonal to $e$. The latter property is clearly always satisfied, regardless of whether $p$ is even or odd. We thus only have to verify the first two conditions. We start with the case $p$ odd. Here we have $z'v = 2k^{-1}\neq 0$ and therefore $\Pi_{\lspan(z)}v \neq 0$. Furthermore $\Pi_{\lspan(z)^{\bot}}v = v - \norm{z}^{-2}z'v z$ can not equal $ce$ for some $c \in \Rel$, because the last and the last but one coordinate of $v$ are unequal. For $p$ even $z'v = (1 + k^{-1} s)$, where $s$ equals either $0$ (if $k$ is even) or $1$ (if $k$ is odd), therefore $z'v \neq 0$ holds and thus $\Pi_{\lspan(z)}v \neq 0$. Furthermore $\Pi_{\lspan(z)^{\bot}}v = v - \norm{z}^{-2} z'v z$ can not equal $ce$ for some $c \in \Rel$, because the second and third coordinate of $v$ are unequal. This proves the claim. Using these properties, we see that $u \in \lspan(z) \backslash \geschw{0}$ and 
\begin{equation}
u = v - x = v - \Pi_{\lspan(e, x)}(v-\Pi_{\lspan(z)}v) = v- \Pi_{\lspan(e, x)}v = \Pi_{\lspan(e, x)^{\bot}}v,
\end{equation}
where we have used that $z$ is orthogonal to both $e$ and $x$ to derive the third equality. We shall now define an auxiliary matrix. Let $L$ denote a $(k+1) \times k$-dimensional matrix such that $L_{\cdot 1} = e$, $L_{\cdot2} = x$ and such that the remaining $k-2$ columns $L_{\cdot j}$ for $j = 3, \hdots, k$ are linearly independent and orthogonal to $\lspan(e, x, v)$. Since $e$ and $x$ are linearly independent, we have $\rank(L) = k$. For later use we observe that
\begin{equation}
\Pi_{\lspan(L)}v = \Pi_{\lspan((e, x))} v = v - \Pi_{\lspan(e, x)^{\bot}}v = v-u = x,
\end{equation}
where the first equality follows immediately from $L_{\cdot j}$ for $j = 3, \hdots, k$ being linearly independent and orthogonal to $\lspan(e, x, v)$. This immediately shows  
\begin{equation}
\hat{\beta}_L(v) = (0, 1, 0, \hdots, 0)'\in \Rel^k.
\end{equation}
Define the two $k$-vectors $r_{-} = (1, -1, 0, \hdots, 0)'$ and $r_{+} = (1, 1, 0, \hdots, 0)'$. Let $X \in \Rel^{n \times k}$ be such that $X_{j_i^* \cdot} = L_{i \cdot}$ for $i = 1, \hdots, k+1$, and if the index $j \notin \geschw{j^*_1, \hdots, j^*_{k+1}}$, then let $X_{j \cdot} = r_+$ if $[e_{-}]_j = 1$, and let $X_{j \cdot} = r_-$ if $[e_{-}]_j = -1$. By construction the matrix $X$ is of the form $X = (e_+, \tilde{X})$. We claim that $\hat{\beta}_X(e_-) = \hat{\beta}_L(v)$. To see this denote the set of indices $j \in \geschw{1, \hdots, n} \backslash \geschw{j_1^*, \hdots, j_{k+1}^*}$ such that $[e_{-}]_j  = -1$ by $\mathcal{I}_-$, and the set of indices $j \in \geschw{1, \hdots, n} \backslash \geschw{j_1^*, \hdots, j_{k+1}^*}$ such that $[e_{-}]_j = 1$ by $\mathcal{I}_+$. The sum of squares $S(\beta) = \norm{e_- - X \beta}^2$ can be written as
\begin{align}
S(\beta) &= \sum_{i = 1}^{k+1} ([e_-]_{j^*_i} - X_{j^*_i \cdot} \beta)^2 + \sum_{j \in \mathcal{I}_-} (-1 - r_- \beta)^2 + \sum_{j \in \mathcal{I}_+} (1 - r_+ \beta)^2 \\
&= \sum_{i = 1}^{k+1} (v_i - L_{i \cdot} \beta)^2 + \sum_{j \in \mathcal{I}_-} (-1 - r_- '\beta)^2 + \sum_{j \in \mathcal{I}_+} (1 - r_+ '\beta)^2 \\
&= \norm{v - L\beta}^2 + \sum_{j \in \mathcal{I}_-} (-1 - r_-' \beta)^2 + \sum_{j \in \mathcal{I}_+} (1 - r_+' \beta)^2 
\end{align}
If we now plug in $\beta =  \hat{\beta}_L(v)$ and note that $r_+' \hat{\beta}_L(v) = 1$ and $r_-' \hat{\beta}_L(v) = -1$ we see that
\begin{equation}
S(\hat{\beta}_L(v)) = \sum_{i = 1}^{k+1} (v_i - L_{i \cdot}\hat{\beta}_L(v))^2 = \min_{\beta \in \Rel^{k}} \norm{v - L\beta}^2.
\end{equation}
This immediately proves the claim $\hat{\beta}_X(e_-) = \hat{\beta}_L(v)$. Hence, the residual vector satisfies
\begin{equation}
[\hat{u}_{X}(e_-)]_j =  \begin{cases}
u_i & \text{ if } j = j^*_i \text{ for some } i = 1, \hdots, k+1 \\
0 & \text{else}.
\end{cases}
\end{equation}
This immediately entails that $\hat{V}_X(e_-) = X' \diag(\hat{u}_{X}(e_-))$ equals
\begin{equation}\label{hatVeminus}
(u_1 L'_{1 \cdot}, 0_{k, p}, u_2 L'_{2 \cdot}, 0_{k, p} , \hdots, u_{k}L_k', 0_{k, p^*}, u_{k+1} L_{(k+1) \cdot}', 0_{k, n- [k(p+1) + p^* - p + 1]}),
\end{equation}
where the indices of the nonzero columns of this matrix are precisely $j_i^*$ for $i = 1, \hdots, k+1$, because the first column of $L$ is $e$ and $u_i \neq 0$ for $i = 1, \hdots, k+1$, the latter following since $u \in \lspan(z) \backslash \geschw{0}$ and $z_i \neq 0$ for $i = 1, \hdots, k+1$ by definition. In deriving the dimension of $ 0_{k, n- [k(p+1) + p^* - p + 1]}$ we used 
\begin{equation}
j_{k+1}^* = k(p+1) + 1 + p^* - p = \begin{cases}
k(p+1) + 2& \text{ if } p \text{ odd} \\
k(p+1) + 1& \text{ if } p \text{ even}.
\end{cases}
\end{equation}
Now we apply Lemma \ref{AUXCONSTR} (with $t = k$). Clearly, $\rank(X) = \rank(L) = k$. From Equation \eqref{hatVeminus}, and the discussion following it, we see that the tuple $(e_-, X) \in \Rel^n \times \mathfrak{X}_0$ satisfies Assumption (A1). Assumption (A2) is satisfied, because $j_{i+1}^* - j_{i}^* \geq p+1$ for $i = 1, \hdots, k$, and because we see from the previous display that $n-j_{k+1}^* = n- [k(p+1) + 1 + p^* - p]$, which together with the assumption $k(p+1)+ p^* + \mathbf{1}_{\mathbb{M}_{AM}}(M) \leq n$ implies $n-j_{k+1}^* \geq p-1$. Assumption (A3) is satisfied because $\rank(\hat{V}_X(e_-)) = \rank(L) = k$. If $M \in \mathbb{M}_{KV}$ we are done. Consider the case $M \in \mathbb{M}_{AM}$. We show that Condition (CAM) is satisfied. But this is obvious, because the assumption $k(p+1)+p^* + 1 \leq n$ immediately implies $n-j_{k+1}^* > p-1$. Suppose $M \in \mathbb{M}_{NW}$. We apply Part 4 of Lemma \ref{AUXCONSTR}. For this we claim that either $[\hat{V}_X(e_-)]_{1j_i^*} > 0$ for $i = 2, \hdots, k+1$ or $[\hat{V}_X(e_-)]_{1j_i^*} < 0$ for $i = 2, \hdots, k+1$. Assuming that this claim is true, the lemma shows that there exists a regular matrix $Q \in \Rel^{k \times k}$ such that $X Q \in \mathfrak{X}_0$, the first column of $X Q$ is $e_+$ and $g_{\kappa, M, p}^*(e_+, XQ, R) \neq 0$, and we are done. To prove the claim recall that by construction $u \in \lspan(z) \backslash \geschw{0}$ holds, which shows that either $u_i < 0$ for $i = 2, \hdots, k+1$ or $u_i > 0$ for $i = 2, \hdots, k+1$. Furthermore, the first column of $L$ is the vector $e = (1, \hdots, 1)$. Equation \eqref{hatVeminus} now shows that $[\hat{V}_X(e_-)]_{1j_i^*} = u_i$ for $i = 2, \hdots, k+1$. This proves the claim.

The part of the statement concerning $\tilde{X}_2(e_{-})$ is established by exploiting an argument similar to the one given in Part 1 of the proof. Firstly, if $M \in \mathbb{M}_{KV}$, then we know from Part 1 of Lemma \ref{gradientT} that $g_{\kappa, M, p}^*(e_-, (e_+, \tilde{X}), R) \neq 0$ implies existence of $\grad(T_{(e_+, \tilde{X})}(.))|_{e_- + \mu_{0, (e_+, \tilde{X})}^*}$. Therefore, $\tilde{X}_2(e_{-})$ is empty in Case (a). It remains to prove the remaining three cases, in all of which Assumption \ref{CD} holds. Clearly we can also assume that $M \notin \mathbb{M}_{KV}$. We start with the following observation: Combining Assumption \ref{CD} with Part 2a of Lemma \ref{gradientT} as in Part 1 of the proof, we see that there exists an integer $m \geq 0$ and real numbers $\delta_0, \hdots, \delta_m$, such that
\begin{equation} \label{INCL}
\tilde{X}_2(e_{-}) \subseteq \geschw{\tilde{X} \in \Rel^{n \times (k-1)}: \prod_{i = 0}^m g_{\kappa, M, p}^{(\delta_i)}(e_-, (e_+, \tilde{X})) = 0}.
\end{equation}
It follows with the same argument as in Part 1 that either $\tilde{X}_2(e_{-})$ is a $\lambda_{\Rel^{n \times (k-1)}}$-null set, or there exists an index $i$ such that $g_{\kappa, M, p}^{(\delta_i)}(e_-, (e_+, .)) \equiv 0$. In the former case we are done. In the latter case one can show, with a similar argument as in Part 1 of the proof, that
\begin{equation}\label{supset}
\tilde{\mathfrak{X}}_2(e_-) \subseteq \geschw{\tilde{X} \in \Rel^{n \times (k-1)}: p_{\kappa, M, p}(e_-, (e_+, \tilde{X}), C) = 0},
\end{equation}
where $p_{\kappa, M, p}(e_-, (e_+, .), C): \Rel^{n \times (k-1)} \rightarrow \Rel$ is a multivariate polynomial. Either we have that $p_{\kappa, M, p}(e_-, (e_+, .), C) \not \equiv 0$ and $\tilde{\mathfrak{X}}_2(e_-)$ is a $\lambda_{\Rel^{n\times (k-1)}}$-null set, or $p_{\kappa, M, p}(e_-, (e_+, .), C) \equiv 0$ and the superset in the previous display coincides with $\tilde{\mathfrak{X}}_0 \backslash \tilde{\mathfrak{X}}_1(e_-)$. Consider Condition (d). If the function $\tilde{X} \mapsto T_{(e_+, \tilde{X})}(\mu_{0, (e_+, \tilde{X})}^* + e_-)$ is not constant $C$ on $\tilde{\mathfrak{X}}_0 \backslash \tilde{\mathfrak{X}}_1(e_-)$, then $p_{\kappa, M, p}(e_-, (e_+, .), C) \not \equiv 0$, and hence the superset in Equation \eqref{supset} is a $\lambda_{\Rel^{n\times (k-1)}}$-null set. This shows that $\tilde{\mathfrak{X}}_2(e_{-})$ is a null set under Condition (d).

For Condition (b) we consider again the inclusion in Equation \eqref{INCL}. Either the superset is a null set and we are done, or there must exist a real number $\delta_i$ such that $g_{\kappa, M, p}^{(\delta_i)}(e_-, (e_+, .)) \equiv 0$. Assume the latter. We exploit a property of the matrix $X = (e_+, \tilde{X})$ with $\tilde{X} \in \tilde{\mathfrak{X}}_0 \backslash \tilde{\mathfrak{X}}_1(e_-)$ constructed above. In the proof of Part 2 (CAM) of Lemma \ref{AUXCONSTR} it is shown that for this specific $X$ we have $M(e_-) = 0$. Therefore, $g_{\kappa, M, p}^{(0)}(e_-, (e_+, .)) \equiv 0$, or equivalently $M(e_-) = 0$ for every design matrix $X = (e_+, \tilde{X})$ with $\tilde{X} \in \tilde{\mathfrak{X}}_0 \backslash \tilde{\mathfrak{X}}_1(e_-)$. But since the kernel satisfies Assumption \ref{CDL}, the existence of $(\grad T_{(e_+, \tilde{X})}(.))|_{\mu_{0, (e_+, \tilde{X})}^* + e_-}$ for every $\tilde{X} \in \tilde{\mathfrak{X}}_0 \backslash \tilde{\mathfrak{X}}_1(e_-)$ then follows from Part 2b of Lemma \ref{gradientT}. Hence, $\tilde{\mathfrak{X}}_2(e_-)= \emptyset$. 

Consider Condition (c). We use a similar argument as under Condition (b): We establish the existence of a sequence of matrices $(e_+, \tilde{X}_m)$ with $\tilde{X}_m$ eventually in $\tilde{\mathfrak{X}}_0 \backslash \tilde{\mathfrak{X}}_1(e_-)$, such that $M(e_-) \rightarrow 0$ as $m \rightarrow \infty$. From an argument as in the proof under Condition (b) this then implies that either $\tilde{\mathfrak{X}}_2(e_-)$ is a null set, or that $M(e_-) \equiv 0$ on $\tilde{\mathfrak{X}}_0 \backslash \tilde{\mathfrak{X}}_1(e_-)$. But since $\kappa$ satisfies Assumption \ref{CDL}, it then follows from Part 2b of Lemma \ref{gradientT} that in the latter case $\tilde{\mathfrak{X}}_2(e_-)$ is empty. This then proves the claim. It remains to construct a sequence $\tilde{X}_m$ as claimed. By assumption $\omega_i > 0$ for some $i > 1$. Assume without loss of generality that $i = 2$ (otherwise we have to interchange the columns of the $\tilde{X}_m$ sequence to be constructed accordingly). Let $X = (e_+, \tilde{X})$ be as constructed above. Let $\gamma_m > 0$ be a sequence diverging to $\infty$. Recall that by construction $\tilde{X}_{j^*_i 1} = x_{i}$ for $i = 1, \hdots, k+1$. Since $p$ is odd, a simple calculation shows that $x = \Pi_{\lspan(z)^{\bot}} v$ equals
\begin{equation}
x = \left(-1 + \frac{2}{k+1}, -1 - \frac{2k^{-1}}{k+1}, \hdots, -1 - \frac{2k^{-1}}{k+1}, 1 - \frac{2k^{-1}}{k+1}\right)'.
\end{equation}
Let $c = 1 + \frac{2k^{-1}}{k+1}$ and note that 
\begin{equation}
x + ce = 2 (\frac{k^{-1}+1}{k+1}, 0, \hdots, 0, 1)'.
\end{equation}
Define the $k \times k$ dimensional regular matrix
\begin{equation}
Q_m = Q D_m = 
\begin{pmatrix}
1 &  c & \hdots & \hdots & \hdots & 0 \\
0 &  1 & 0 & \hdots & \hdots & 0 \\
0 & 0 & 1 & 0 & \hdots & 0 \\
\vdots & \vdots \vdots & \vdots & \vdots & \vdots \\
0 & 0 & 0 & 0 & \hdots & 1
\end{pmatrix} \diag(1, \gamma_m, 1, \hdots, 1).
\end{equation}
Clearly, post-multiplying a matrix with $k$ columns by $Q_m$ has the same effect as adding $c$ times the first column to the second column, then multiplying the column so obtained by $\gamma_m$ and leaving all other columns unchanged. Since $L_{\cdot 1} = e$ and $L_{\cdot 2} = x$, the expression for $x + ce$ above shows that the second column of $L Q$, has precisely two nonzero elements with indices $1$ and $k+1$, respectively. Since $X = (e_+, \tilde{X}) \in \mathfrak{X}_0$ and $(e_-,X)$ satisfies (A1)-(A3) as established above, Part 3 of Lemma \ref{AUXCONSTR} shows that 
\begin{equation}
X_m = (e_+, \tilde{X}) Q_m = (e_+, \tilde{X}_m) \in \mathfrak{X}_0,
\end{equation}
and that the tuple $(e_-, X_m) \in \Rel^n \times \mathfrak{X}_0$ satisfies (A1), (A2) and (A3). Hence $\hat{A}^{(p)}_{ X_m }(e_-) = 0$ holds. As a consequence $\hat{Z}_{X_m}(e_-)$ is well defined for every $m$. Using 
\begin{equation}
\hat{Z}_{X_m}(e_-) = D_m Q' \hat{Z}_{X}(e_-)
\end{equation}
(cf. the proof of Lemma \ref{AUXCONSTR} Part 4) and $\rank(\hat{Z}_{X}(e_-)) = k$, which was shown above, we see that $\rank(\hat{Z}_{X_m)}(e_-)) = k$ must hold, which together with $\omega \neq 0$ immediately implies $\bar{\sigma}_{0, X_m} (e_-) \neq 0$. Since $\hat{Z}_{X}(e_-)$ is obtained from $\hat{V}_{X}(e_-)$ by deleting its first $p$ columns, we observe, using the remark concerning the second column of $L Q$ above together with Equation \eqref{hatVeminus}, that the second row of $Q' \hat{Z}_{X}(e_-)$ has exactly one non-zero coordinate, namely $2 u_{k+1}$. Consider 
\begin{align}\label{quotsigma}
\frac{\bar{\sigma}_{i, X_m} (e_-)}{\bar{\sigma}_{0, X_m} (e_-)} 
=& \frac{\sum_{j=|i| + 1}^{n-p} \omega'[\hat{Z}_{ X_m}(y)]_{\cdot j}[\hat{Z}_{ X_m}(y)]_{\cdot (j-|i|)}'\omega}{\sum_{j=1}^{n-p} \omega'[\hat{Z}_{ X_m}(y)]_{\cdot j}[\hat{Z}_{ X_m}(y)]_{\cdot j}'\omega} \\
=& \frac{\sum_{j=|i| + 1}^{n-p} \bar{\omega}_m'[Q' \hat{Z}_{ X}(y)]_{\cdot j}[Q' \hat{Z}_{ X}(y)]_{\cdot (j-|i|)}'\bar{\omega}_m}{\sum_{j=1}^{n-p} \bar{\omega}_m'[Q' \hat{Z}_{ X}(y)]_{\cdot j}[Q' \hat{Z}_{ X}(y)]_{\cdot j}'\bar{\omega}_m}, 
\end{align}
for $|i| = 1, \hdots, n-p-1$, where $\bar{\omega}_m= D_m \omega /\norm{D_m \omega}$. Clearly $\bar{\omega}_m \rightarrow (0, 1, 0, \hdots, 0)$. Since the second row of $Q' \hat{Z}_{ (e_+, \tilde{X})}$ contains by construction precisely one nonzero entry, it follows that the limit of \eqref{quotsigma} must be $0$ for $i = 1, \hdots, n-p-1$. Furthermore we have $\bar{\sigma}_{0, X_m} (e_-) \rightarrow \infty$. It immediately follows from $w(0) = 1$ and the definition of $M$ that $M(e_-)$ is well defined for $m$ large and that it converges to $0$ as $m \rightarrow \infty$. The remaining part of the proposition is obvious.

\medskip

3) Let $\tilde{X} \in \mathfrak{\tilde{X}}_0$ and assume that $X = (e_{+}, \tilde{X})$ satisfies $g_{\kappa, M, p}^*(., X, R) \not \equiv 0$. Obviously, $e_{+} \in \text{span}(X)$. Note that $\hat{\beta}_X(e_{+}) = e_{1}(k)$. The first column of $R$ is non-zero. Therefore $R\hat{\beta}_X(e_{+}) \neq 0$. Thus we can (since Assumption \ref{AAR(1)} holds) apply Part 4 of Theorem \ref{thmlrvPW}.
\end{proof}

\section{Proofs of Results in Section \ref{pos}}\label{AD}

\begin{proof}[Proof of Theorem \ref{excPW}]
We verify the assumptions of Theorem 5.21 in \cite{PP13} with $\hat{\Omega}_{\kappa, M, p} = \check{\Omega}$ and $\hat{\beta} = \check{\beta}$. Because $g^*_{\kappa, M, p}(., X, R) \not \equiv 0$ by assumption, and  since the triple $\kappa$, $M$, $p$ satisfies Assumption \ref{weightsPWRB}, we can use Lemma \ref{A567PW} to conclude that $\hat{\beta}$ and $\hat{\Omega}_{\kappa, M, p}$ satisfy Assumptions 5, 6 and 7 in \cite{PP13}, that $\hat{\Omega}_{\kappa, M, p}$ is almost everywhere positive definite and that $T$ is invariant w.r.t. $G(\mathfrak{M}_0)$. Assumption \ref{approxAAR(1)}, Remark \ref{RapproxAAR(1)} (Part (iii)) together with Remark 5.14 (ii) in \cite{PP13} now shows that $J(\mathfrak{C}) = \lspan(e_+) \cup \lspan(e_-)$ and that all assumptions on $\mathfrak{C}$ appearing in Theorem 5.21 in \cite{PP13} are satisfied. Because $e_+, e_- \in \mathfrak{M}$ is assumed we have $J(\mathfrak{C}) \subseteq \mathfrak{M}$. The assumption $R\hat{\beta}(e_+) = R\hat{\beta}(e_-) = 0$ even implies $J(\mathfrak{C}) \subseteq \mathfrak{M}_0 -\mu_0$ (for some arbitrary $\mu_0 \in \mathfrak{M}_0$). Invariance of $T$ w.r.t. $G(\mathfrak{M}_0)$ then shows that Equation (34) in Theorem 5.21 of \cite{PP13} is satisfied. The assumptions on $\check{\Omega}$ appearing in Parts 2 and 3 of that theorem are satisfied, because $\hat{\Omega}_{\kappa, M, p}$ is positive definite almost everywhere. The theorem now follows from Theorem 5.21 in \cite{PP13}, using a standard subsequence argument, positive definiteness of every element of $\mathfrak{C}$ and compactness of $\mathfrak{C}^*$, to obtain the second statement in Part 3 from the corresponding Part of Theorem 5.21 in \cite{PP13}. The claim in parenthesis in Part 3 follows from the corresponding claim in parenthesis in Theorem 5.21 of \cite{PP13}, together with the observation that the conditions on $e_+$ and $e_-$ have only been used to verify the condition in Equation (34) of Theorem 5.21 of \cite{PP13} (cf. the proof of Theorem 3.7 in \cite{PP13}).
\end{proof}

\begin{proof}[Proof of Theorem \ref{TU_3}]
We apply Theorem 5.21 of \cite{PP13} with the estimators $\hat{\Omega}_{\kappa, \bar{M}, p, \bar{X}} = \check{\Omega}$ and $(I_k, 0) \hat{\beta}_{\bar{X}} = \check{\beta}$. Obviously, the test statistic defined in Equation (28) of \cite{PP13} based on these estimators coincides with the test statistic $\bar{T}$ as defined in the statement of the present theorem. Since the assumptions concerning $\mathfrak{C}$ in the present theorem are the same as in Proposition \ref{excPW}, we see from the proof of this proposition that it suffices to verify that $\hat{\Omega}_{\kappa, \bar{M}, p, \bar{X}}$ and $(I_k, 0) \hat{\beta}_{\bar{X}}$ satisfy Assumption 5 in \cite{PP13}, that $\hat{\Omega}_{\kappa, \bar{M}, p, \bar{X}}$ is almost everywhere positive definite (implying that Assumptions 6 and 7 in \cite{PP13} are satisfied), and that the invariance condition in Equation (34) of \cite{PP13} is satisfied by $\bar{T}$. By definition, $\hat{\Omega}_{\kappa, \bar{M}, p, \bar{X}}$ is the estimator one would obtain following Steps 1-3 of the construction in Section \ref{tpwc} based on $\kappa$, $\bar{M}$ and $p$, if $\bar{X}$ was the underlying design matrix (observe that $\bar{X}$ is of full column rank) and $(\bar{R}, r)$ was the hypothesis to be tested. By assumption, the triple $\kappa$, $M$, $p$ satisfies Assumption \ref{weightsPWRB} w.r.t. the (dimensions $k$ and $n$ of the) design matrix $X$ and additionally $1 \leq p \leq n/(k+3)$ holds. From $1 \leq \bar{k}-k \leq 2$, and the definition of $\bar{M}$ it follows that the triple $\kappa$, $\bar{M}$, $p$ satisfies Assumption \ref{weightsPWRB} w.r.t. (the dimensions $\bar{k}$ and $n$ of) $\bar{X}$. Furthermore, it is assumed that $g^*_{\kappa, \bar{M}, p}(., \bar{X}, \bar{R}) \not \equiv 0$. Therefore, we can apply Lemma \ref{A567PW}, acting as if $\bar{X}$ was the underlying design matrix, to conclude that $\hat{\beta}_{\bar{X}}$ and $\hat{\Omega}_{\kappa, \bar{M}, p, \bar{X}}$ satisfy Assumption 5 in \cite{PP13} with $N = N(\hat{\Omega}_{\kappa, \bar{M}, p, \bar{X}})$ and $k$ replaced by $\bar{k}$, $X$ replaced by $\bar{X}$ and $\mathfrak{M}$ replaced by $\bar{\mathfrak{M}} = \lspan(\bar{X})$). Furthermore, Lemma \ref{A567PW} shows that $\hat{\Omega}_{\kappa, \bar{M}, p, \bar{X}}$ is almost everywhere positive definite. We now apply Part 1 of Proposition 5.23 in \cite{PP13} to obtain that  $\hat{\Omega}_{\kappa, \bar{M}, p, \bar{X}}$ and $(I_k, 0) \hat{\beta}_{\bar{X}}$ satisfy (the original) Assumption 5 in \cite{PP13}, and that the invariance condition is satisfied. To this end, it suffices to verify that in each of the four cases we have $\lspan(J(\mathfrak{C})) \cap \mathfrak{M} \subseteq  \mathfrak{M}_0 - \mu_0$ (for some arbitrary $\mu_0 \in \mathfrak{M}_0$). This is obvious in the first three cases. For Case 4 we can use exactly the same argument as in the proof of Part 4 in Theorem 3.8 in \cite{PP13}. 
\end{proof}

\begin{proof}[Proof of Proposition \ref{genericADJ}]
We begin with the proof of the first statement. We note that for $\lambda_{\Rel^{n \times k}}$- almost every $X \in \mathfrak{X}_0$ Case 3 of Theorem \ref{TU_3} applies: Since by assumption $(k+3)(p^* + 2) + p-1 \leq n$ and by definition $p^* \geq 1$, we have $k+2 < n$. Therefore, the set of matrices $X$ in $\Rel^{n \times k}$ such that $\det((X, e_+, e_-)(X, e_+, e_-)') = 0$ holds is a $\lambda_{\Rel^{n \times k}}$- null set. Hence, $e_+, e_- \notin \mathfrak{M}_X$ and $\rank((X, e_+, e_-)) = k+2$  holds for $\lambda_{\Rel^{n \times k}}$- almost every $X \in \mathfrak{X}_0$. It remains to verify that $g^*_{\kappa, \bar{M}, p}(., \bar{X}, \bar{R}) \not \equiv 0$ for almost every $X \in \mathfrak{X}_0$, where $\bar{X} = \bar{X}(X) = (X, e_+, e_-)$, $\bar{R} = (R, 0, 0)$ and $\bar{M}$ is constructed as outlined in Theorem \ref{TU_3}. For that it suffices to find a matrix $X \in \Rel^{n \times k}$ and a vector $y \in \Rel^{n}$ such that $g^*_{\kappa, \bar{M}, p}(y, \bar{X}, \bar{R}) \neq 0$. To see this note that the triple $\kappa$, $\bar{M}$, $p$ satisfies Assumption \ref{weightsPWRB} w.r.t. (the dimensions of) $\bar{X}$ (cf. the proof of Theorem \ref{TU_3}). Therefore, Lemma \ref{N*PW} shows that $(y, X) \mapsto g^*_{\kappa, \bar{M}, p}(y, \bar{X}, \bar{R})$ is a multivariate polynomial. If we can find a matrix $X$ and a vector $y$ as above, this implies that the multivariate polynomial $(y, X) \mapsto g^*_{\kappa, \bar{M}, p}(y, \bar{X}, \bar{R})$ is not the zero polynomial, and therefore the zero set of this multivariate polynomial is a $\lambda_{\Rel^n \times \Rel^{n \times k}}$- null set. It follows that for $\lambda_{\Rel^{n \times k}}$- almost every $X$ we must have $g^*_{\kappa, \bar{M}, p}(., \bar{X}, \bar{R}) \not \equiv 0$ [Assuming the opposite, there exists a set $A \in \mathcal{B}(\Rel^{n \times k})$ of positive $\lambda_{\Rel^{n \times k}}$- measure such that $g^*_{\kappa, \bar{M}, p}(., \bar{X}, \bar{R}) \equiv 0$ for every $X \in A$, which implies that $\Rel^n \times A \subseteq \Rel^n \times \Rel^{n \times k}$ is a subset of the zero set of $(y, X) \mapsto g^*_M(y, \bar{X}, \bar{R})$. But clearly $\Rel^n \times A$ has positive Lebesgue measure, a contradiction.]. In the following we shall construct such a pair $(y, X)$ as above:

\medskip

Let $\delta \neq 0$ and define $w_1, w_2, v(\delta) \in \Rel^{k+4}$ as $w_1 = (1, 1, \hdots, 1)'$, $w_2 = (-1, -1, 1, \hdots, 1)'$ and $v(\delta) = (-\delta, \delta, -(k+1), 1, \hdots, 1, 1)'$. By construction $v(\delta)$ is orthogonal to $w_1$ and $w_2$. Noting that $[w_1]_i = [w_2]_i$ for $i \geq 3$ a dimensionality argument implies existence of $k$ normalized vectors $w_3, \hdots, w_{k+2} \in \Rel^{k+4}$, that are functionally independent of $\delta$, linearly independent and orthogonal to $e_1(k+4)$, $e_2(k+4)$, $w_1$, $w_2$ and $v(\delta)$ (for every $\delta \neq 0$). Recall that $e_1(k+4)$ and $e_2(k+4)$ are the first two elements of the canonical basis of $\Rel^{k+4}$. Hence, the first two coordinates of $w_i$ for $i = 3, \hdots, k+2$ are zero. These orthogonality properties readily imply
\begin{equation}\label{orthoy}
\Pi_{\lspan(w_3, \hdots, w_{k+2}, w_1, w_2)^{\bot}} v(\delta) = \Pi_{\lspan(w_3 + w_1, \hdots, w_{k+2} + w_1, w_1, w_2)^{\bot}} v(\delta) = v(\delta)
\end{equation}
and $\rank(\bar{W}) = k+2$ for $\bar{W} = (w_3 + w_1, \hdots, w_{k+2} + w_1, w_1, w_2)$. Inserting zero coordinates and rows, respectively, we shall now suitably embed $v(\delta) \in \Rel^{k+4}$ and $W = (w_3 + w_1, \hdots, w_{k+2} + w_1) \in \Rel^{k+4 \times k}$ into $\Rel^n$ and $\Rel^{n \times k}$. Define $y(\delta)$ as 
\begin{equation}
(v_1(\delta), 0_{1,p^*+1}, v_2(\delta), 0_{1,p^*}, v_3(\delta), 0_{1,p^*+1}, v_4(\delta), 0_{1,p^*+1}, \hdots, v_{k+4}(\delta), 0_{1,p-1}, 0_{1,n-n^*})'
\end{equation}
and $X$ as
\begin{equation}
(W'_{1\cdot}, 0_{k,p^*+1}, W'_{2\cdot}, 0_{k,p^*}, W'_{3\cdot}, 0_{k,p^*+1}, W'_{4\cdot}, 0_{k,p^*+1},  \hdots,  W'_{(k+4)\cdot} ,0_{1,p-1}, 0_{1,n-n^*})',
\end{equation}
where $n^* = (k+3)(p^* + 2) + p-1$, a number that does not exceed $n$ by assumption. We emphasize that by construction $X$ does \textit{not} depend on $\delta$. Furthermore, if we delete from $e_+$ and $e_-$ those coordinates that correspond to the zero coordinates that have been inserted to obtain $y(\delta)$ from $v(\delta)$, we obtain the vectors $w_1$ and $w_2$. Therefore, it follows from Equation \eqref{orthoy} that $y(\delta)$ is orthogonal to $\lspan(\bar{X}) = \lspan((X, e_+, e_-))$, that $\rank(X, e_+, e_-) = \rank(\bar{W}) = k+2$ and that for every $\delta \neq 0$ we have
\begin{equation}
\hat{u}_{\bar{X}}(y(\delta)) = y(\delta).
\end{equation}
As an immediate consequence we obtain
\begin{align}\label{VconstrA}
\hat{V}_{\bar{X}}(y(\delta)) & = \bar{X}' \diag(y(\delta)) \\
&= (v_1(\delta) \bar{W}'_{1\cdot}, 0_{k+2,p^*+1}, v_2(\delta) \bar{W}'_{2\cdot}, 0_{k+2,p^*}, v_3(\delta) \bar{W}'_{3\cdot}, 0_{k+2,p^*+1}, \hdots  
\\ &  \hspace{5cm} \hdots, v_{k+4}(\delta) \bar{W}'_{{k+4}\cdot} ,0_{k+2,p-1}, 0_{k+2,n-n^*}),
\end{align}
where we recall that all coordinates of $v(\delta)$ are nonzero and therefore $\hat{V}_{\bar{X}}(y(\delta))$ has precisely $k+4$ nonzero columns. We now intend to apply Lemma \ref{AUXCONSTR} with $t = k+3$, acting as if $\bar{X} \in \Rel^{n \times (k+2)}$ was the underlying design, $(\bar{R}, r)$ was the hypothesis to be tested and with the triple $\kappa$, $\bar{M}$, $p$ which obviously satisfies Assumption \ref{weightsPWRB} with respect to $\bar{X}$ (a matrix with $k+2$ columns), since by assumption we have $1 \leq p \leq \frac{n}{k+3}$  (Note that due to interpreting $\bar{X}$ as the underlying design, $k+2$ corresponds to the `k' in Lemma \ref{AUXCONSTR}). We note first that removing the first or last row of $\bar{W}$ does not reduce its rank, because $v(\delta)$ (a vector all coordinates of which are nonzero) is orthogonal to every column of this matrix (cf. the argument in the beginning of the proof of Lemma \ref{AUXCONSTR}). Using $p^* \geq p$ we hence see that Assumptions (A1)-(A3) in Lemma \ref{AUXCONSTR} are satisfied by construction. Now consider the case $M \in \mathbb{M}_{KV}$. By definition $\bar{M}$ is an element of $\mathbb{M}_{KV}$ (acting as if $\bar{X}$ was the underlying design matrix). Therefore, Condition (CKV) is satisfied for $\delta \neq 0$ arbitrary, and $g^*_{\kappa, \bar{M}, p}(y(\delta), \bar{X}, \bar{R}) \neq 0$ follows. Consider the case where $M \in \mathbb{M}_{AM}$. Since $n - n^* \geq \mathbf{1}_{\mathbb{M}_{AM}}(M) = 1$ and because of $j_{k+4} = n^*-p+1$ it follows that $n-j_{k+4} > p-1$. Therefore, Condition (CAM) in Lemma \ref{AUXCONSTR} is satisfied and therefore $g^*_{\kappa, \bar{M}, p}(y(\delta), \bar{X}, \bar{R}) \neq 0$ for $\delta \neq 0$ arbitrary. It remains to consider the case where $M \in \mathbb{M}_{NW}$. It suffices to find a $\delta^* \neq 0$ such that $\bar{M}(y(\delta^*))$ is well defined (see the proof of Part 2 of Lemma \ref{AUXCONSTR}). The latter statement is equivalent to the denominator in the fraction appearing in the definition of $\bar{M}(y(\delta^*))$ being nonzero, i.e., 
\begin{equation}
\sum_{i = -(n-p-1)}^{n-p-1} w(i) \bar{\sigma}_i(y(\delta^*)) \neq 0,
\end{equation}
where 
\begin{equation}
\bar{\sigma}_i(y(\delta^*)) = (n-p)^{-1} \sum_{j= |i|+1}^{n-p} \bar{\omega}' [\hat{Z}_{\bar{X}}(y(\delta^*))]_{\cdot j} [\hat{Z}_{\bar{X}}(y(\delta^*))]_{\cdot (j-|i|)}' \bar{\omega} ~~~ \text{ for } |i| = 0, \hdots, n-p-1.
\end{equation}
By definition $\bar{\omega} = (\omega', 0, 0)'$ and we recall that $\hat{Z}_{\bar{X}} = \hat{V}_{p, \bar{X}}(y(\delta))$ which implies via Equation \eqref{VconstrA} that $(\omega', 0, 0) \hat{Z}_{\bar{X}}(y(\delta))$ equals
\begin{align}
(0_{k,p^*-p+2}, v_2(\delta) \omega' W'_{2\cdot}, 0_{k,p^*}, v_3(\delta) \omega'W'_{3\cdot}, 0_{k,p^*+1}, \hdots,  v_{k+4}(\delta) \omega' W'_{{k+4}\cdot} ,0_{k,p-1}, 0_{k,n-n^*})
\end{align}
The only coordinate of this vector that depends on $\delta$ is $v_2(\delta) \omega' W'_{2\cdot} = \delta \sum_{i = 1}^k \omega_i$, the latter equation following from $v_2(\delta) = \delta$ and $W_{2\cdot} = (1, \hdots, 1)$. Since $\sum_{j = 1}^k \omega_i > 0$, the denominator appearing in the definition of $M(y(\delta))$ interpreted as a function of $\delta$ is now seen to be a polynomial of degree $2$ in $\delta$. Hence, there must exist a $\delta^* \neq 0$ such that the denominator does not vanish. It follows that $g_{\kappa, \bar{M}, p}^*(y(\delta^*), \bar{X}, \bar{R}) \neq 0$. 

\medskip

Concerning the second statement we observe that $(k+2)(p^* + 2) + p-1 \leq n$ implies $k+1 <n$, and therefore we have $\rank(\tilde{X}, e_+, e_-) = k+1$ for $\lambda_{\Rel^{n \times (k-1)}}$-almost every $\tilde{X} \in \tilde{\mathfrak{X}}_0$. By assumption the first column of $R$ is zero. Therefore, for $\lambda_{\Rel^{n \times (k-1)}}$-almost every $X = (e_+, \tilde{X}) \in \Rel^{n \times k}$ we have $e_+ \in \mathfrak{M}_X$, $R\hat{\beta}_X(e_+) = 0$ and $e_- \notin \mathfrak{M}_X$, i.e., for $\lambda_{\Rel^{n \times (k-1)}}$-almost every $(e_+, \tilde{X}) \in \Rel^{n \times k}$ Scenario (1) in Theorem \ref{TU_3} applies. As above, it suffices to construct a pair $y \in \Rel^{n}$ and $\tilde{X} \in \Rel^{n \times (k-1)}$ (recall that $k \geq 2$), such that $g^*_{\kappa, \bar{M}, p}(y , \bar{X}, \bar{R}) \neq 0$, where $\bar{X} = (e_+, \tilde{X}, e_-) \in \Rel^{n \times (k+1)}$ and $\bar{R} = (R, 0)$. Here, the matrix $\tilde{X}$ is $n \times (k-1)$ dimensional. By assumption $k^{\sharp} = k-1$ obviously satisfies $(k^{\sharp}+3)(p^* + 2) + p-1 + 1_{\mathbb{M}_{AM}}(M) \leq n$. To construct the matrix $\tilde{X}$ we can thus use the same argument as was used to construct $X$ in the proof of the first statement ($k^{\sharp}$ replacing $k$). The matrix $\bar{X}$ so obtained has (after a permutation of its columns) the same structure as has the matrix $\bar{X}$ constructed in the proof of the first statement. We can therefore use almost the same arguments to conclude that $g^*_{\kappa, \bar{M}, p}(y(\delta^*) , \bar{X}, \bar{R}) \neq 0$ for some $\delta^* \neq 0$ and $y(\delta)$ as constructed in the first part of the proof.
\end{proof}

\small{
\bibliographystyle{ims}	
\bibliography{refs}		
}

\end{document}